\documentclass[sigconf]{acmart}
\AtBeginDocument{%
  \providecommand\BibTeX{{%
    \normalfont B\kern-0.5.em{\scshape i\kern-0.25em b}\kern-0.8em\TeX}}}

\copyrightyear{2023}
\acmYear{2023}

\acmConference[KDD '23]{}{Long Beach, CA, USA}

\usepackage{algorithm}
\usepackage{algorithmic}
\usepackage{enumitem}
\usepackage{hhline}
\usepackage{microtype}
\usepackage{graphicx}
\usepackage{subfigure}
\usepackage{amsmath}\allowdisplaybreaks
\usepackage{amsfonts,bm}
\def\floor#1{\left\lfloor #1 \right\rfloor}
\def\1{\bf{1}}

\newcommand{\Norm}[1]{\left\| #1 \right\|}

\def\inner#1#2{\langle #1, #2 \rangle}

\def\va{{\bf{a}}}

\def\ve{{\bf{e}}}

\def\vg{{\bf{g}}}

\def\vv{{\bf{v}}}

\def\vx{{\bf{x}}}
\def\vy{{\bf{y}}}

\def\fO{{\mathcal{O}}}

\def\BR{{\mathbb{R}}}

\def\mA {{\bf A}}

\def\mH {{\bf H}}
\def\mI {{\bf I}}

\DeclareMathOperator*{\argmin}{arg\,min}

\usepackage{amsthm}
\usepackage{etoolbox}
\theoremstyle{plain}

\makeatletter
\def\Ddots{\mathinner{\mkern1mu\raise\p@
\vbox{\kern7\p@\hbox{.}}\mkern2mu
\raise4\p@\hbox{.}\mkern2mu\raise7\p@\hbox{.}\mkern1mu}}
\makeatother

\makeatletter
\newcommand*{\rom}[1]{\expandafter\@slowromancap\romannumeral #1@}
\makeatother

\newtheorem{theorem}{Theorem}[section]

\newtheorem{lemma}[theorem]{Lemma}

\theoremstyle{definition}
\newtheorem{definition}[theorem]{Definition}
\newtheorem{assumption}[theorem]{Assumption}
\theoremstyle{remark}
\newtheorem{remark}[theorem]{Remark}

\def\vx {{{\bf x}}}

\def\vy {{{\bf y}}}
\def\va {{{\bf a}}}
\def\BR{{\mathbb{R}}}

\def\e{{\bf e}}

\def\g{{\bf g}}

\def\H{{\bf H}}
\def\I{{\bf I}}

\def\v{{\bf v}}

\def\x{{\bf x}}

\def\y{{\bf y}}

\def\z{{\bf z}}
\def\0{{\bf 0}}
\def\1{{\bf 1}}

\def\OM{{\mathcal O}}

\def\RB{{\mathbb R}}

\def\server{\underline{server}:}
\def\client{\underline{$i$-th client}:}

\usepackage{algorithm}
\usepackage{algorithmic}
\usepackage{enumitem}
\usepackage{hhline}

\def \Floor #1{{\left \lfloor #1  \right \rfloor }}

\usepackage{multirow, hhline}
\usepackage[para,online,flushleft]{threeparttable}

\usepackage{threeparttable}
\usepackage{makecell}
\usepackage{multirow}

\usepackage{colortbl}
\definecolor{bgcolor}{rgb}{0.93,0.99,1}
\definecolor{bgcolor2}{rgb}{0.8,1,0.8}
\definecolor{bgcolor3}{rgb}{0.50,0.90,0.50}
\usepackage{tcolorbox}
\usepackage{pifont}
\definecolor{mydarkgreen}{RGB}{39,130,67}
\definecolor{mydarkred}{RGB}{192,25,25}

\newcommand{\cmark}{\ding{51}}%
\newcommand{\xmark}{\ding{55}}%

\usepackage{xspace}
\usepackage{bm}

\usepackage{relsize} %SEB: feel free to revert, I think the font looked a bit too big otherwise

\usepackage{relsize} %SEB: feel free to revert, I think the font looked a bit too big 

\usepackage[capitalize,noabbrev]{cleveref}

\usepackage[textsize=tiny]{todonotes}

\begin{document}

\title[C2EDEN]{Communication Efficient Distributed Newton Method with Fast Convergence Rates}

\author{Chengchang Liu}
\affiliation{%
 \department{Department of Computer Science $\&$ Engineering}
  \institution{The Chinese University of Hong Kong}
%   \city{Hong Kong}
   \country{}
}
\email{7liuchengchang@gmail.com}

\author{Lesi Chen}
\affiliation{%
  \department{School of Data Science}
  \institution{Fudan University}
%   \city{Hefei}
   \country{}
}
\email{lschen19@fudan.edu.cn}

\author{Luo Luo}
\authornote{The corresponding author}
\affiliation{%
 \department{School of Data Science}
  \institution{Fudan University}
%   \city{Shanghai}
   \country{}
}
\email{luoluo@fudan.edu.cn}

\author{John C.S. Lui}
\affiliation{%
 \department{Department of Computer Science $\&$ Engineering}
  \institution{The Chinese University of Hong Kong}
%   \city{Hong Kong}
   \country{}
}
\email{cslui@cse.cuhk.edu.hk}

\renewcommand{\shortauthors}{Liu et al.}
\begin{abstract}
We propose a communication and computation efficient second-order method for distributed optimization.
For each iteration, our method only requires $\mathcal{O}(d)$ communication complexity, where $d$ is the problem dimension. 
We also provide theoretical analysis to show the proposed method has the similar convergence rate as the classical second-order optimization algorithms.
Concretely, our method can find~$\big(\epsilon, \sqrt{dL\epsilon}\,\big)$-second-order stationary points for nonconvex problem by $\mathcal{O}\big(\sqrt{dL}\,\epsilon^{-3/2}\big)$ iterations, where $L$ is the Lipschitz constant of Hessian. Moreover, it enjoys a local superlinear convergence under the strongly-convex assumption. 
Experiments on both convex and nonconvex problems show that our proposed method performs significantly better than baselines.
\end{abstract}

%    <concept>
%        <concept_id>10002950.10003714.10003716</concept_id>
%        <concept_desc>Mathematics of computing~Mathematical optimization</concept_desc>
%        <concept_significance>500</concept_significance>
%        </concept>
%  </ccs2012>

% \ccsdesc[500]{Mathematics of computing~Mathematical optimization}
% %%
% %% Keywords. The author(s) should pick words that accurately describe
% %% the work being presented. Separate the keywords with commas.

\maketitle
\section{Introduction}

Distributed optimization has received lots of attention in machine learning and data mining communities~\cite{molzahn2017survey,yang2019survey,boyd2011distributed,boyd2011convex,liu2017distributed,wangni2018gradient,smith2015l1,lee2018distributed,verbraeken2020survey}. 
Its major attractive feature is that it allows solving large-scale problems in a parallel fashion, which significantly improves the efficiency for training the models~\cite{li2013parameter,li2014communication,yan2014coupled}. %On the other hand, it avoid the data sharing among the local devices during the training process and thus can keep data privacy~\cite{konevcny2016federated1,mcmahan2017communication}. 
This paper considers the distributed optimization problem of the form
\begin{align}\label{eq:obj}
    \min_{\x\in\RB^d} f(\x) \triangleq \frac{1}{n}\sum_{i=1}^n f_i(\x),
\end{align}
where $d$ is the problem dimension, $n$ is the number of clients and~$f_i$ is the smooth local function associated with the $i$-th client.
We focus on the scenario that each client can access its local function and communicate with a central server.

The first-order methods are popular in such distributed setting \cite{nesterov2018lectures,aji2017sparse,li2019convergence,mitra2021linear,karimireddy2020scaffold,mishchenko2022proxskip,konevcny2016federated}. 
For each iteration of these methods, the clients compute the gradients of their local functions in parallel and the server updates the variable after receiving the local information from clients, which results the computation on each machine be efficient.
The main drawback of first-order methods is that they require a large amount of iterations to find an accurate solution.
As a consequence, factors like bandwidth limitation and latency in the network leads to expensive communication cost in total.

The classical second-order methods have superior convergence rates~\cite{nesterov2018lectures,nocedal1999numerical} than the first-order methods. 
However, their extension to a distributed setting is non-trivial.
The key challenge is how to deliver the second-order information efficiently.
One straightforward strategy is to perform classical Newton-type step on the server within the aggregate Hessian, but this requires~$\fO(d^2)$ communication complexity at each iteration, which is unacceptable in distributed settings.

Several previous works address the communication issues in distributed second-order optimization.
They avoid $\OM(d^2)$ communication cost per iteration but have several other shortcomings. 
We briefly review three main approaches as follows.
%In previous work, there are three main types of approaches to address the communication issues in distributed second-order optimization. 
%Although such methods require only $\OM(d)$ communication cost per iteration, they may arise several other problems. We summarize these approaches below.
\begin{itemize}[leftmargin=0.5cm]
\item %\textbf{Doing Local Update:} 
For the first one, each client performs Newton step with its local Hessian and sends the descent direction to the server, then the server aggregates all of local descent directions to update the global variable~\cite{wang2018giant,shamir2014communication,yuan2020convergence,ghosh2021escaping,crane2019dingo,reddi2016aide}.
Their convergence rates depend on the assumption of specific structure of the problems or data similarity.
In general case, their theoretical results are no stronger than first-order methods.
\item %\textbf{Solving Sub-Routine:} 
For the second one, the algorithms formulate a Newton step by a (regularized) quadratic sub-problem and solve it with distributed first-order methods~\cite{reddi2016aide,zhang2015disco,ye2022accelerated,uribe2020distributed}. 
The procedure of solving sub-problem introduces additional computation and communication complexity, which increases the total cost and leads to complicated implementation. 
\item %\textbf{Constructing Local Hessian Estimator:} 
For the third one, each client keeps the local second-order information whose change at each iteration can be stored in~$\OM(d)$ space complexity, which leads to the efficient communication~\cite{soori2020dave,islamov2022distributed,safaryan2021fednl}.
However, the storage of local second-order information on each client requires~$\fO(d^2)$ space complexity.
Additionally, the convergence guarantees of these methods require strong assumptions such as Lipschitz continuity of each local Hessian.
\end{itemize}

Furthermore, most existing distributed second-order methods are designed for convex optimization problems~\cite{wang2018giant,shamir2014communication,yuan2020convergence,crane2019dingo,zhang2015disco,ye2022accelerated,uribe2020distributed,soori2020dave,islamov2022distributed,safaryan2021fednl}.
Few works consider the theory for nonconvex case~\cite{reddi2016aide,ghosh2021escaping}.
\citet{reddi2016aide} provided non-asymptotic convergence rate to find approximate first-order stationary point of nonconvex objective function by distributed second-order method, but the convergence result has no advantage over first-order methods.
\citet{ghosh2021escaping} proposed a distributed cubic-regularized Newton method to find the approximate second-order stationary point, while its convergence guarantee requires very strong assumption that all of the local Hessians during iterations are sufficiently close to the global Hessian.

\begin{table*}[!t]
	\centering
	\caption{We summarize the distributed optimization algorithms by their communication complexity per iteration, space complexity on each client, global convergence in the general nonconvex (NC) case, local convergence in the strongly-convex (SC) case and additional assumption for convergence guarantee.}\label{tbl:main-2}
	\begin{threeparttable}
\footnotesize\setlength\tabcolsep{5.pt}
\begin{tabular}{ c  c  c  c  c c }
			\toprule[.1em]
			 \begin{tabular}{c}\bf Method \end{tabular} &  \begin{tabular}{c}\bf Communication  \\\bf  Complexity
        			 \end{tabular} &   \begin{tabular}{c}\bf Space \\\bf Complexity \end{tabular}  &  \begin{tabular}{c}\bf Global Convergence \\ \bf  (NC)  \end{tabular}&\begin{tabular}{c}\bf  Local Superlinear \\ \bf  Convergence \\ \bf  (SC)  \end{tabular}  & \begin{tabular}{c}
			    \bf  No Additional Assumption    \\
			    \bf  for Convergence
			 \end{tabular}\\
			\midrule
         \begin{tabular}{c}
         First-Order Methods \\ {\cite{nesterov2018lectures,mishchenko2022proxskip}}
         \end{tabular} 
    & $\OM(d)$ & $\OM(d)$&  $\OM(\epsilon^{-2})$& \xmark &\cmark  \\  
			\cmidrule{1-1}
	\begin{tabular}{c}	{Distributed-(Cubic)-Newton}\\{ \cite{nesterov2006cubic,nesterov2008accelerating} }\end{tabular}
	&$\OM(d^2)$ &$\OM(d^2)$ & $\OM(\epsilon^{-3/2})$&\cmark & \cmark \\	
	  	\cmidrule{1-1}
 \begin{tabular}{c}
AIDE~\tnote{(a)} \\
 {\cite{reddi2016aide}}
 \end{tabular} &$\OM(d)$&$\OM(d)$ &\xmark&\xmark&\xmark~\tnote{(f), (g)}\\[0.2cm]
 \begin{tabular}{c}
 Disco/Accelerated ADAN~\tnote{(a)} \\
{ \cite{zhang2015disco,ye2022accelerated}} \end{tabular} &$\OM(d)$&$\OM(d)$&\xmark&\xmark&\xmark~\tnote{(b)}\\
		\cmidrule{1-1}
	\begin{tabular}{c}
 (Inexact)-DANE/DANE-HB \\
{\cite{shamir2014communication,reddi2016aide,yuan2020convergence}}
	\end{tabular}
 & $\OM(d)$ & $\OM(d)$ & $\OM(\epsilon^{-2})$\tnote{ (d)}&\xmark &\xmark~\tnote{(b), (f), (g)}  \\[0.2cm]
  \begin{tabular}{c}
GIANT \\
 {\cite{wang2018giant}}
 \end{tabular} &$\OM(d)$&$\OM(d)$&\xmark& \xmark &\xmark~\tnote{(b)}\\[0.2cm]
 \begin{tabular}{c}
Local CRN \\
 {\cite{ghosh2021escaping}}
 \end{tabular} &$\OM(d)$&$\OM(d)$& $\OM(\epsilon^{-3/2})$ \tnote{(e)}&\xmark&\xmark~\tnote{(b)}\\
 	\cmidrule{1-1}
 \begin{tabular}{c}
 Compressed Distributed Newton \\
 {\cite{islamov2021distributed,islamov2022distributed,safaryan2021fednl}}
 \end{tabular} &~~$\OM(d)$&$\OM(d^2)$&\xmark&\cmark~\tnote{(c)} &\xmark\tnote{(h)}\\[0.2cm]
 \begin{tabular}{c}
Distributed-Quasi-Newton \\
{ \cite{soori2020dave}}
 \end{tabular} &$\OM(d)$&$\OM(d^2)$&\xmark&\cmark&\xmark\tnote{(f), (g), (h)}\\
 \midrule 
	 \begin{tabular}{c}{C2EDEN} \\ {Algorithm~\ref{alg:C2EDEN} }\end{tabular} & ${\OM(d)}$ & ${\OM(d)}$ &${\OM(\epsilon^{-3/2})}$&\cmark&\cmark\\			
			\bottomrule		
	\end{tabular} % default value: 6pt
	\begin{tablenotes}
		{\scriptsize     
  			\item [{ (a)}] The method(s) have to solve a sub-problem at each iteration, which requires additional communication complexity. \\
		\item[{ {(b)}}] The method(s) require data similarity to guarantee their convergence.\\
			\item [{ (c)}] The local superlinear convergence requires all of the local client Hessian estimator be close to $\nabla^2 f_i(\x^*)$ at each iteration.\\
           \item [ (d)] The global convergence rate is only established for finding approximate first-order stationary point. \\
   			\item [{ (e)}] The global convergence rate requires very strong assumption such that $\|\nabla f_i(\x)-\nabla f(\x)\|\leq \epsilon$ and $\|\nabla^2 f_i(\x)-\nabla^2 f(\x)\|\leq \sqrt{\epsilon}$.\\
      \item [{ (f)}] The convergence guarantee requires each $f_i(\cdot)$ is strongly-convex.\\
       \item [{ (g)}] The convergence guarantee requires each $\nabla f_i(\cdot)$ is Lipschitz continuous.\\
         \item [{ (h)}] The convergence guarantee requires each $\nabla^2 f_i(\cdot)$ is Lipschitz continuous.\\
			}
	\end{tablenotes}  	
	\end{threeparttable}
\end{table*}

In this paper, we provide a novel mechanism to reduce the communication cost for distributed second-order optimization.
It only communicates by gradients and Hessian-vector products. 
Concretely, the server forms a cubic-regularized Newton step~\cite{nesterov2006cubic,doikov2022second} based on the current aggregated gradient and the delayed Hessian which is formed by Hessian-vector products received in recent iterations.
As a result, we propose a Communication and Computation Efficient DistributEd Newton (C2EDEN) method, which has the following attractive properties.
\begin{itemize}[topsep=0pt]\setlength\itemsep{-0.1em}
    \item It only requires $\OM(d)$ communication complexity at each iteration, matching the cost of first-order methods.
    \item It has simple local update rule, which does not require any subroutine with additional communication cost. Furthermore, the client avoids storing the local Hessian with $\fO(d^2)$ storage complexity.
    \item It can find an $\big(\epsilon, \sqrt{dL\epsilon}\,\big)$-second-order 
    stationary point within $\fO\big(\sqrt{dL}\,\epsilon^{-3/2}\big)$ iterations for nonconvex problem, and exhibit local superlinear convergence for strongly-convex problem.
    \item Its convergence guarantees do not require additional assumptions such as Lipschitz continuouity of the local Hessian or strong convexity of the local function.
\end{itemize}
We summarize the main theoretical results of C2EDEN and related work in Table \ref{tbl:main-2}.
We also provide numerical experiments on both convex and nonconvex problems to show the empirical superiority for the proposed method. 

\paragraph{Paper Organization}

The remainder of the paper is organized as follows. In Section~\ref{sec:pre} presents preliminaries. 
Section~\ref{sec:C2EDEN} provide the details of C2EDEN method and provides the convergence analysis. 
Section~\ref{sec:exp} validates our algorithm
empirically to show its superiority. 
Finally, we conclude our work in Section~\ref{sec:conclu}.

%% 
%%%
%%
%%
%%
\section{Preliminaries}
\label{sec:pre}

We use $\Norm{\,\cdot\,}$ to present spectral norm and Euclidean norm of matrix and vector respectively. 
%We denote the local function on the $i$-th client as $f_i(\,\cdot\,)$ and the corresponding local gradient and local Hessian are $\nabla f_i(\cdot)$ and $\nabla^2 f_i(\cdot)$ respectively. 
We denote the standard basis for~$\BR^d$ by~$\{\ve_1,\dots,\ve_d\}$ and let~$\I$~be the identity matrix. 
Given matrix~$\mH\in\BR^{d\times d}$, we write its $j$-th column as~$\mH(j)$, which is equal to~$\mH\ve_j$. 
We also denote the smallest eigenvalue of symmetric matrix by~$\lambda_{\min}(\cdot)$. 

% We also introduce the following notation for index
% \begin{align}
%     \pi(k;m):= k-(k~{\text{mod}}~m)~~~\text{and}~~~\iota(k;m):= \pi(k;m)-m.
% \end{align}

Throughout this paper, we make the following assumption.
\begin{assumption}\label{ass:heslip}
We assume each local function $f_i(\,\cdot\,)$ is twice differentiable and the objective $f(\,\cdot\,)$ has a Lipschitz continuous Hessian, i.e., there exists constant $L>0$ such that 
\begin{align}
    \|\nabla^2 f(\x) -\nabla^2 f(\y)\| \leq L\|\x-\y\|
\end{align}
for any $\vx,\vy\in\BR^d$.
\end{assumption}

Compared with assuming each local Hessian $\nabla^2 f_i(\cdot)$ is $L$-Lipschitz continuous~\cite{shamir2014communication,reddi2016aide,zhang2015disco,yuan2020convergence,islamov2022distributed,islamov2021distributed,safaryan2021fednl, soori2020dave}, the Lipschitz continuity on the global Hessian $\nabla f^2(\,\cdot\,)$ in Assumption \ref{ass:heslip} is much weaker.

For the nonconvex case, we assume the global objective is lower bounded.
\begin{assumption}
\label{ass:boundedbelow}
We assume the global objective $f(\,\cdot\,)$ is lower bounded, i.e., we have $f^* \triangleq \inf_{\x\in\RB^{d}} f(\x) > -\infty.$
% \begin{align}
% f^* \triangleq \inf_{\x\in\RB^{d}} f(\x) > -\infty.
% \end{align}
\end{assumption}

We also consider the specific case in which the global objective is strongly-convex.
\begin{assumption}
\label{ass:mustrongly}
    We assume the global objective $f(\,\cdot\,)$ is strongly-convex, i.e., there exists a constant $\mu>0$ such that
    % \begin{align}
    %     f(\y)\geq f(\x)+\inner{\nabla f(\x)}{\y-\x} +\frac{\mu}{2}\|\y-\x\|^2,
    % \end{align}
    % for any $\x,\y\in\RB^d$. 
    % If $f(\,\cdot\,)$ is twice-differentiable, it is equal to
    $\nabla^2 f(\x)\succeq \mu\I$ for any $\vx\in\BR^d$.
\end{assumption}
Note that all existing superlinear convergent distributed second-order methods require the strongly-convex assumption on each local function $f_i(\,\cdot\,)$~\cite{soori2020dave,islamov2021distributed,safaryan2021fednl,islamov2022distributed}, while our  Assumption \ref{ass:mustrongly} only requires the strong convexity of the global objective.

This paper studies second-order optimization on both convex and nonconvex problems.
For the convex case, we target to find approximate first-order stationary point.
\begin{definition}
We call $\vx$ is an $\epsilon$-first-order stationary point of~$f(\cdot)$ if it satisfies $\|\nabla f(\vx)\|\leq\epsilon$.
\end{definition}

For the nonconvex case, finding global solution is intractable and the first-order stationary point may lead to the undesired saddle point. 
Hence, we target to find the approximate second-order stationary point to characterize the local optimality. 
\begin{definition}
We call $\vx$ is an $(\epsilon, \delta)$-second-order stationary point of $f(\cdot)$ if it satisfies
$\|\nabla f(\vx)\|\leq\epsilon$ and $\nabla^2 f(\vx)\succeq-\delta\mI$.
% \begin{align*}
%     \|\nabla f(\vx)\|\leq\epsilon\quad \text{and}\quad \nabla^2 f(\vx)\succeq-\delta\mI.
% \end{align*}
\end{definition}

\section{Algorithm and Main Results}
\label{sec:C2EDEN}
%直接传Hessian（带宽），分成d次：传完才能算
%我们能一边传，一边算，每次都能利用二阶信息
%Local Hessian每次都不用算 

% It is not possible to communicate a full Hessian at some iteration because on the one hand, communicating a Hessian matrix will significantly influence the width of the communication pipe, on the other hand, if we divide the full Hessian in $d$ part and communicate $d$-times before the Server aggregate them, then, the server will not be able to do any update but only wait for the the communication of the Hessian matrix. 

We first introduce our Communication and Computation Efficient DistributEd Newton (C2EDEN) method and the main ideas, then we provide the convergence guarantees.

\subsection{The C2EDEN Method}
We present the details of C2EDEN in Algorithm~\ref{alg:C2EDEN}, where the notation $k\%d$ presents the remainder of $k$ divided by $d$, $M>0$ is the cubic-regularized parameter and we use 
\begin{align}
\label{eq:cubic}
\begin{split}
&{\rm T}_M(\g,\mA;\x)  \\
&\quad~~= \argmin_{\y\in\RB^d} \left\{\inner{\g}{\y-\x}+\frac{1}{2}\inner{\mA(\y-\x)}{\y-\x}+\frac{M}{6}\|\y-\x\|^3\right\}
\end{split}
\end{align}
to present the solution of the cubic-regularized sub-problem for a given vector $\g\in\BR^d$ and a symmetric matrix $\mA\in\BR^{d\times d}$.
% For strongly-convex case, we directly set $M=0$.

We partition the Hessian at at snapshot point~$\tilde\vx$ into $d$ columns
\begin{align*}
\nabla^2 f_i(\tilde{\x})=
\begin{bmatrix}
\Big| & \Big| & & \Big| \\
\nabla^2f_i(\tilde{\x})\e_1 &  \nabla^2f_i(\tilde{\x})\e_2 & \cdots&\nabla^2 f_i(\tilde{\x})\e_d \\
\Big| & \Big| & & \Big|
\end{bmatrix}.
\end{align*}
C2EDEN atomize the communication of local $\nabla^2 f_i(\tilde{\x})$ into~$d$ consecutive iterations by sending ~$\v_{i,k}=\nabla^2 f_i(\tilde{\x})\e_{k\%d+1}$ for $i=1,\dots,n$ at the $k$-th iteration.
Note that the cubic-regularized Newton steps in our method only use the Hessian at snapshot point~$\tilde\vx$, which is updated per $d$ iterations.
Also note that any iteration of C2EDEN avoids the construction of a full local Hessian on the client, which significantly reduces computational cost of accessing the second-order information.
Additionally, the iteration of C2EDEN only communicates the local gradient $\vg_{i,k}$ and Hessian-vector product $\vv_{i,k}$ together, which means each client only requires~$\fO(d)$ communication complexity and $\fO(d)$ space complexity to deal with its local second-order information.

The computational cost of C2EDEN is dominated by computing  Hessian-vector products on clients and solving the cubic-regularized sub-problem on server.
We emphasis that these two steps can be executed in parallel.
Since the computation of $\vx_{k+1}$ (line 22) on the server only depends on previous Hessian at $\tilde\vx$, 
all of the clients are allowed to compute $\vv_{i,k}$ (line 16) at the same time.
Additionally, the reuse of Hessian at snapshot~$\tilde\vx$ leads to a computational complexity for achieving $\vx_{k+1}$ in~$\fO(d^2)$ on average~\cite{doikov2022second}.

We present the pipeline for one epoch of C2EDEN in Figure \ref{figure:pipeline} to illustrate our mechanism.

\begin{algorithm}[t] % follow the format of \citet{ghosh2021escaping}
\caption{C2EDEN$\,(\vx_0,M, K)$}\label{alg:C2EDEN}
\begin{algorithmic}[1]\vskip0.05cm
%\STATE $\x_d=\x_{d-1}\cdots\x_0$\\[0.14cm]
\STATE \textbf{for $k=0,1\cdots d-1$ do} \\[0.14cm]
\STATE \quad \textbf{for} $i=1,\dots,n$ \textbf{do in parallel}  \\[0.14cm]
\STATE \quad\quad  \client \\[0.14cm]
\STATE \quad \quad \quad compute $\v_{i,k}=\nabla^2 f_i(\x_0)\e_{k+1}$ \\[0.14cm] 
\STATE \quad \textbf{end if}\\[0.14cm]
\STATE \quad \server \\[0.14cm]
\STATE \quad \quad aggregate ${\H}^{+}(k+1)=\frac{1}{n}\sum_{i=1}^{n}\v_{i,k}$ \\[0.14cm]
\STATE \textbf{end for}\\[0.14cm]
\STATE $\x_d=\x_0$\\[0.14cm]
\STATE \textbf{for} $k=d,d+1,\dots,K$ \textbf{do} \\[0.14cm]
\STATE \quad \textbf{if} $k\%d = 0$ \textbf{then}\\[0.14cm]
\STATE \quad \quad \server \\[0.14cm] 
\STATE \quad \quad \quad update $\tilde{\x}=\x_{k}$ and $\H={\H}^{+}$ \\[0.14cm]
\STATE \quad \textbf{end if}\\[0.14cm]
\STATE \quad \textbf{for} $i=1,\dots,n$ \textbf{do in parallel}  \\[0.14cm]
\STATE \quad\quad  \client \\[0.14cm]
\STATE \quad\quad\quad compute $\v_{i,k}=\nabla^2 f_i(\tilde{\x})\e_{k\%d+1}$ \\[0.14cm]  
\STATE \quad\quad\quad  compute $\g_{i,k}=\nabla f_i(\x_k)$\\[0.14cm]
\STATE \quad \quad \quad send $\g_{i,k}$ and $\v_{i,k}$ \\[0.14cm]
\STATE \quad \textbf{end for}\\[0.14cm]
\STATE \quad \server \\[0.14cm]
\STATE \quad \quad aggregate $\g_k=\frac{1}{n}\sum_{i=1}^{n} \g_{i,k}$\\[0.14cm]
\STATE \quad \quad compute and broadcast $\x_{k+1}={\rm T}_M(\g_k,\H;\x_k)$ \\[0.14cm]
\STATE \quad \quad aggregate ${\H}^{+}(k\%d+1)=\frac{1}{n}\sum_{i=1}^{n}\v_{i,k}$ \\[0.14cm]
\STATE \textbf{end for}\\[0.05cm]
\end{algorithmic}
\end{algorithm}

\begin{figure*}[ht]
	\centering
	\renewcommand{\arraystretch}{2}
	%\setlength{\tabcolsep}{3pt}
	%\arrayrulecolor[HTML]{66BB6A}	
 
{\small\begin{tabular}{l|c|c|c|c|c|}
\hhline{------}
		\multicolumn{1}{|c|}{{\bf Iteration}} &
\multicolumn{2}{c|}{$k=td$} & 
        $\cdots\cdots$ & 
        \multicolumn{2}{c|}{$k=td+(d-1)$}\\
\hhline{------}
		\multicolumn{1}{|c|}{{\bf $i$-th Client}} &
		$\vg_{i,td}=\nabla f_i(\x_{td})$ &
		$\vv_{i,td}=\nabla^2 f_i(\tilde{\x})\e_1$ & 
        %$\vg_{i,td+1}=\nabla f_i(\x_{td+1})$ & 
        %$\vv_{i,td+1}=\nabla^2 f_i(\tilde{\x})\e_2$ & 
        $\cdots\cdots$ & 
        $\vg_{i,td+(d-1)}=\nabla f_i\big(\x_{td+(d-1)}\big)$ & 
        $\vv_{i,td+(d-1)}=\nabla^2f_i(\tilde{\x})\e_{d}$\\
  \hhline{------}
		\multicolumn{1}{|c|}{{\bf Server}} &
		 &
		$\x_{td+1}={\rm T}_M\big(\g_{td},\H;\x_{td}\big)$ & 
        %  & 
        %$\x_{td+2}={\rm T}_M\big(\g_{td+1},\H;\x_{td+1}\big)$ & 
        $\cdots\cdots$ & 
        & 
        $\x_{(t+1)d}={\rm T}_M\big(\g_{td+(d-1)},\H;\x_{td+(d-1)}\big)$\\
  \hhline{------}
\end{tabular}}
\caption{We present the pipeline of C2EDEN from its $td$-th iteration to $(td+(d-1))$-th iteration, where  $\tilde{\x}=\x_{td}$ and~$\mH=\nabla^2 f\big(\vx_{(t-1)d}\big)$. 
The algorithm atomizes the cost of constructing $\nabla^2 f (\tilde{\x})=\nabla^2 f (\x_{td})$ by computing Hessian-vector products~$\nabla^2 f_i(\tilde{\x})\e_1,\dots,\nabla^2 f_i(\tilde{\x})\e_d$ on clients during these $d$ iterations.
The server has already obtained matrix $\mH$ at the $(td-1)$-th iteration and the Hessian~$\nabla^2 f (\tilde{\x})$ is constructed for the next epoch, that is from the $(t+1)d$-th iteration to the $((t+1)d+(d-1))$-th iteration. }\label{figure:pipeline}
\end{figure*}

\subsection{Convergence Analysis}
\label{sec:analysis}
This section provides theoretical guarantees of C2EDEN for both nonconvex case and strongly-convex case. 

% In the remainder of this paper, we denote
% \begin{align}
% \label{eq:cubic}
% \begin{split}
% &{\rm T}_M(\g,\H;\x) = \argmin_{\y\in\RB^d} \Big\{\inner{\g}{\y-\x}\\
% &\quad\quad\quad+\frac{1}{2}\inner{\H(\y-\x)}{\y-\x}+\frac{M}{6}\|\y-\x\|^3\Big\}
% \end{split}
% \end{align}
% to simply the solution of the cubic-regularized sub-problem.
At the $k$-th iteration of Algorithm~\ref{alg:C2EDEN} with $k=td+q$, it performs the cubic-regularized Newton step on the server by using the matrix $\H$ that is exactly the Hessian at the previous snapshot point~$\vx_{t(d-1)}$ and vector gradient $\vg_k$ that is the gradient at the current point~$\vx_k$. 
Hence, the update rule of C2EDEN on the server can be written as
\begin{align}
\label{eq:server-update}
\x_{k+1}= {\rm T}_M\left(\nabla f(\x_k),\nabla^2 f\big(\x_{\tau(k;d)}\big);\x_k\right),
\end{align}
where $\tau(k;d)=d(\Floor{k/d}-1)$.

\subsubsection{The Analysis in Nonconvex Case}

For the general non-convex case, we introduce the  auxiliary quantity~\cite{nesterov2006cubic,doikov2022second}
\begin{align}
\label{eq:gamma}
\begin{split}
    \!\!\gamma(\x) \triangleq \max\left\{-\frac{1}{648M^2}\lambda_{\min}\big(\nabla^2 f(\x)\big)^3, \frac{1}{72\sqrt{2M}}\|\nabla f(\x)\|^{3/2}\right\}.
\end{split}
\end{align}
for our analysis.

We provide the following lemma to show the decrease of function value by a step of solving cubic-regularized sub-problem.
\begin{lemma}[{\cite[Theorem 1]{doikov2022second}}]
    Suppose Assumption~\ref{ass:heslip} holds and we have $M\geq L$, then the update ${\rm T}={\rm T}_M\big(\nabla f(\x),\nabla^2 f(\z);\x\big)$ holds that 
    \begin{align}
    \label{eq:lazy-update}
        f(\x)-f({\rm T})\geq \gamma({\rm T}) +\frac{M}{48}\|{\rm T}-\x\|^3 -\frac{11L^3}{M^2}\|\z-\x\|^3,
    \end{align}
    for any $\x$, $\z\in\RB^d$.
\end{lemma}

Then we present some technical lemmas for later analysis.
\begin{lemma}
For any sequence of positive numbers $\{r_k\}_{k\geq 0}$, it holds for any $m\geq 1$ that
\begin{align}
\label{eq:jenson}
\left(\sum_{i=0}^{m-1} r_i\right)^3\leq m^2\sum_{i=0}^{m-1}r_i^3.
\end{align}
\end{lemma}
\begin{proof}
We directly prove this result by using Jensen's inequality that $  \sum_{i=0}^{m-1}\left(\frac{1}{m}r_i^3\right)^3\leq \frac{1}{m}\sum_{i=0}^{m-1}r_i^3.$
% \begin{align*}
%     \sum_{i=0}^{m-1}\left(\frac{1}{m}r_i^3\right)^3\leq \frac{1}{m}\sum_{i=0}^{m-1}r_i^3.
% \end{align*}
\end{proof}
\begin{lemma}
\label{lm:first-r}
For any sequence of positive numbers $\{r_k\}_{k\geq 0}$, it holds for any $m\geq 1$ that
\begin{align}
\label{eq:induct1}
    \sum_{k=m}^{2m-1}\left(\sum_{i=0}^{k}r_i\right)^3 \leq 3(m+1)^3\sum_{i=0}^{2m-1}r_i^3.
\end{align}
\end{lemma}
\begin{proof}
We prove this statement by induction.

For $m=1$, it follows that $(r_0+r_1)^3 \overset{\eqref{eq:jenson}}{\leq} 4(r_1^3+r_2^3).$
% \begin{align*}
%     (r_0+r_1)^3 \overset{\eqref{eq:jenson}}{\leq} 4(r_1^3+r_2^3).
% \end{align*}
Suppose the inequality \eqref{eq:induct1} holds for $m=1,2\dots,n-1$, then we have
    \begin{align*}
   &     \sum_{k=n}^{2n-1}\left(\sum_{i=0}^{k}r_i\right)^3 =  \sum_{k=n}^{2n-3}\left(\sum_{i=0}^{k}r_i\right)^3 + \left(\sum_{i=0}^{2n-1}r_i\right)^3+\left(\sum_{i=0}^{2n-2}r_i\right)^3\\
        &\leq  \sum_{k={n-1}}^{2n-3}\left(\sum_{i=0}^{k}r_i\right)^3 + \left(\sum_{i=0}^{2n-1}r_i\right)^3+\left(\sum_{i=0}^{2n-2}r_i\right)^3 \\
       & \overset{\eqref{eq:jenson}}{\leq} \sum_{k={n-1}}^{2n-3}\left(\sum_{i=0}^{k}r_i\right)^3 + (2n)^2\sum_{i=0}^{2n-1}r_i^3 + (2n-1)^2\sum_{i=0}^{2n-2}r_i^3 \\
       &\overset{\eqref{eq:induct1}}{\leq}3n^3\sum_{i=0}^{2n-3}r_i^3+8n^2\sum_{i=0}^{2n-1}r_i^3 \leq 3(n+1)^3\sum_{i=0}^{2n-1}r_i^3,
    \end{align*}
which finishes the induction.
\end{proof}
\begin{lemma}
\label{lm:last-r}
For any sequence of positive numbers $\{r_k\}_{k\geq 0}$, it holds that
\begin{align}
\label{eq:last-r}
    \sum_{k=m}^{m+p}\left(\sum_{i=0}^{k-1}r_i\right)^3 \leq 4(m+1)^3 \sum_{i=0}^{m+p}r_i^3
\end{align}

for any $m$ and $p$ such that $1 \leq p \leq m$.
\begin{proof}
We have
\begin{align*}
    & \sum_{k=m}^{m+p}\left(\sum_{i=0}^{k-1}r_i\right)^3\leq \sum_{k=m}^{m+p}\left(\sum_{i=0}^{m+p-1}r_i\right)^3\\
     &\overset{\eqref{eq:jenson}}{\leq} (m+p)^2(p+1)\sum_{i=0}^{m+p-1}r_i^3\leq 4(m+1)^3\sum_{i=0}^{m+p}r_i^3.
\end{align*}

\end{proof}
\end{lemma}

Now, we formally present the global convergence of C2EDEN in the following theorem, which shows that our algorithm can find an $\big(\epsilon, \sqrt{dL\epsilon}\,\big)$-second-order stationary point by~$\OM\big(\sqrt{dL}\,\epsilon^{-3/2}\big)$ iterations.
\begin{theorem}
\label{theorem:global}
    Suppose Assumption~\ref{ass:heslip} and \ref{ass:boundedbelow} hold.  Running C2EDEN (Algorithm~\ref{alg:C2EDEN}) with $M=12dL$ holds that
    \begin{align}
    \label{eq:ncobj}
        \min_{d<i\leq K}\gamma(\x_i) \leq \frac{f(\x_0)-f^*}{K-d},
    \end{align}
    which means by setting the number iterations as $K=\OM\big({\sqrt{dL}}{\epsilon^{-3/2}}\big)$ there exists some $\x_i$ with $d<i\leq K$ such that
    \begin{align*}
        \|\nabla f(\x_i)\|\leq \epsilon \quad \text{and} \quad  \lambda_{\min}\big(\nabla^2 f(\x_i)\big)\geq -\sqrt{dL\epsilon}.
    \end{align*}
    % \begin{align*}
    %     K=\OM\left({\sqrt{dL}}{\epsilon^{-3/2}}\right)
    % \end{align*}
    % iterations.
\end{theorem}
\begin{proof}
We write the total number of iterations as $K=dt+p$, where $t=\floor{K/d}$ and $p=K\%d$. We denote $\x_{d-1}=\x_{d-2}\cdots =\x_0$ in the following analysis. We have
\begin{align}
\small\begin{split}
\label{eq:ncf}
    &f(\x_d)-f(\x_K) = \sum_{i=d}^{K-1} f(\x_i)-f(\x_{i+1})\\
    \overset{\eqref{eq:lazy-update}}{\geq}& \sum_{i=d}^{K-1}\left(\gamma(\x_{i+1}) + \frac{M}{48}\|\x_{i+1}-\x_i\|^3 - \frac{11L^3}{M^3}\|\x_i-\x_{\tau(i;d)}\|^3\right)\\
   =& \left(\sum_{i=d}^{K-1} \gamma(\x_{i+1})\right) + \sum_{i=d}^{K-1}\left(\frac{M}{48}\|\x_{i+1}-\x_i\|^3 - \frac{11L^3}{M^2}\|\x_i-\x_{\tau(i;d)}\|^3\right).
    \end{split}
    % & = \left(\sum_{i=m}^K \gamma(\x_{i+1})\right) +\sum_{i=m}^{mt}\left(\frac{M}{48}\|\x_{i+1}-\x_i\|^3 - \frac{11L^3}{M^2}\|\x_i-\x_{\pi(i;m)-m}\|^3\right) 
\end{align}
We first focus on the term of
\begin{align*}
    &\sum_{i=d}^{K-1}\|\x_i-\x_{\tau(i;d)}\|^3=\underbrace{\sum_{i=d}^{dt-1} \|\x_i-\x_{\tau(i;d)}\|^3}_{A} + \underbrace{\sum_{i=dt}^{dt+p-1} \|\x_i-\x_{\tau(i;d)}\|^3}_{B}.
\end{align*}
For any integer $N\geq 1$, we have
\begin{align}
\begin{split}
\label{eq:batchsum}
    \sum_{i=Nd}^{(N+1)d-1}\|\x_i-\x_{\tau(i;d)}\|^3 
    \leq & \sum_{i=Nd}^{(N+1)d-1}\left(\sum_{k=(N-1)d}^{i-1} \|\x_{k+1}-\x_{k}\|\right)^3\\
    \leq & 3(d+1)^3\sum_{i=(N-1)d}^{(N+1)d-1}\|\x_{i+1}-\x_i\|^3,
\end{split}
\end{align}
where the first step comes from the triangle inequality such that
\begin{align*}
     \|\x_i-\x_{\tau(i;d)}\| 
    =\left\|\sum_{k=\tau(i;d)}^{i-1} (\x_{k+1}-\x_k)\right\|
    \leq \sum_{k=\tau(i;d)}^{i-1}\|\x_{k+1}-\x_k\|,
\end{align*}
and the second step is obtained by using Lemma~\ref{lm:first-r} with 
\begin{align*}
    r_k=\|\x_{(N-1)d+k+1}-\x_{(N-1)d+k}\|,
\end{align*}
which leads to 
\begin{align*}
     &\sum_{i=Nd}^{(N+1)d-1}\left(\sum_{k=(N-1)d}^{i-1} \|\x_{k+1}-\x_{k}\|\right)^3 = \sum_{i=d}^{2d-1}\left(\sum_{k=0}^{i-1}r_k\right)^3
    \\
    &\leq\sum_{i=d}^{2d-1}\left(\sum_{k=0}^{i}r_k\right)^3\overset{\eqref{eq:induct1}}{\leq} 3(d+1)^3\sum_{k=0}^{2d-1}r_k^3 \\
     &= 3(d+1)^3\sum_{i=(N-1)d}^{(N+1)d-1}\|\x_{i+1}-\x_i\|^3.
\end{align*}
Hence, we have
\begin{align*}
    A \leq & \sum_{N=1}^{t}\left(\sum_{i=Nd}^{(N+1)d-1}\|\x_i-\x_{\tau(i;d)}\|^3 \right)\\
    \overset{\eqref{eq:batchsum}}{\leq} & 6(d+1)^3\sum_{k=d}^{(t-1)d-1}\|\x_{k+1}-\x_k\|^3 \\
    &+3(d+1)^3 \left(\sum_{k=0}^{d-1}\|\x_{k+1}-\x_k\|^3+\sum_{k=(t-1)d}^{td} \|\x_{k+1}-\x_k\|^3\right).
\end{align*}
Using Lemma~\ref{lm:last-r} with $r_k=\|\x_{d(t-1)+k+1}-\x_{d(t-1)+k}\|$, we obtain
\begin{align*}
    B&=\sum_{i=dt}^{dt+p}\|\x_{i}-\x_{d(t-1)}\|^3\leq \sum_{i=d}^{d+p}\left(\sum_{k=0}^{i-1}r_k^3\right)\\
    &\overset{\eqref{eq:last-r}}{\leq} 4(m+1)^3\sum_{i=d(t-1)}^{dt-1}\|\x_{i+1}-\x_i\|^3 + 4(d+1)^3\sum_{dt}^{dt+p}\|\x_{i+1}-\x_i\|^3.
\end{align*}
Combining above results, we have
\begin{align}
\label{eq:leftsumup}
\begin{split}
  \!\!& \sum_{i=d}^{K-1}\|\x_i-\x_{\tau(i;d)}\|^3 =   A+B \\
  \!\!\leq & 7(d+1)^3  \sum_{k=d}^{K-1} \|\x_{k+1}-\x_k\|^{3} + 3(d+1)^3 \sum_{k=0}^{d-1}\|\x_{k+1}-\x_k\|^3,
\end{split}
\end{align}
which implies
\begin{align*}
     & \frac{M}{48}\sum_{i=d}^{K-1}\left(\|\x_{i+1}-\x_i\|^3\right) - \frac{11L^3}{M^2}\sum_{i=d}^{K-1}\left(\|\x_i-\x_{\tau(i;d)}\|^3\right)\\
     \overset{\eqref{eq:leftsumup}}{\geq} & \left(\frac{M}{48}-\frac{77(d+1)^3L^3}{M^2}\right) \sum_{i=d}^{K-1}\|\x_{i+1}-\x_i\|^3 \\
     &\quad-\frac{ 33L^3 (d+1)^3}{M^2}\sum_{i=0}^{d-1}\|\x_{i+1}-\x_i\|^3 \geq 0.
\end{align*}
The last inequality comes from the facts that we set $M=12dL$ and~$\x_0=\x_1\cdots\x_{d-1}=\x_d$.
Connecting above results to inequality~\eqref{eq:ncf}, we obtain
\begin{align*}
   &f(\x_0)-f^*\geq f(\x_d)-f(\x_K) \\
   &\overset{\eqref{eq:ncf}}{\geq}  \sum_{i=d}^{K-1} \gamma(\x_{i+1})+\frac{M}{48}\sum_{i=d}^{K-1}\left(\|\x_{i+1}-\x_i\|^3\right) \\
   &\quad\quad- \frac{11L^3}{M^2}\sum_{i=d}^{K-1}\|\x_i-\x_{\tau(i;d)}\|^3\geq\sum_{i=d}^{K-1} \gamma(\x_{i+1}),
\end{align*}   
which proves~\eqref{eq:ncobj}.
% \begin{align*}
%     \min_{d \leq i \leq K} \gamma(\x_i)\leq \frac{f(\x_0)-f^*}{K-d}.
% \end{align*}
By setting $K=\OM\big(\sqrt{dL}\epsilon^{-3/2}\big)$, we can find some $\x_i$ with $d \leq i\leq K$ such that $\x_i$ is a $\big(\epsilon,\sqrt{dL\epsilon}\big)$-second-order stationary point of $f(\cdot)$.
\end{proof}

\subsubsection{The Analysis in Strongly-Convex Case}
The classical second-order methods enjoy local superlinear convergence for minimizing the strongly-convex objective on single machine~\cite{nesterov2006cubic,nesterov2018lectures,nocedal1999numerical}. 
However, most of distributed second-order methods 
only achieve linear convergence rate because of the trade-off between the communication complexity and the convergence rate~\cite{wang2018giant,shamir2014communication,ye2022accelerated,reddi2016aide,zhang2015disco}.

In contrast, the efficient communication mechanism of C2EDEN still keeps the superlinear convergence like classical second-order methods. The following theorem formally presents this property.

\begin{theorem}
\label{theorem:local}
Under Assumption~\ref{ass:heslip} and \ref{ass:mustrongly},
running Algorithm~\ref{alg:C2EDEN} with $M\geq 0$ and the initial point $\vx_0$ such that
\begin{align}
    \label{eq:initial}
        \|\nabla f(\x_0)\|\leq \frac{\mu^2}{2(M+3L)},
\end{align}
then for any $k\geq d$, we have
\begin{align}\label{ieq:cvx-grad}
\|\nabla f(\x_k)\|\leq \frac{\mu^2}{M+3L}\left(\frac{1}{2}\right)^{h(k)}
\end{align}
where $  h(k)= {\left((1+d)^{\Floor{k/d}\%2}+k\%d\right)(1+d)^{\Floor{\Floor{k/d}/2}-1}}.$
% \begin{align}\label{ieq:hk}
%         h(k)= {\left((1+d)^{\Floor{k/d}\%2}+k\%d\right)(1+d)^{\Floor{\Floor{k/d}/2}-1}}.
% \end{align}
\end{theorem}
\begin{proof}
According to the analysis in Section 5 of \citet{doikov2022second}, the cubic-regularized update (\ref{eq:server-update}) satisfies
\begin{align*}
    \|\nabla f(\x_{k+1})\|{\leq} \frac{M+3L}{2\mu^2}\|\nabla f(\x_k)\|^2 + \frac{L}{\mu^2}\|\nabla f(\x_k)\|\|\nabla f(\x_{\tau(k;d)})\|.
\end{align*}
We let $c=\frac{M+3L}{\mu^2}$ and $s_{j}=c\|\nabla f(\x_{j})\|$, then above inequality can be written as
$  s_{k+1}\leq \frac{1}{2}s_k^2+\frac{1}{2}s_ks_{\iota(k;d)}.$
% \begin{align*}
%     s_{k+1}\leq \frac{1}{2}s_k^2+\frac{1}{2}s_ks_{\iota(k;d)}.
% \end{align*}
We first use induction to show
\begin{align}
\label{eq:scinduc1}
    s_{k+1}\leq s_k\quad\text{and}\quad s_k\leq \frac{1}{2}.
\end{align}
hold for any  $k\geq d$.
For $k=d$, the initial condition~\eqref{eq:initial} and the setting $\x_0=\x_1\cdots=\x_d$ in the algorithm means 
$  s_{d+1}\leq s_0^2 \leq s_0=s_d$ and $s_d=s_0\leq \frac{1}{2}$.
% \begin{align*}
%     s_{d+1}\leq s_0^2 \leq s_0=s_d\quad\text{and}\quad s_d=s_0\leq \frac{1}{2}.
% \end{align*}
Suppose the results of \eqref{eq:scinduc1} hold for any $k\leq k'-1$. Then for $k=k'$, we have $s_{k'+1}\leq \frac{1}{2} s_{k'}^2 + \frac{1}{2}s_{k'} s_{\tau(k';d)}\overset{\eqref{eq:scinduc1}}{\leq} s_{k'} \overset{\eqref{eq:scinduc1}}{\leq} \frac{1}{2},$
% \begin{align*}
%     s_{k'+1}\leq \frac{1}{2} s_{k'}^2 + \frac{1}{2}s_{k'} s_{\tau(k';d)}\overset{\eqref{eq:scinduc1}}{\leq} s_{k'} \overset{\eqref{eq:scinduc1}}{\leq} \frac{1}{2},
% \end{align*}
which finish the induction. Thus, we have
\begin{align}
\label{eq:sk1sk}
    s_{k+1}\leq\frac{1}{2}s_k^2+\frac{1}{2}s_ks_{\tau(k;d)}\overset{\eqref{eq:scinduc1}}{\leq} s_ks_{\tau(k;d)}.
\end{align}

Then we use induction to show 
\begin{align}
\label{eq:stm}
s_{td}\leq \left(\frac{1}{2}\right)^{(1+d)^{\Floor{(t+1)/2}-1}},
\end{align}
for all $t\geq 1$.
For $t=1$, we directly have $s_d=s_0\leq \frac{1}{2}$.
% \begin{align*}
% s_d=s_0\leq \frac{1}{2}.    
% \end{align*}
Suppose the inequality \eqref{eq:stm} holds for $t=t'$.
For $t=t'+1$, we have 
\begin{align}
\label{eq:t+1}
\begin{split}
    s_{(t'+1)d}&\overset{\eqref{eq:scinduc1}}{\leq} s_{(t'+1)d-1}s_{t'd-d}\overset{\eqref{eq:sk1sk}}{\leq} s_{(t'+1)d-2}s_{(t'-1)d}^2\\
    &\leq\cdots\leq s_{t'd}s_{(t'-1)d}^{d}\\
    & \overset{\eqref{eq:stm}}{\leq} \left(\frac{1}{2}\right)^{(1+d)^{\Floor{(t'+1)/2}-1}}\left(\frac{1}{2}\right)^{(1+d)^{\Floor{t'/2}-1}d}.
\end{split}
\end{align}
If $t'$ is an even number, we can write $t'=2q$ and it holds that
\begin{align*}
    s_{(t'+1)d} \overset{\eqref{eq:t+1}}{\leq}& \left(\frac{1}{2}\right)^{(1+d)^{q-1}+(1+d)^{q-1}d}\\
    =& \left(\frac{1}{2}\right)^{(1+d)^{q}} 
    = \left(\frac{1}{2}\right)^{(1+d)^{\Floor{(t'+2)/2}-1}}.
\end{align*}
If $t'$ is an odd number, we can write $t'=2q+1$ and it holds that 
\begin{align*}
    s_{(t'+1)d} \overset{\eqref{eq:t+1}}{\leq}& \left(\frac{1}{2}\right)^{(1+d)^{q}+(1+d)^{q-1}d}\\
    \leq & \left(\frac{1}{2}\right)^{(1+d)^{q}} 
    =  \left(\frac{1}{2}\right)^{(1+d)^{\Floor{(t'+2)/2}-1}}.
\end{align*}
Above discussion finishes the induction.

For the $k$-th iteration, we write $k=dt+p$, where $t=\floor{k/d}$ and~$p=k\%d$. Then, we have
\begin{align*}
\small
\begin{split}    
    s_k \overset{\eqref{eq:sk1sk}}{\leq}&  s_{k-1}s_{(t-1)d}\leq s_{td}s^{p}_{(t-1)d}\\
    \overset{\eqref{eq:stm}}{\leq} & \left(\frac{1}{2}\right)^{(1+d)^{\Floor{(t+1)/2}-1}}\left(\frac{1}{2}\right)^{(1+d)^{\Floor{t/2}-1}p}\leq  \left(\frac{1}{2}\right)^{\left((1+d)^{t\%2}+p\right)(1+d)^{\Floor{t/2}-1}}.
\end{split}    
\end{align*}
Substituting the definition of $s_k$ into above result, we obtain~\eqref{ieq:cvx-grad}.
% \begin{align*}
%           \|\nabla f(\x_k)\|\leq \frac{\mu^2}{M+3L}\left(\frac{1}{2}\right)^{\left((1+d)^{\Floor{k/d}\%2}+k\%d\right)(1+d)^{\Floor{\Floor{k/d}/2}-1}}.
% \end{align*}
% \begin{align}
% s_{(t+1)m}\leq\left(\frac{1}{2}\right)^{(1+m)^{t}+(1+m)^tm\leq \left(\frac{1}{2}\right)^{(1+m)^{k+1}} = \left(\frac{1}{2}\right)^{(1+m)^\Floor{(t+1)/2}}
% \end{align}
\end{proof}
\begin{remark}
We can verify the superlinear convergence of C2EDEN as follows
\begin{align*}
    \lim_{k\to\infty}\frac{\|\nabla f(\x_{k+1})\|}{\|\nabla f(\x_k)\|} = \lim_{k\to\infty} \frac{s_{k+1}}{s_k} \overset{\eqref{eq:sk1sk}}{\leq} \lim_{k\to\infty}s_{\tau(k;d)} \overset{\eqref{ieq:cvx-grad}}{=} 0.
\end{align*}
\end{remark}

\begin{remark}
Theorem~\ref{theorem:local} indicates the local superlinear convergence rate of C2CEDEN in strongly-convex case can be obtained by taking $M=0$, which leads to the step of line 23 in Algorithm~\ref{alg:C2EDEN} has the closed form expression of $\x_{k+1}=\H^{-1}\g_k$.
\end{remark}

% The local superlinear rate only requires the object function is strongly convex in a local region of $\x^*$ and can make the C2EDEN performs much better in practice even when $f(\,\cdot\,)$ is non-convex.  

\section{Experiment}
\label{sec:exp}

\begin{figure*}[ht]
    \centering
    \begin{tabular}{cccc}
    \includegraphics[scale=0.22]{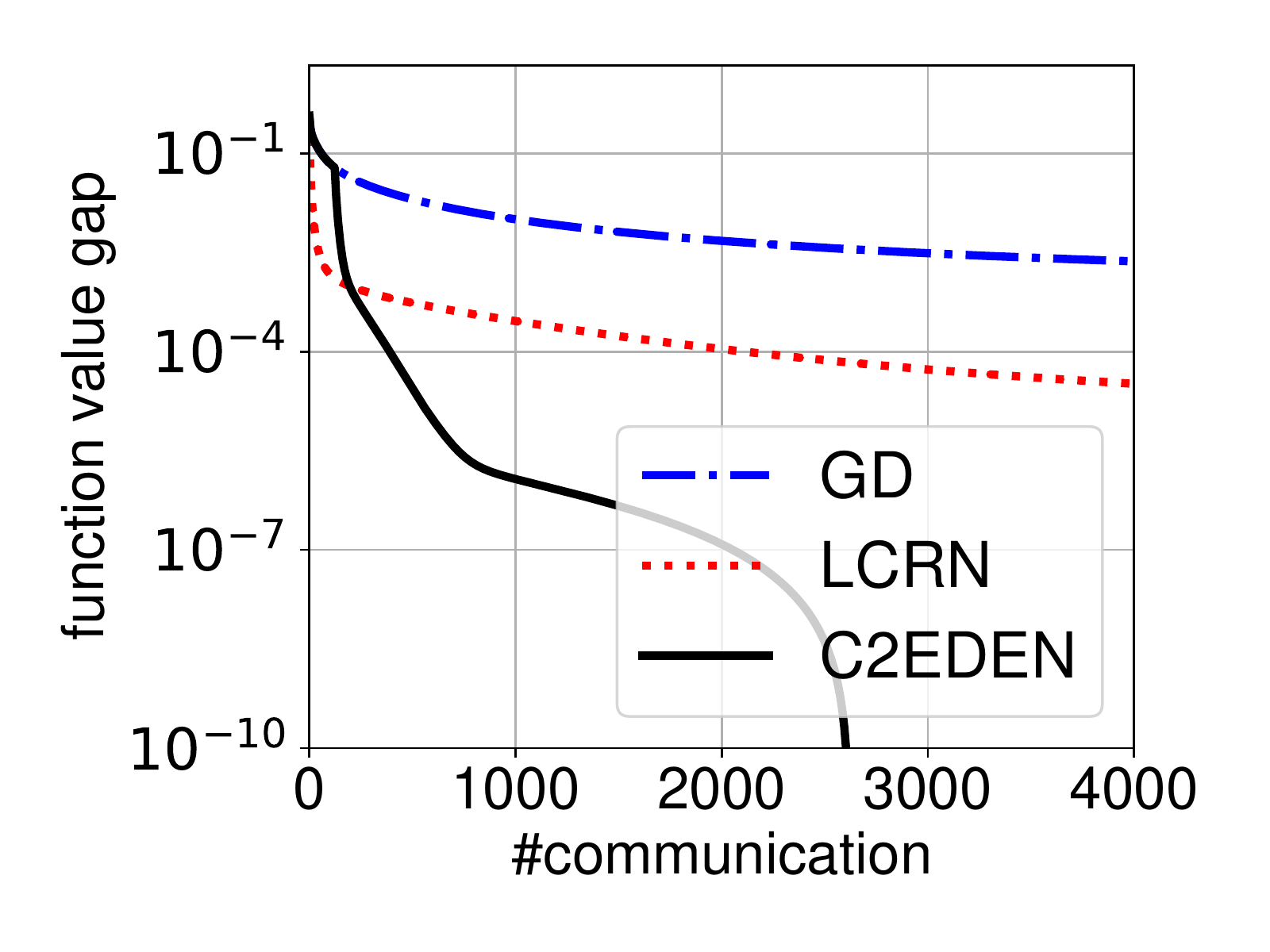} &  \includegraphics[scale=0.22]{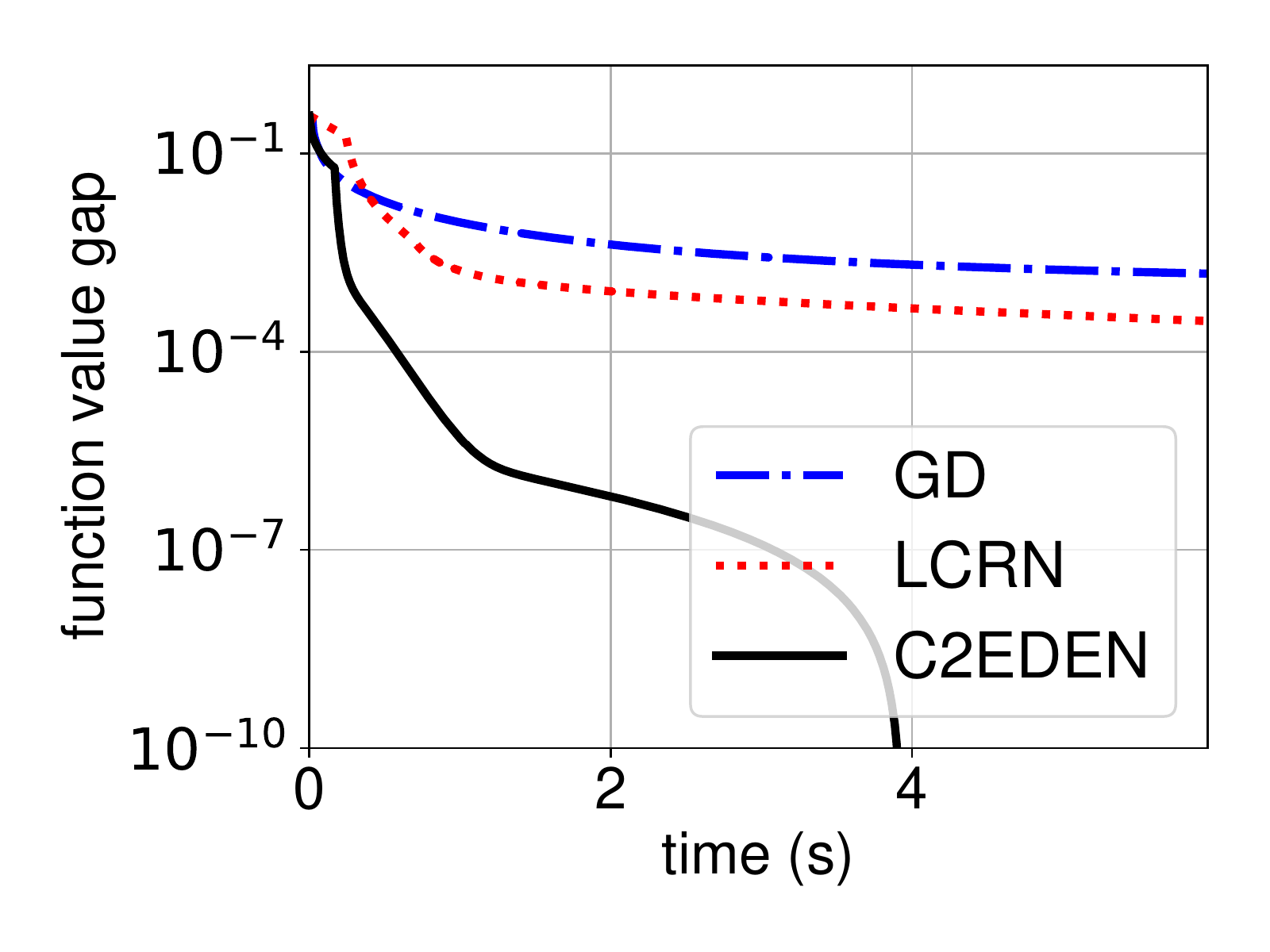} &
    \includegraphics[scale=0.22]{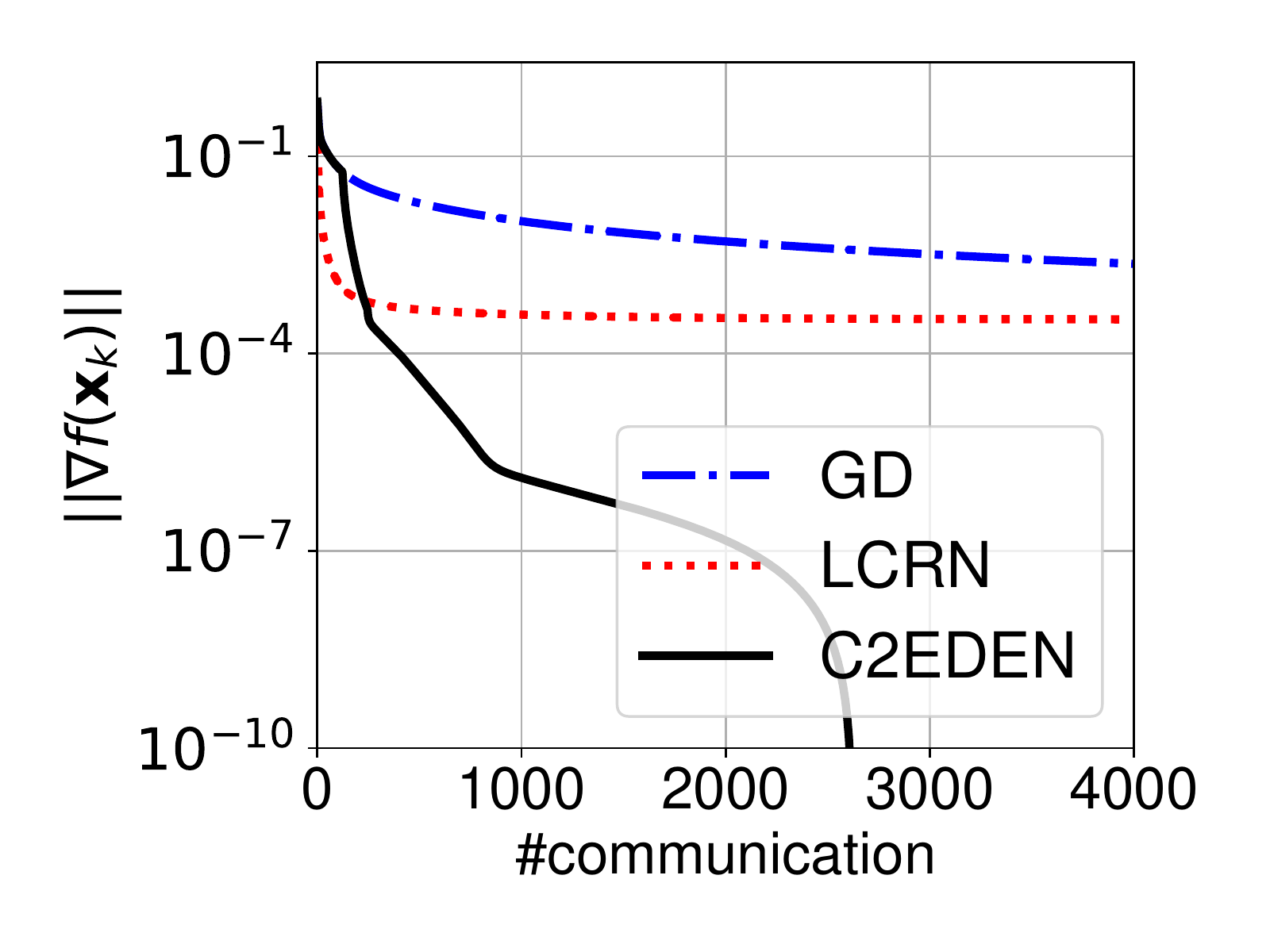} &
    \includegraphics[scale=0.22]{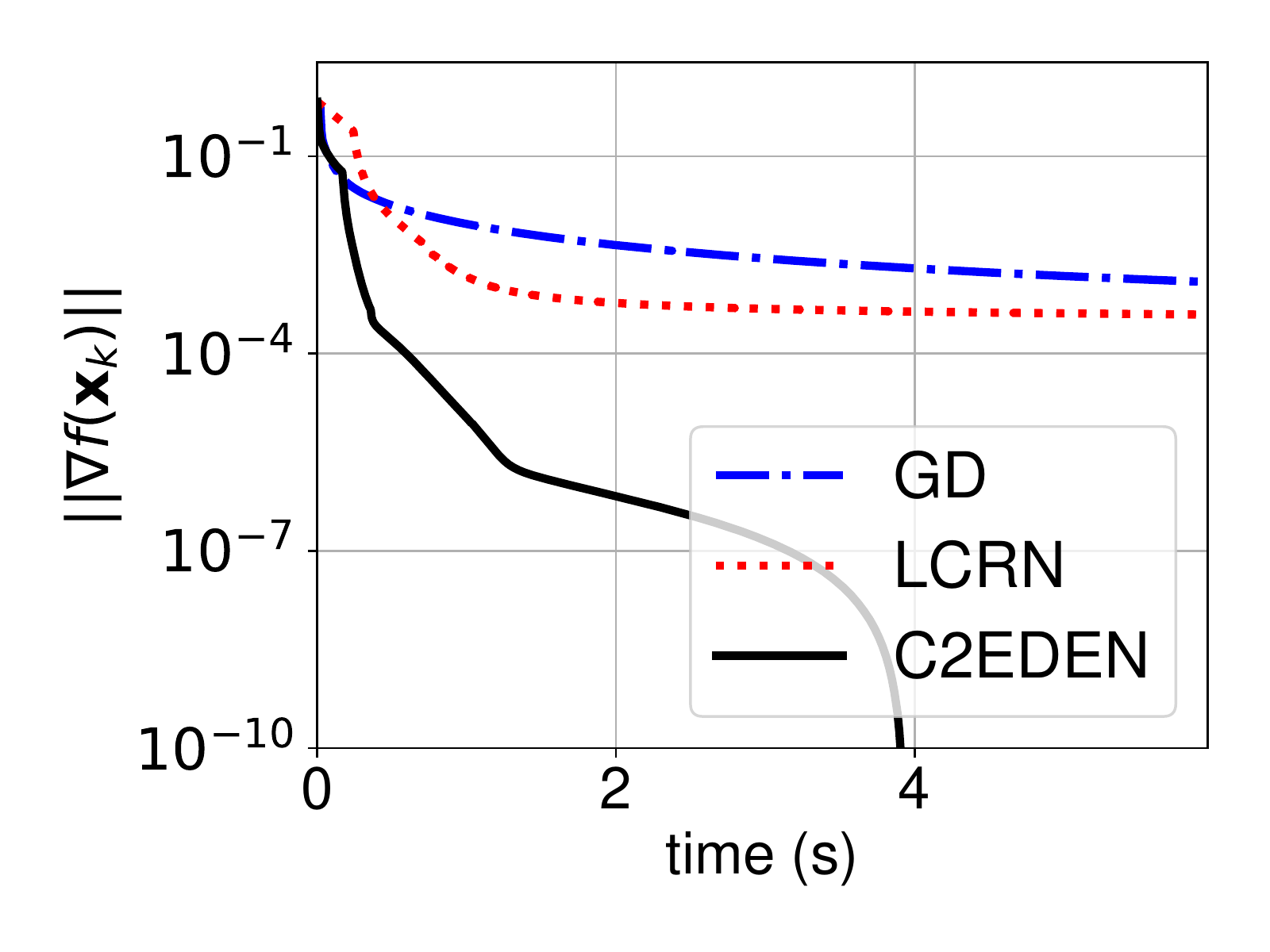} \\
    (a)  $ \#$communication vs. gap    &
    (b)   time (s) vs. gap    &
    (c) $ \#$communication vs. $ \Vert \nabla f(x_k) \Vert$     &
    (d)  time (s) vs. $ \Vert \nabla f(x_k) \Vert$
    \end{tabular}
    \caption{The results of the model of nonconvex regularized logistic regression on a9a ($n$=32).}
    \label{fig:nc-a9a}
\end{figure*}

\begin{figure*}[ht]
    \centering
    \begin{tabular}{cccc}
    \includegraphics[scale=0.22]{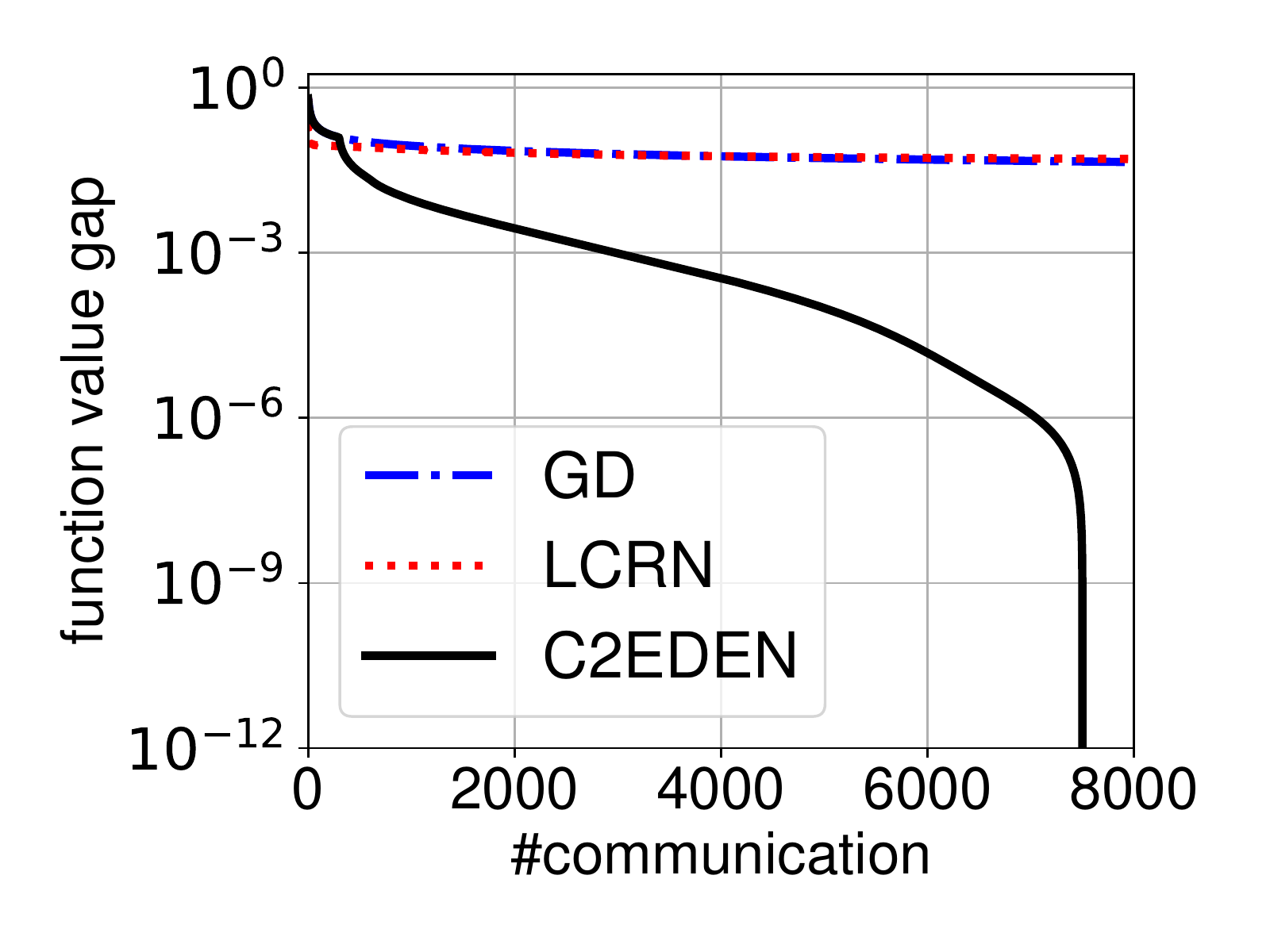}     &  \includegraphics[scale=0.22]{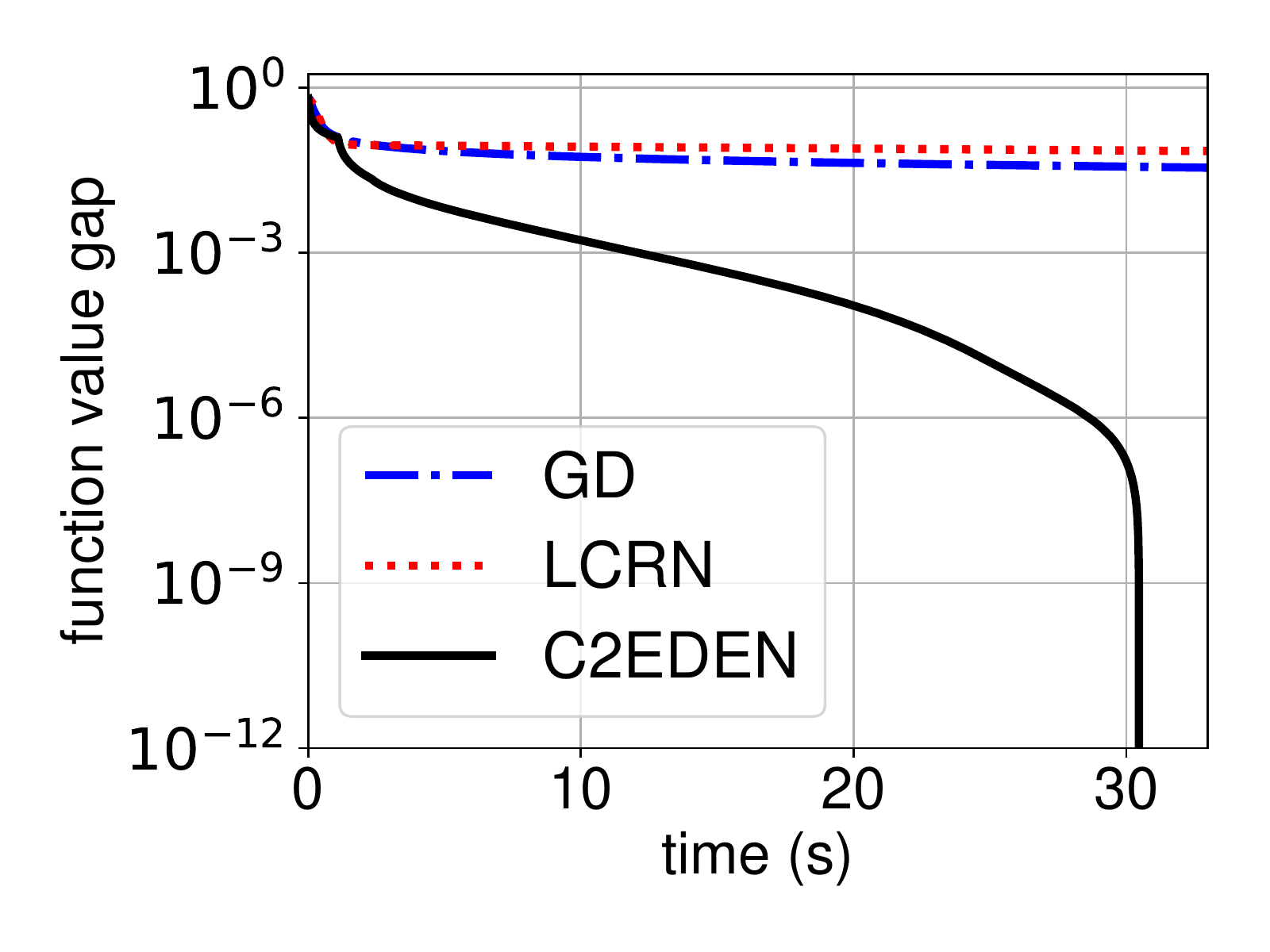} &
    \includegraphics[scale=0.22]{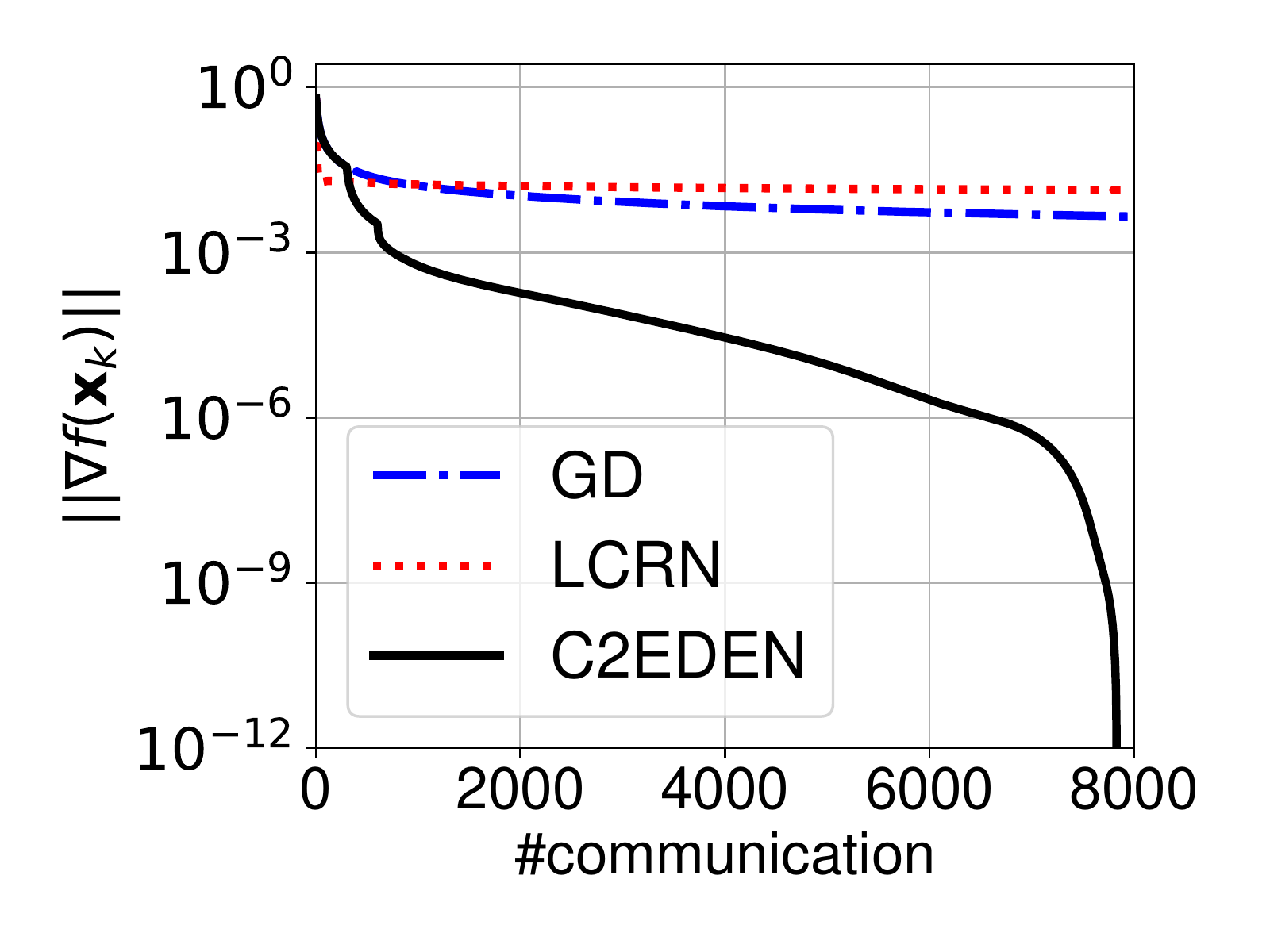}  &
    \includegraphics[scale=0.22]{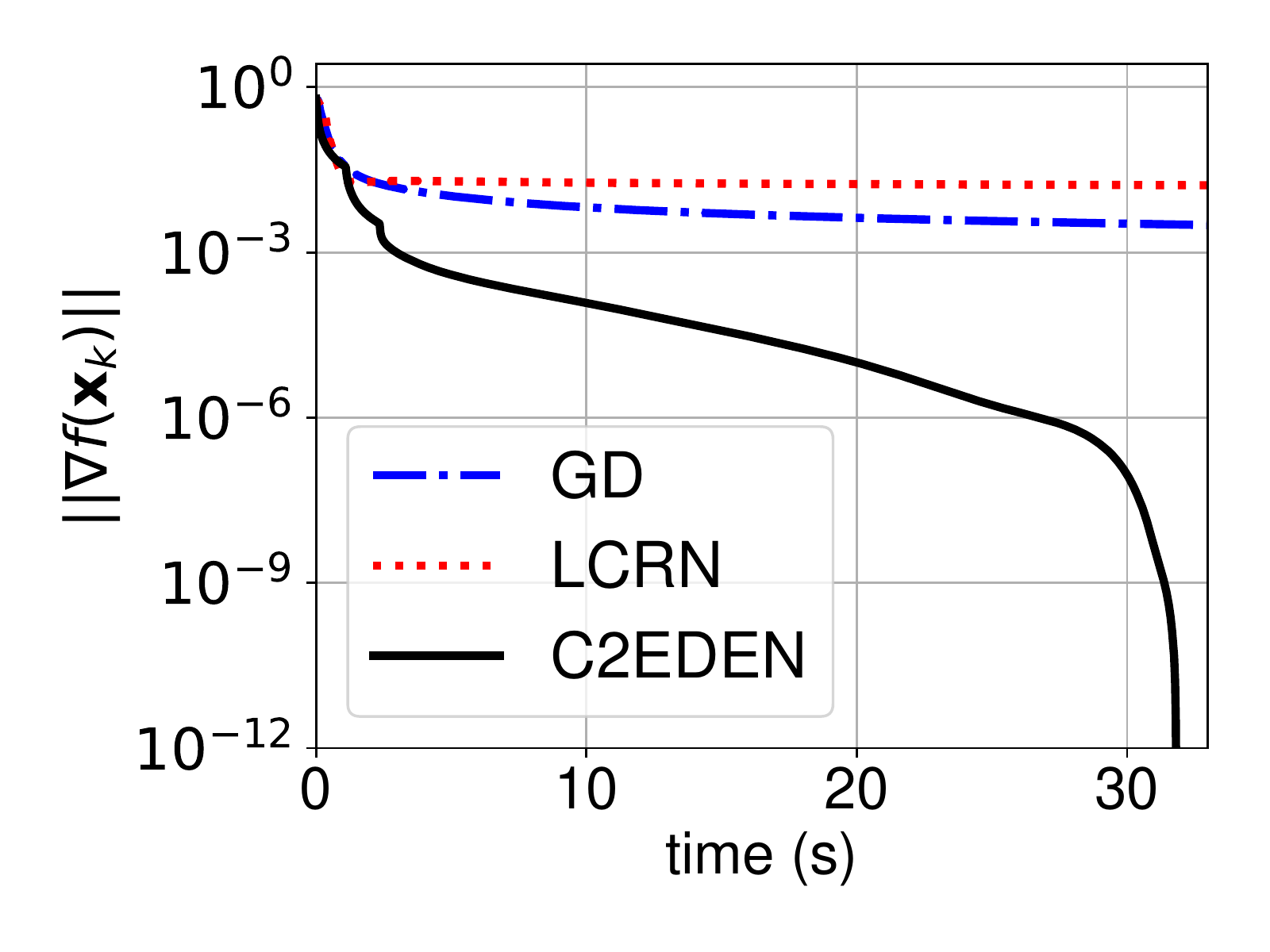}  \\
     (a)   $ \#$communication vs. gap    &
    (b)   time (s) vs. gap     &
    (c) $ \#$communication vs. $ \Vert \nabla f(x_k) \Vert$     &
    (d)   time (s) vs. $ \Vert \nabla f(x_k) \Vert$
    \end{tabular}
    \caption{The results of the model of nonconvex regularized logistic regression on w8a ($n$=32).}
    \label{fig:nc-w8a}
\end{figure*}

\begin{figure*}[ht]
    \centering
    \begin{tabular}{cccc}
    \includegraphics[scale=0.22]{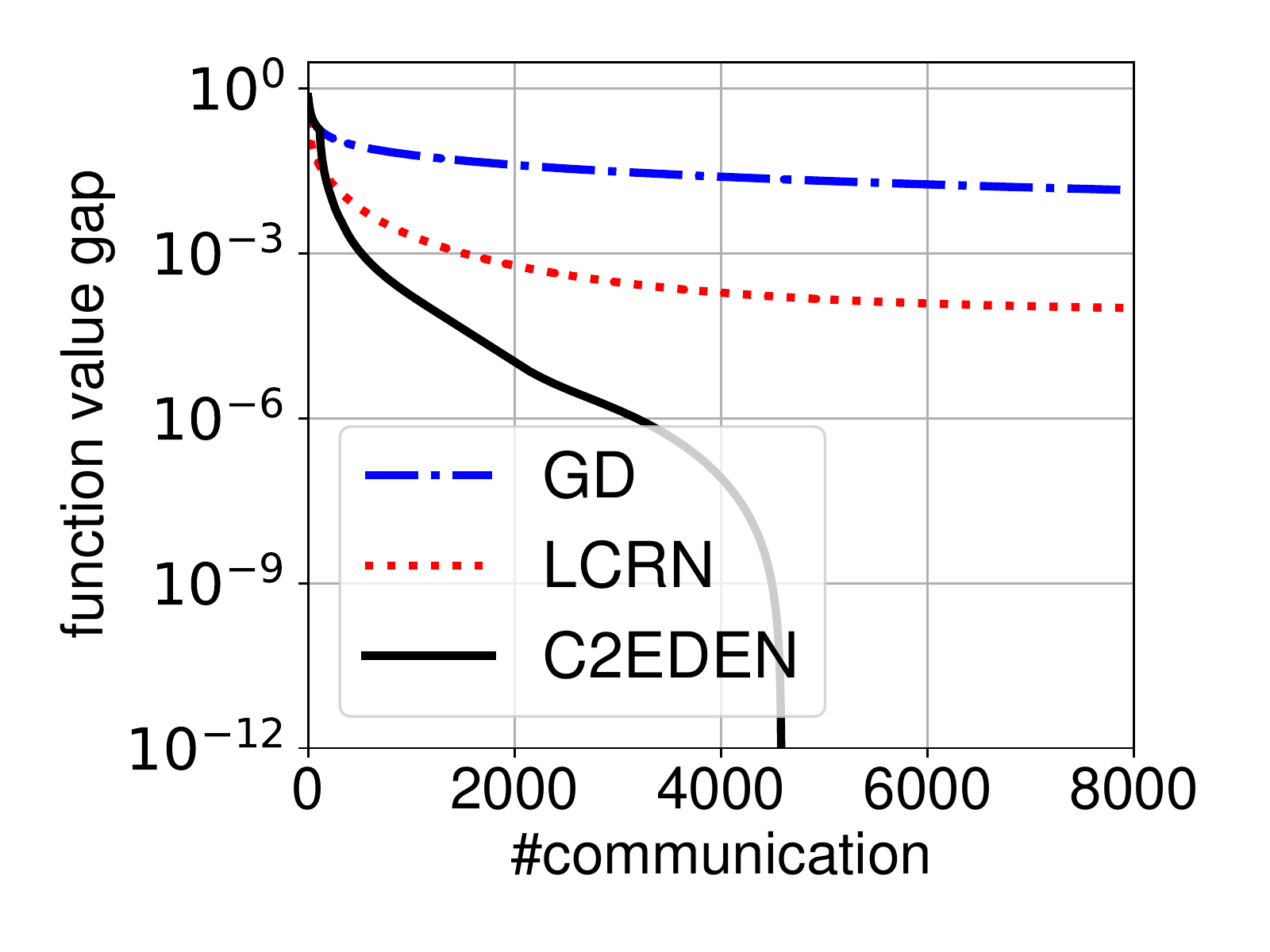}     &  \includegraphics[scale=0.22]{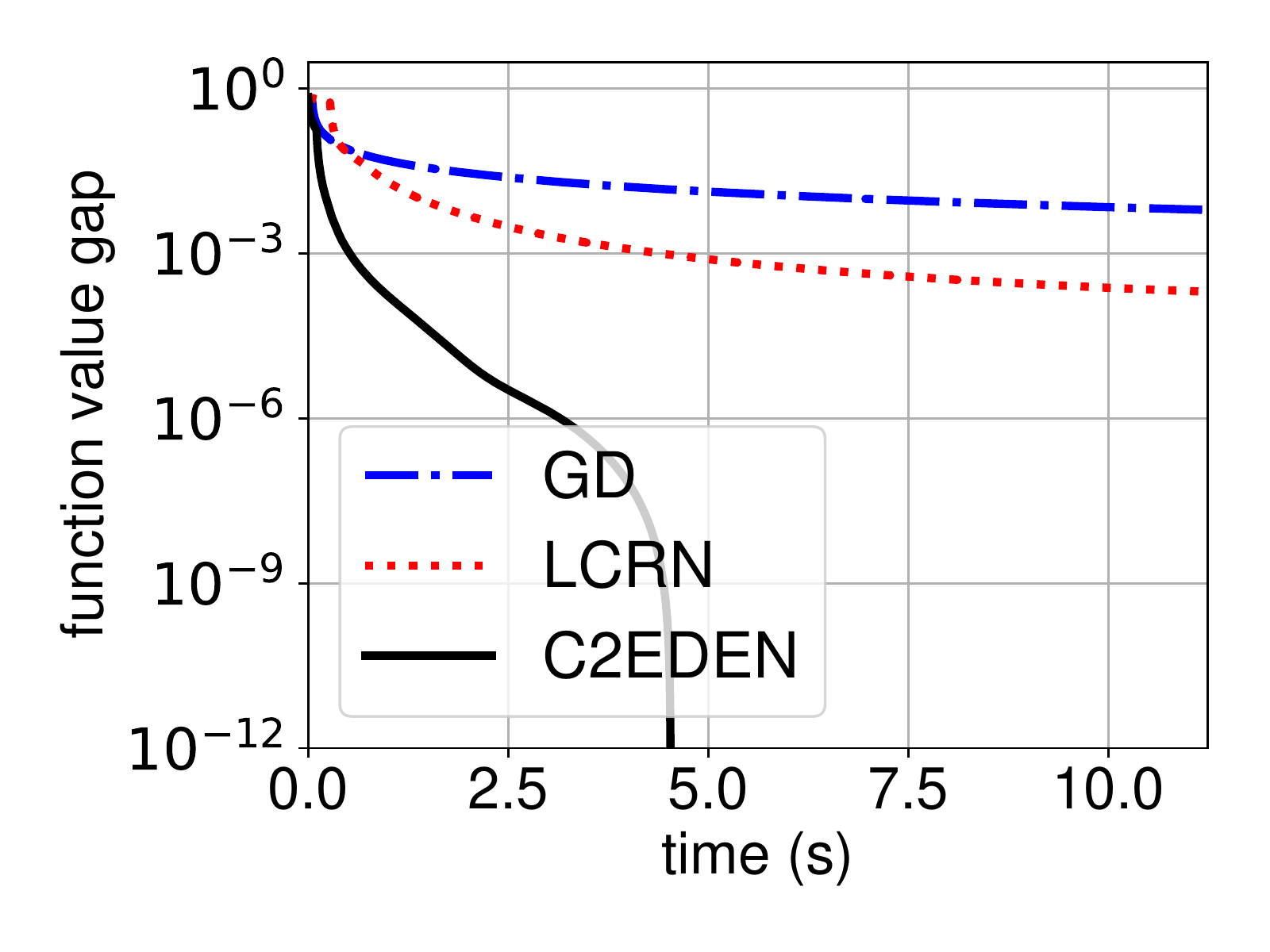} &
    \includegraphics[scale=0.22]{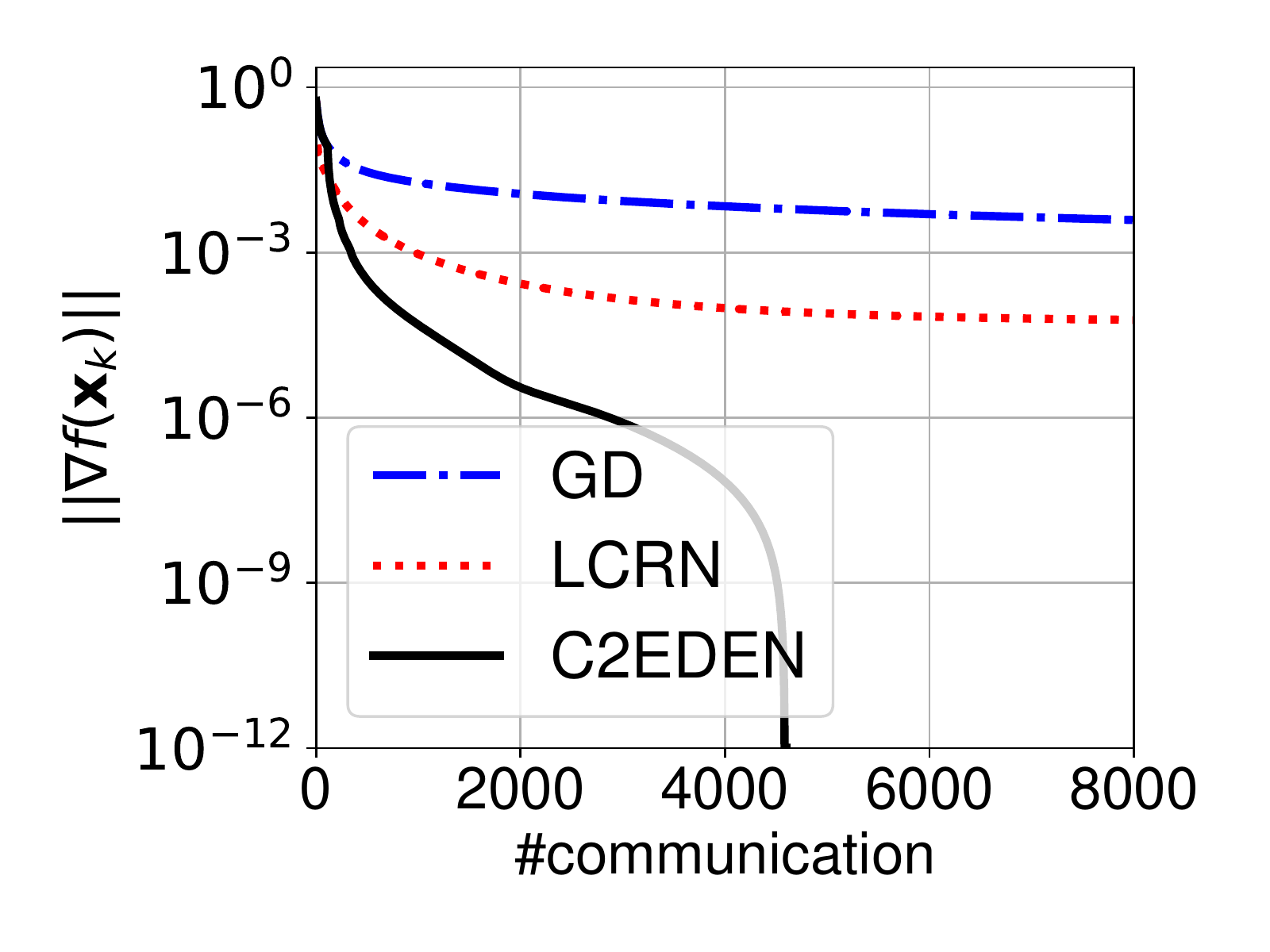}  &
    \includegraphics[scale=0.22]{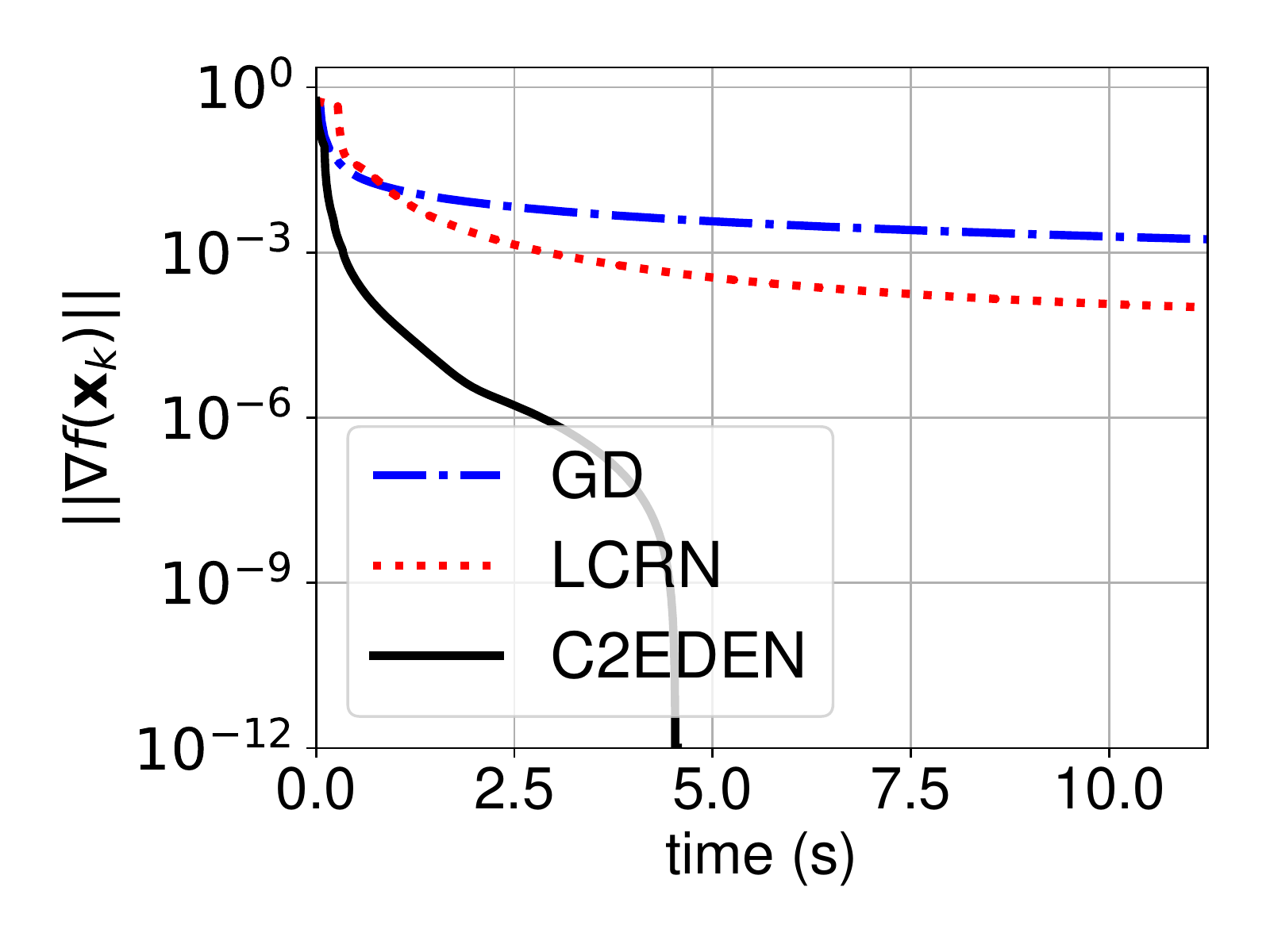}  \\
     (a)  $ \#$communication vs. gap    &
    (b)  time (s) vs. gap    &
    (c) $ \#$communication vs. $ \Vert \nabla f(x_k) \Vert$     &
    (d)   time (s) vs. $ \Vert \nabla f(x_k) \Vert$
    \end{tabular}
    \caption{The results of the model of nonconvex regularized logistic regression on mushrooms ($n$=32).}
    \label{fig:nc-mushrooms}
\end{figure*}

\begin{figure*}[ht]
    \centering
    \begin{tabular}{cccc}
    \includegraphics[scale=0.22]{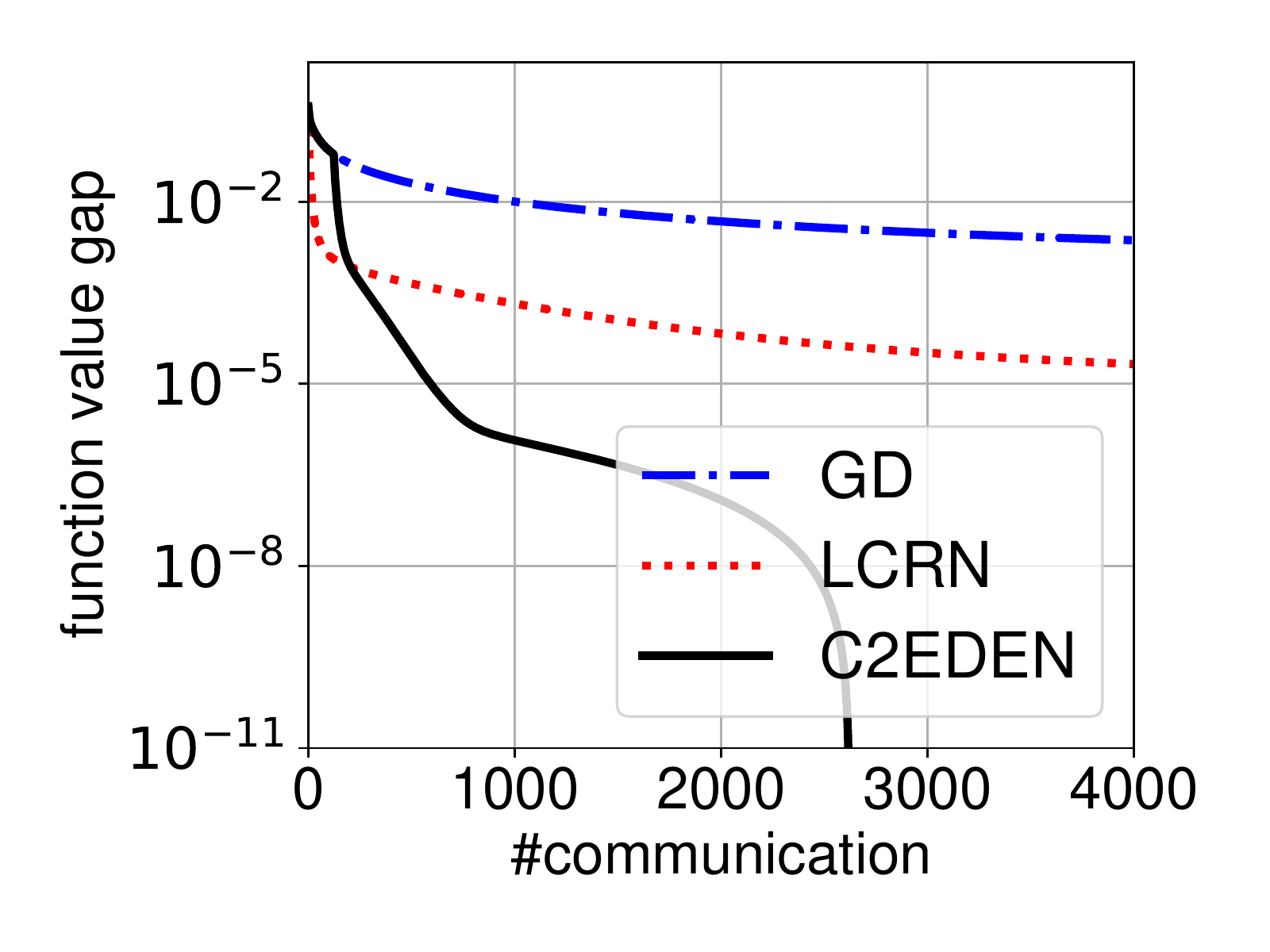}     &  \includegraphics[scale=0.22]{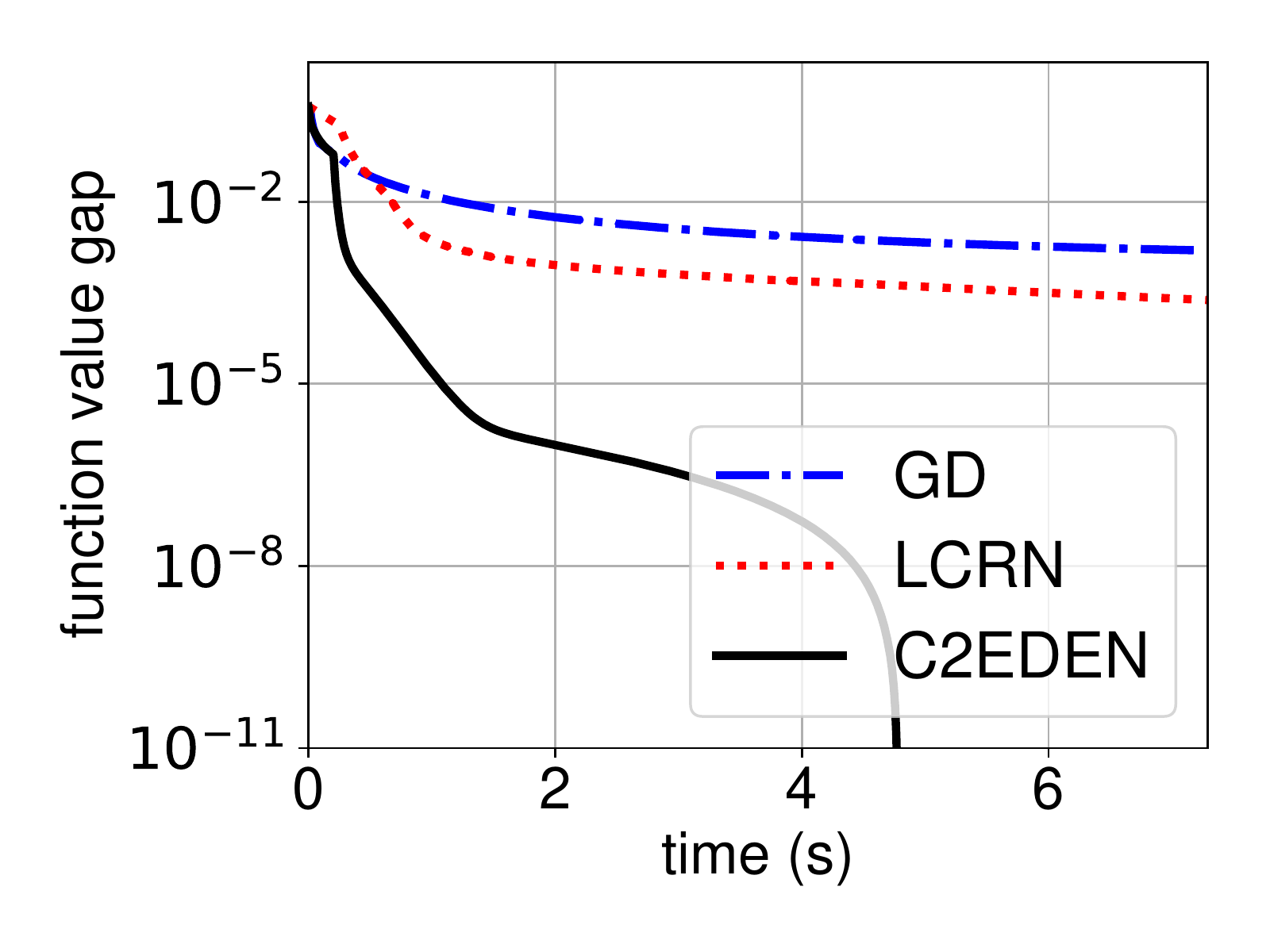} &
    \includegraphics[scale=0.22]{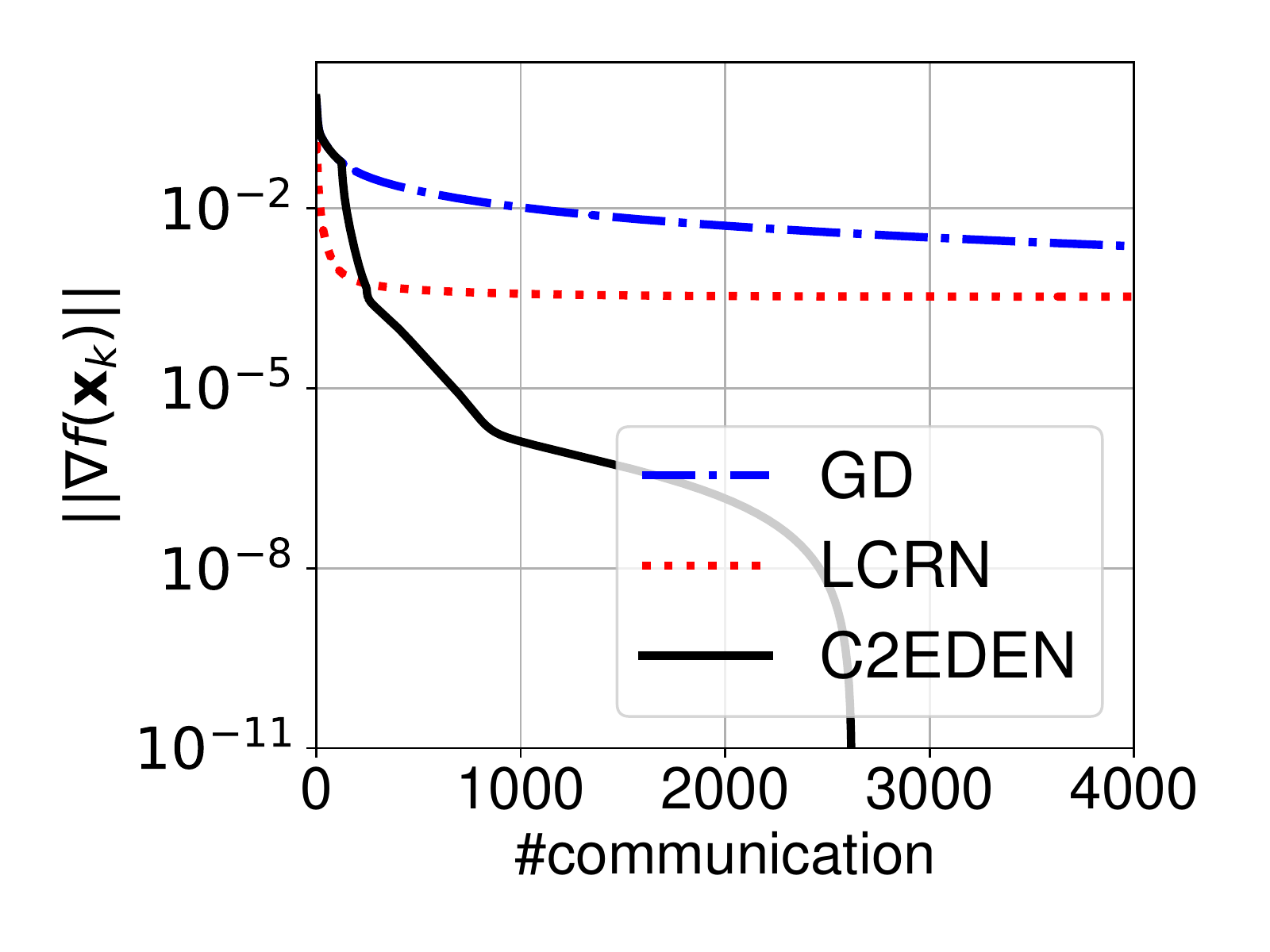} &
    \includegraphics[scale=0.22]{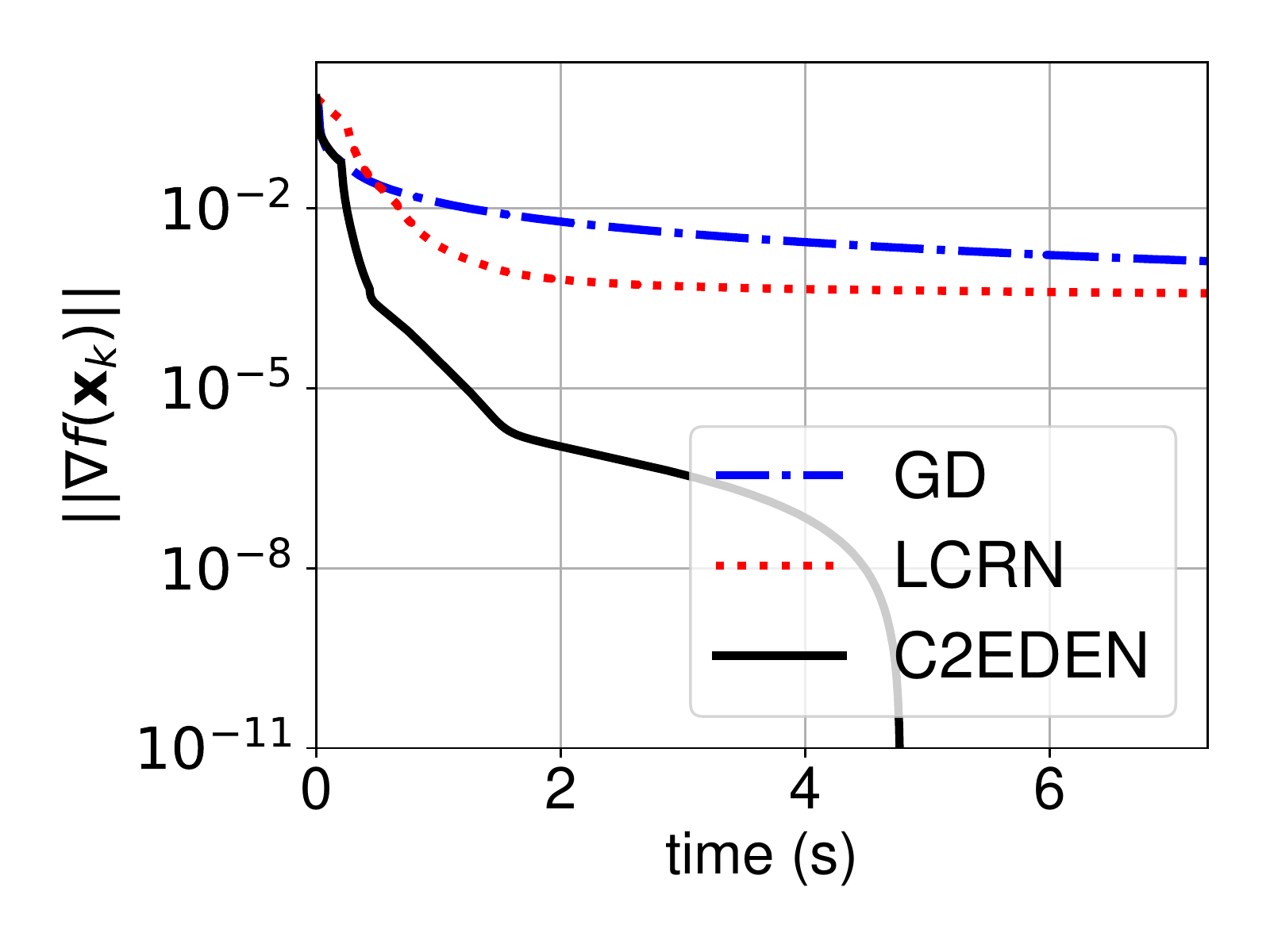}  \\
    (a)   $ \#$communication vs. gap    &
    (b)   time (s) vs. gap      &
    (c) $ \#$communication vs. $ \Vert \nabla f(x_k) \Vert$    &
    (d)  time (s) vs. $ \Vert \nabla f(x_k) \Vert$
    \end{tabular}
    \caption{The results of the model of nonconvex regularized logistic regression on a9a ($n$=16).}
    \label{fig:nc-a9a16}
\end{figure*}

\begin{figure*}[ht]
    \centering
    \begin{tabular}{cccc}
    \includegraphics[scale=0.22]{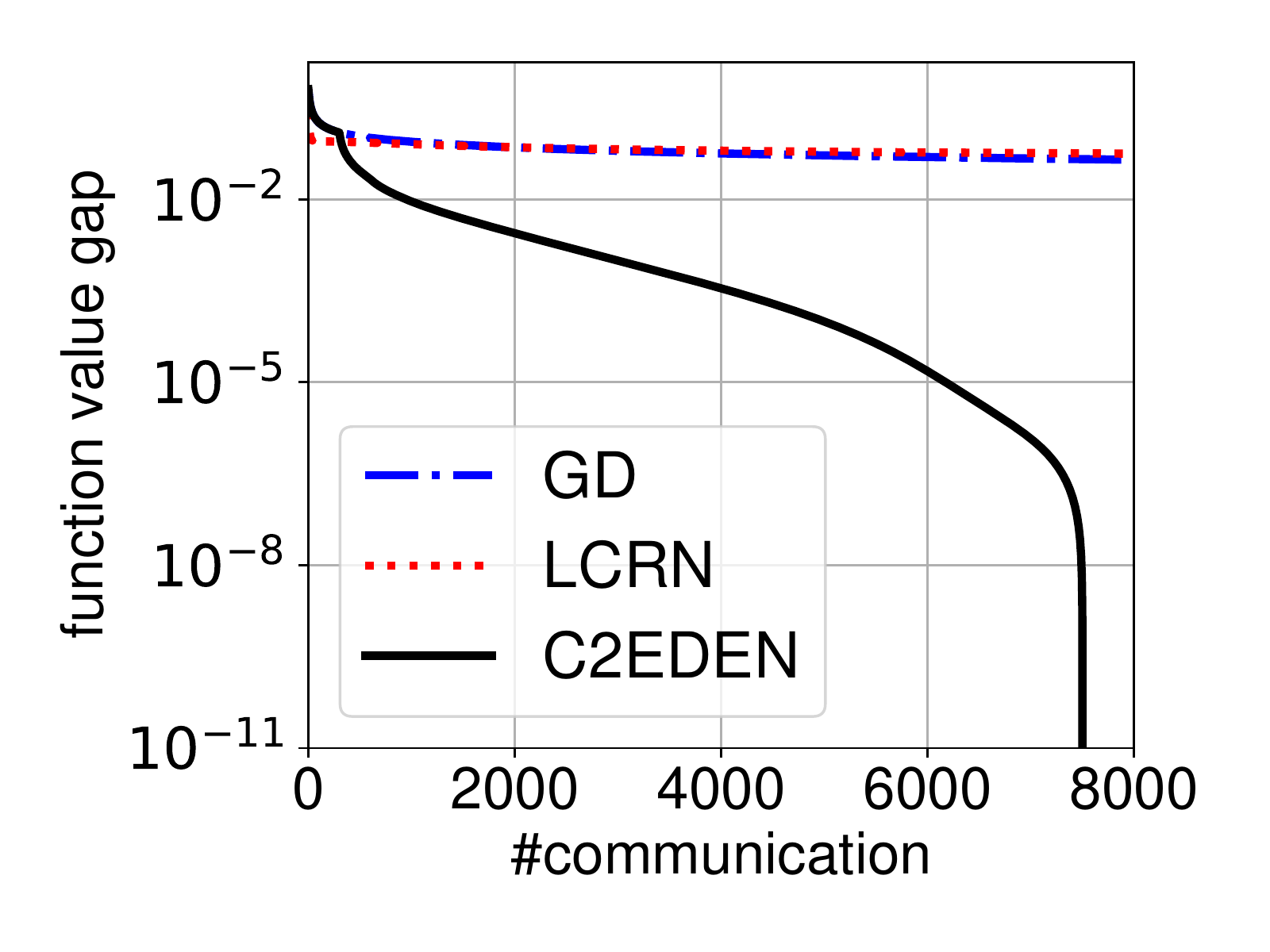}     &  \includegraphics[scale=0.22]{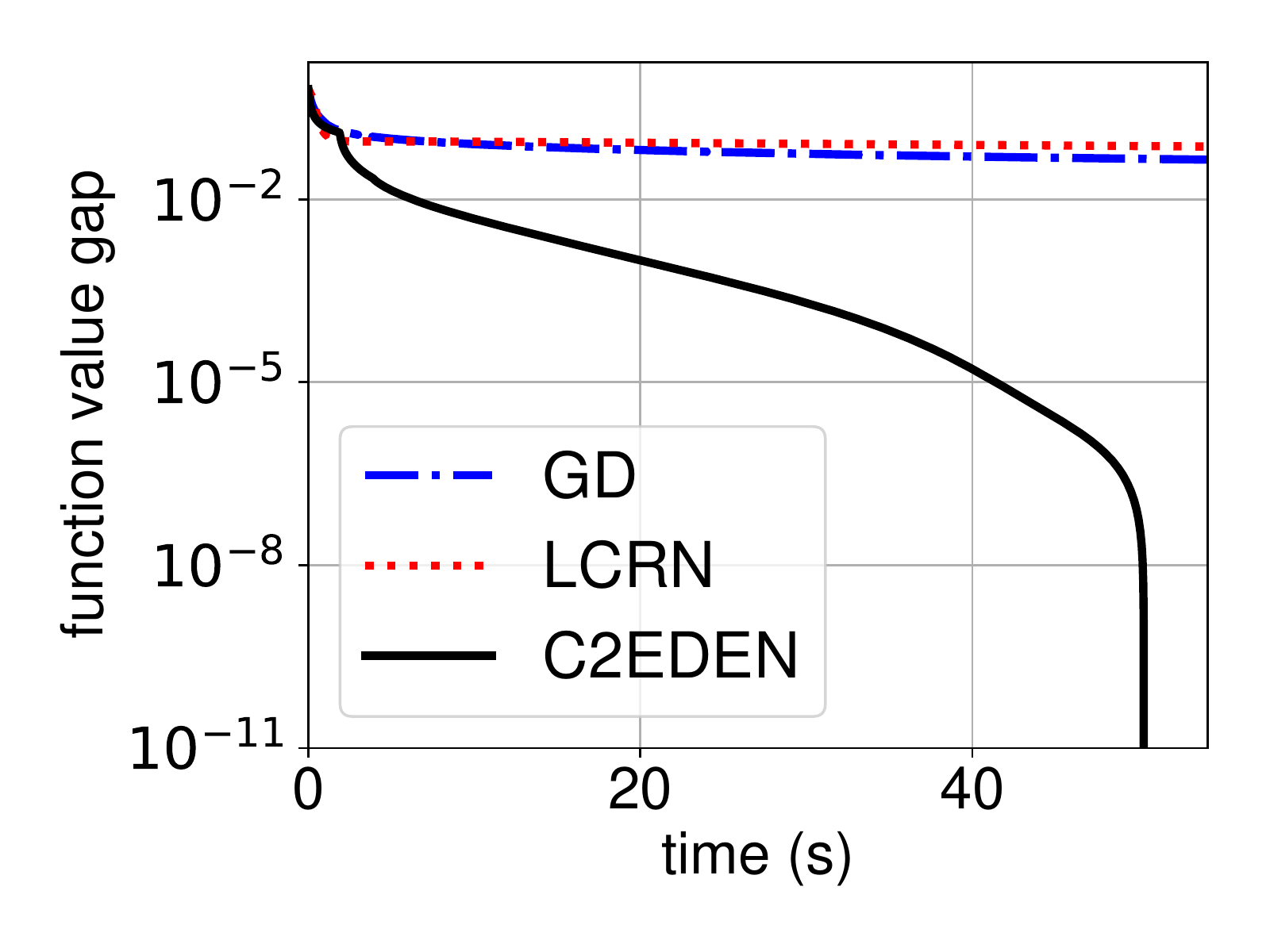} &
    \includegraphics[scale=0.22]{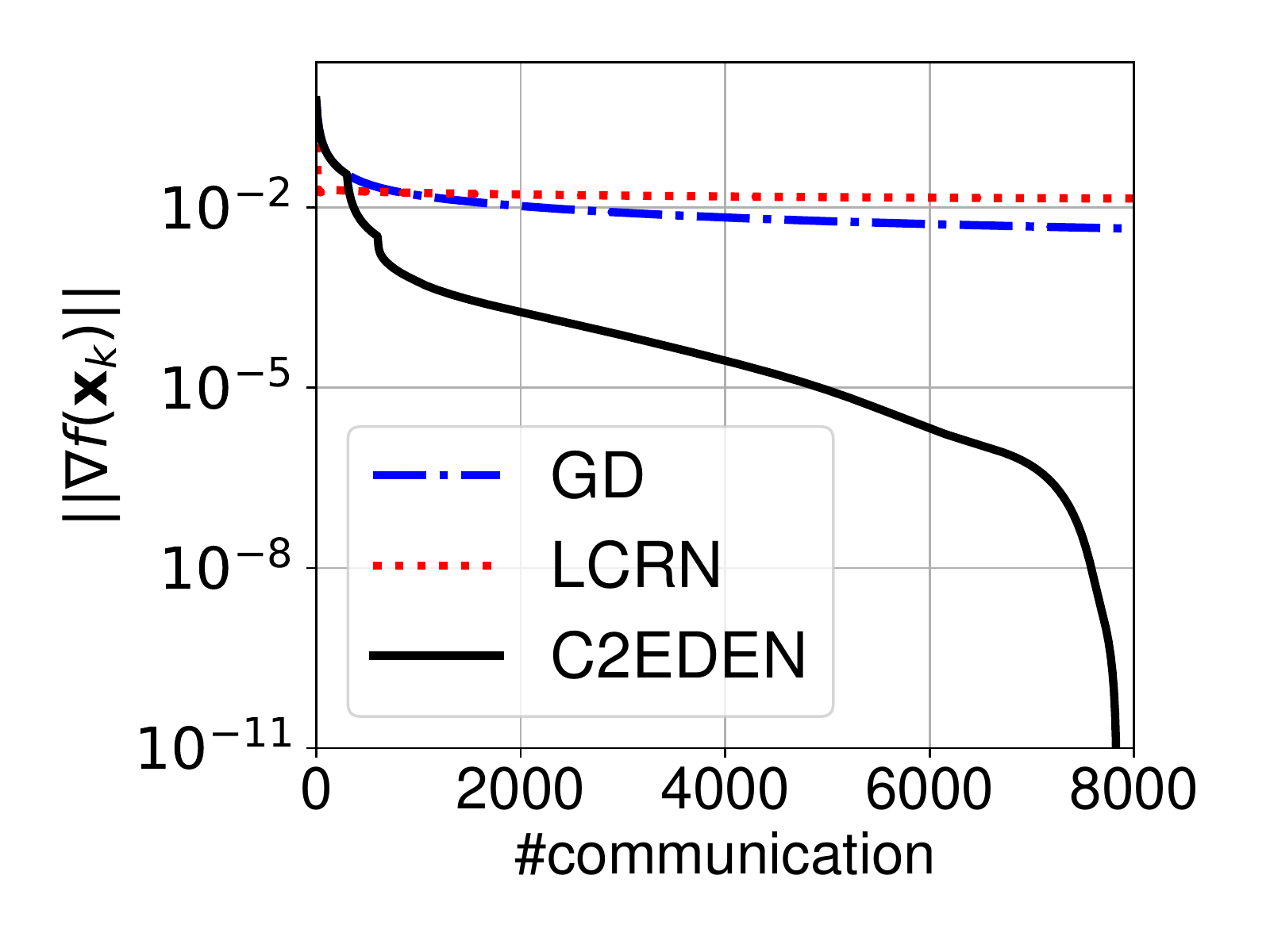}  &
    \includegraphics[scale=0.22]{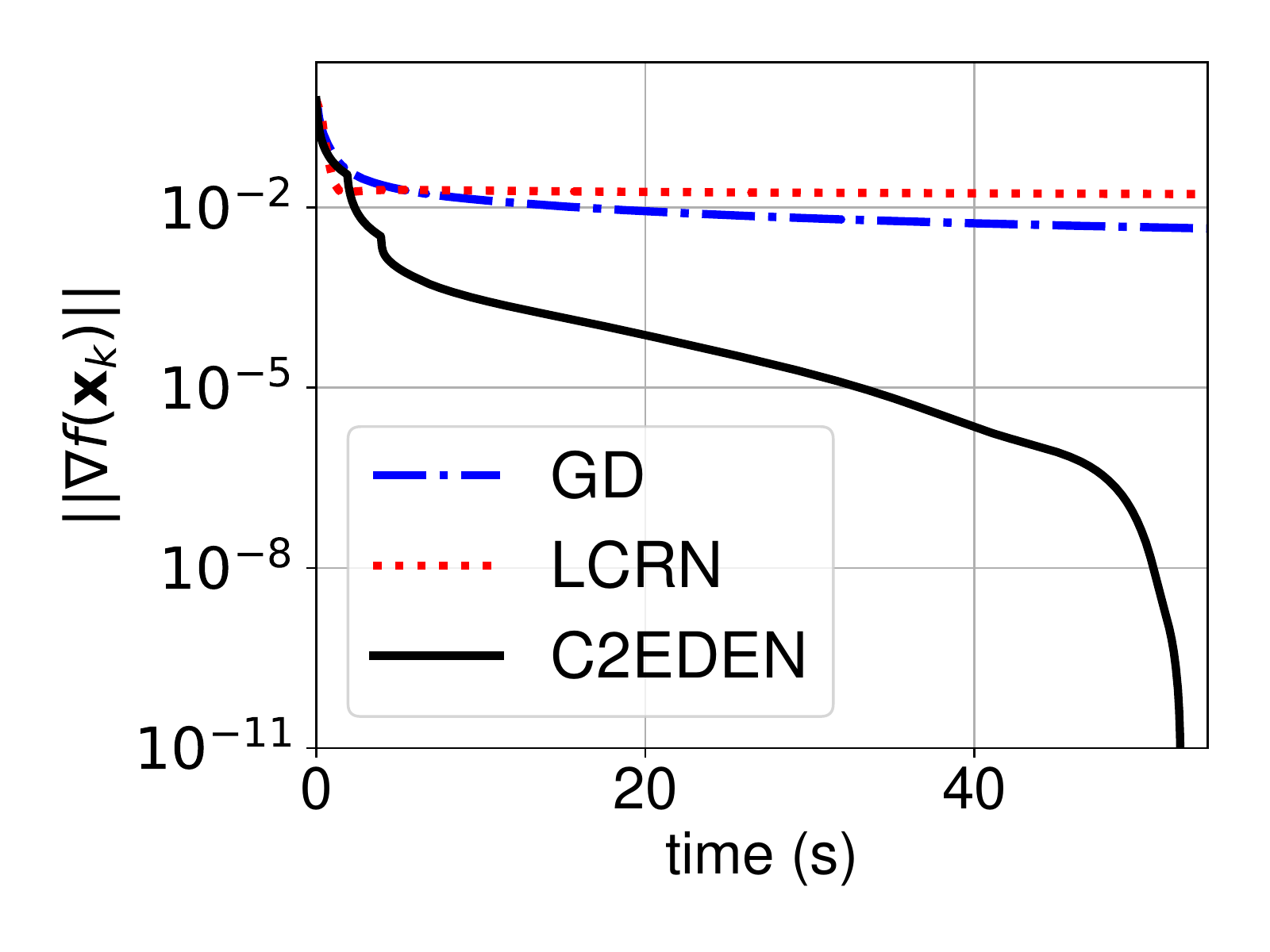}  \\
     (a)   $ \#$communication vs. gap     &
    (b)  time (s) vs. gap   &
    (c)$ \#$communication vs. $ \Vert \nabla f(x_k) \Vert$     &
    (d)  time (s) vs. $ \Vert \nabla f(x_k) \Vert$
    \end{tabular}
    \caption{The results of the model of nonconvex regularized logistic regression on w8a ($n$=16).}
        \label{fig:nc-w8a16}
\end{figure*}

\begin{figure*}[ht]
    \centering
    \begin{tabular}{cccc}
    \includegraphics[scale=0.22]{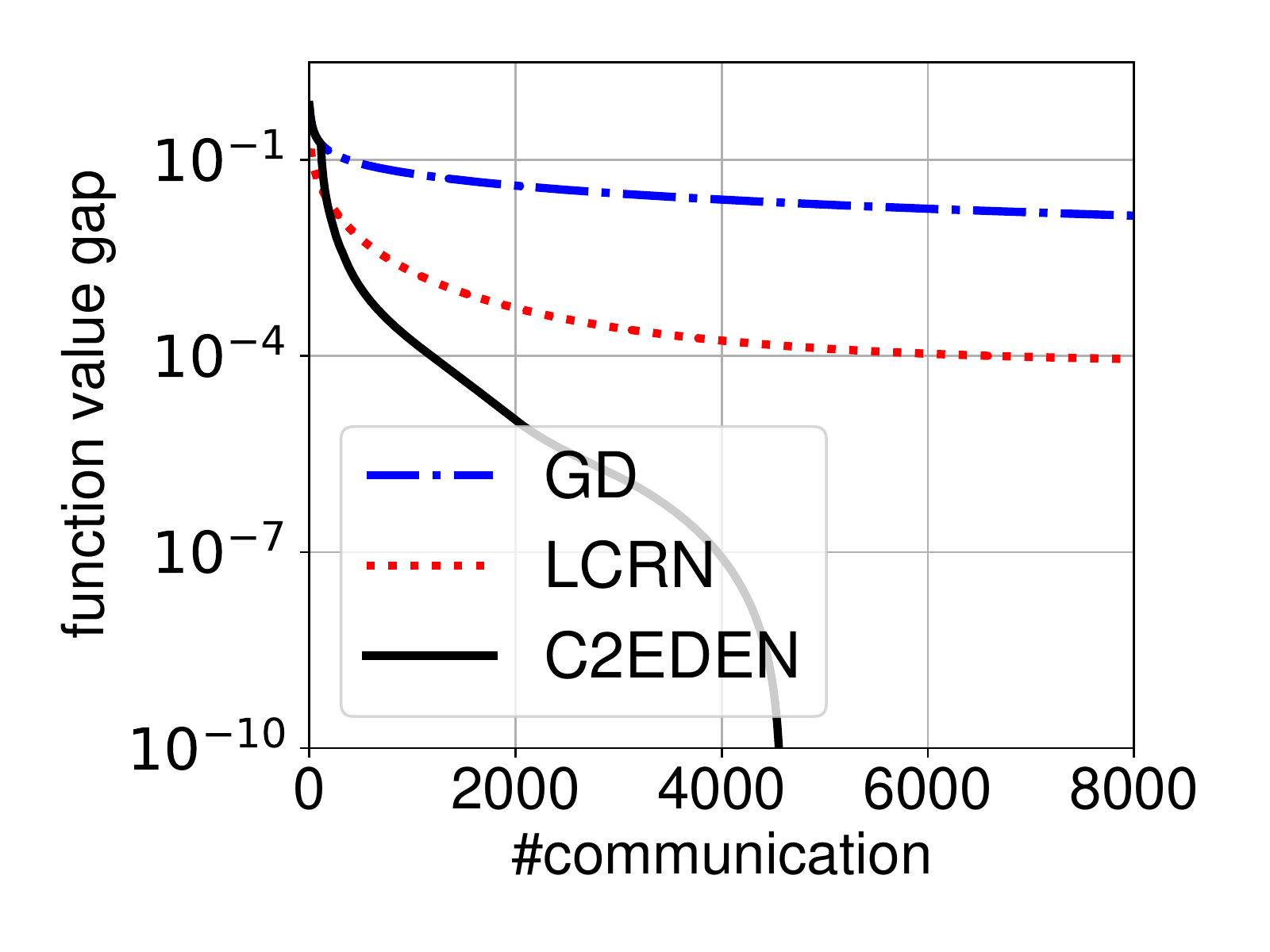}     &  \includegraphics[scale=0.22]{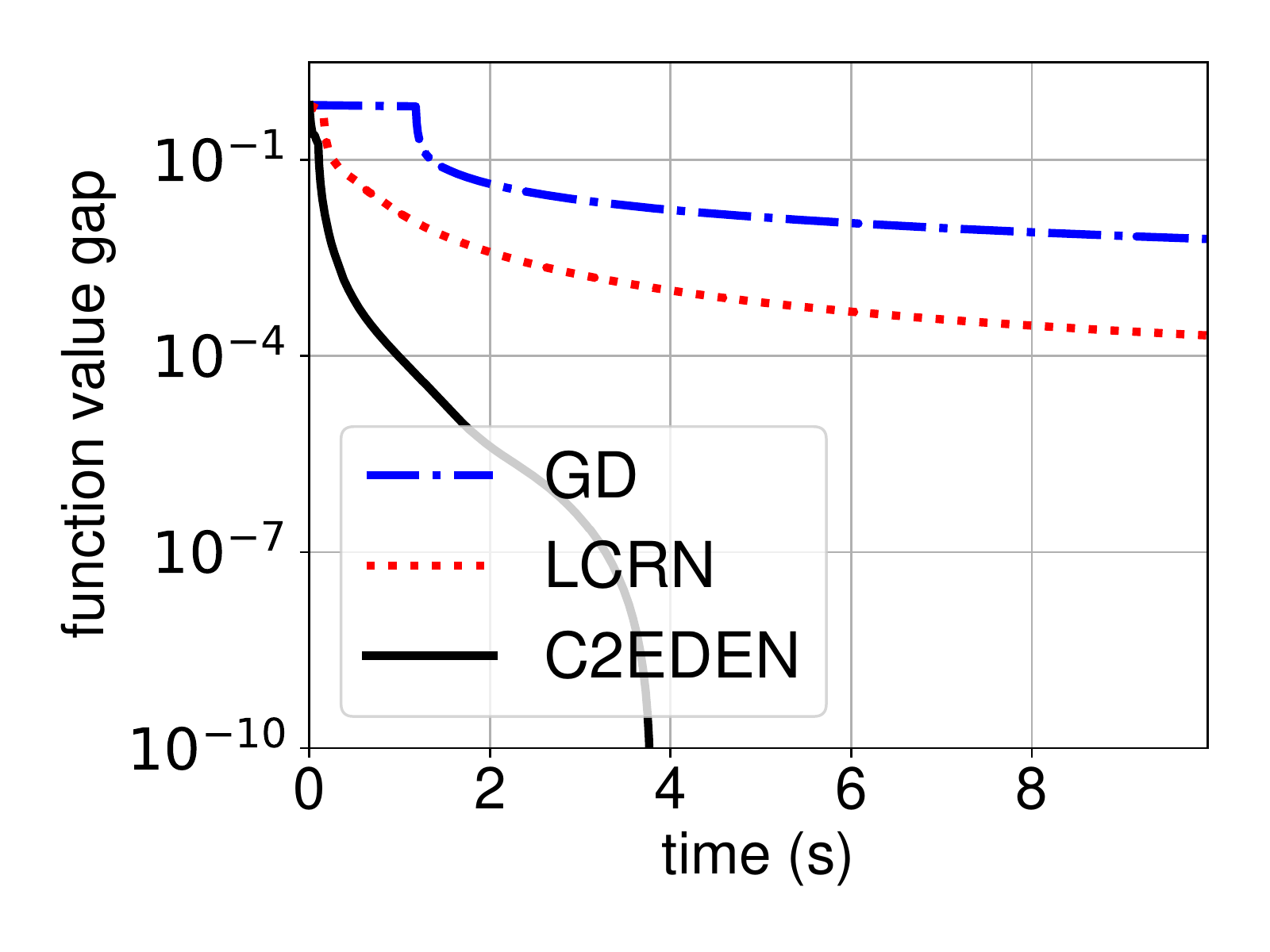} &
    \includegraphics[scale=0.22]{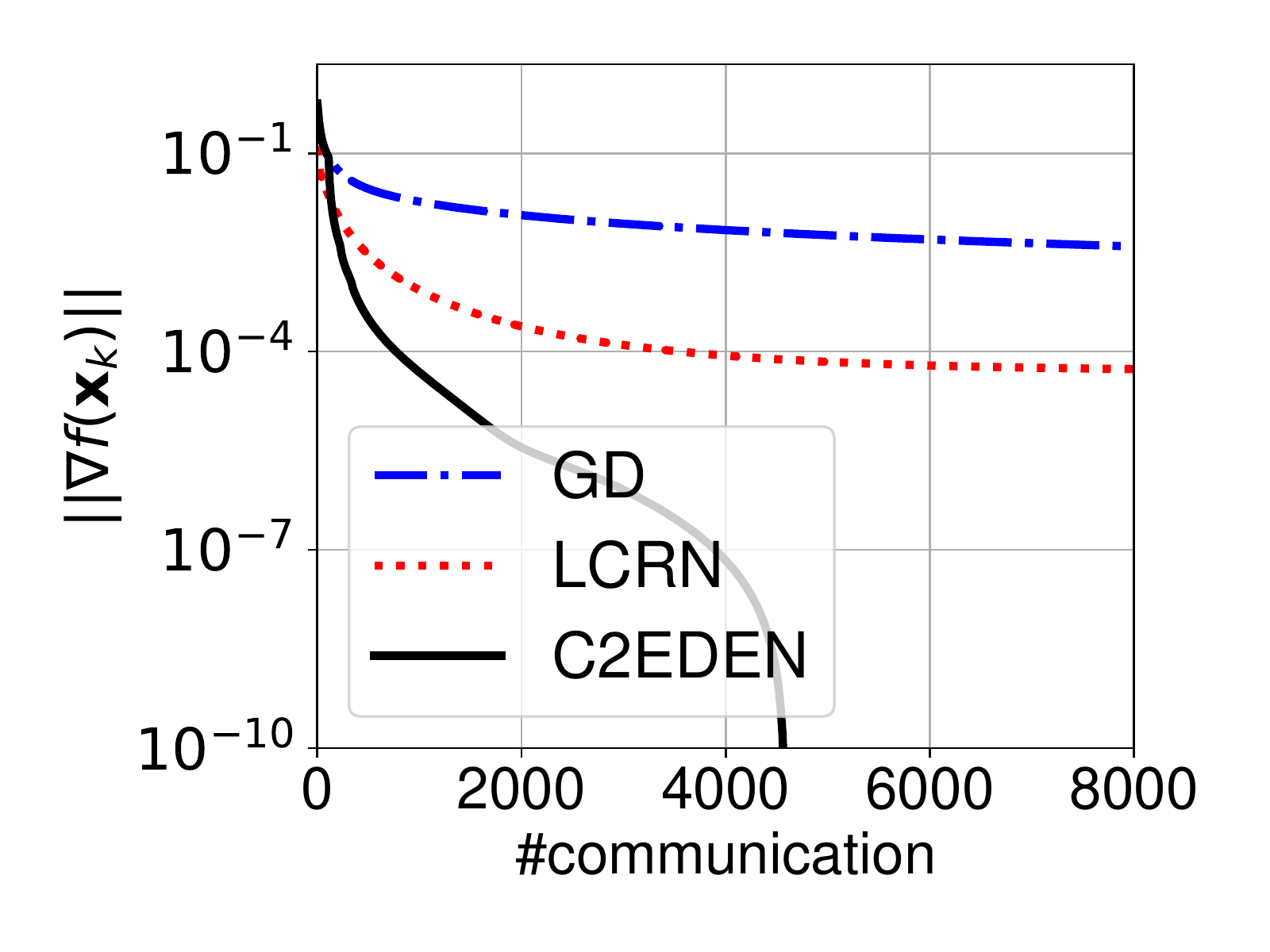}  &
    \includegraphics[scale=0.22]{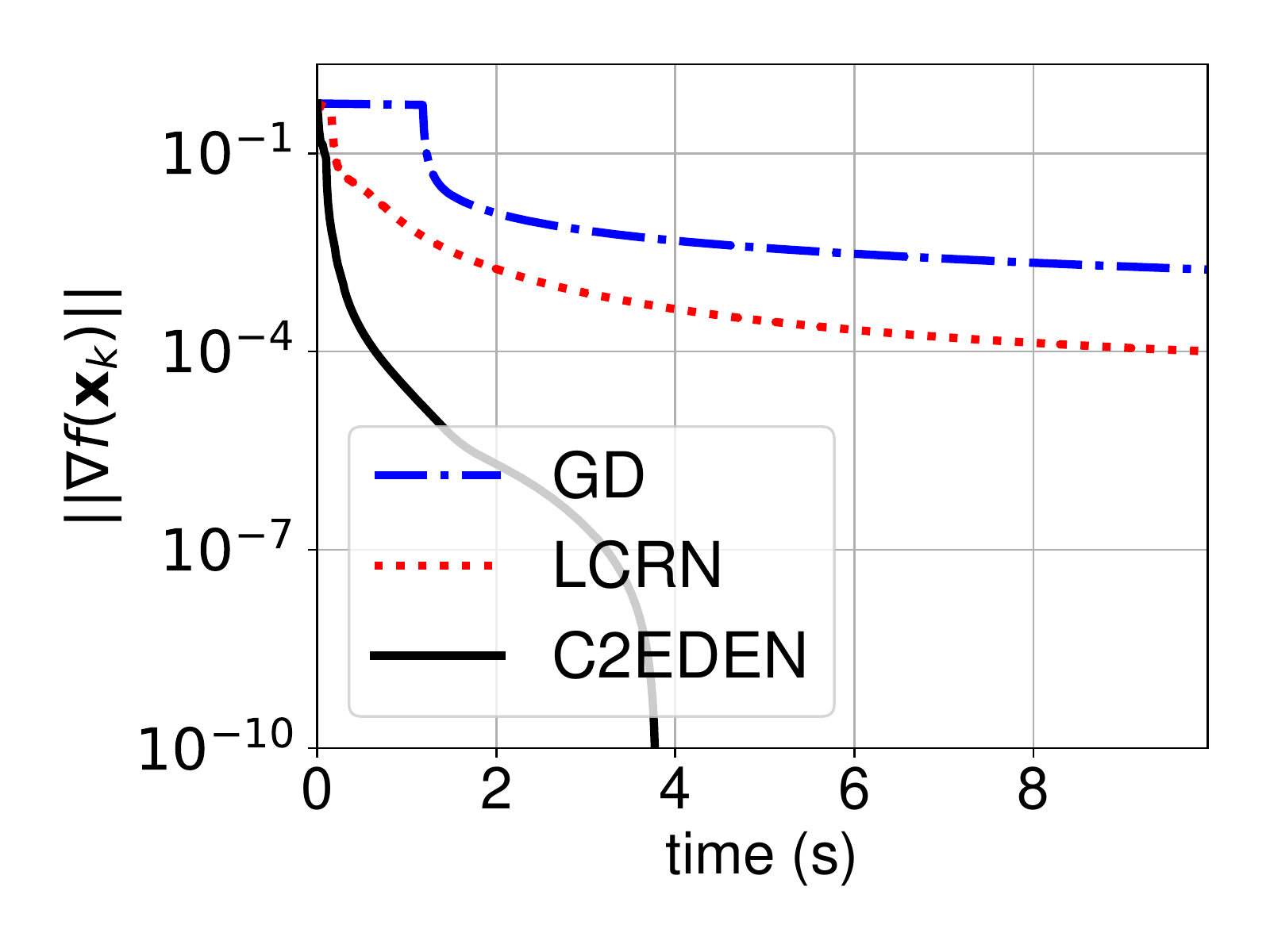}  \\
     (a)  $ \#$communication vs. gap     &
    (b)  time (s) vs. gap      &
    (c) $ \#$communication vs. $ \Vert \nabla f(x_k) \Vert$    &
    (d)   time (s) vs. $ \Vert \nabla f(x_k) \Vert$
    \end{tabular}
    \caption{The results of the model of nonconvex regularized logistic regression on mushrooms ($n$=16).}
        \label{fig:nc-mushrooms16}
\end{figure*}

\begin{figure*}[ht]
    \centering
    \begin{tabular}{cccc}
    \includegraphics[scale=0.22]{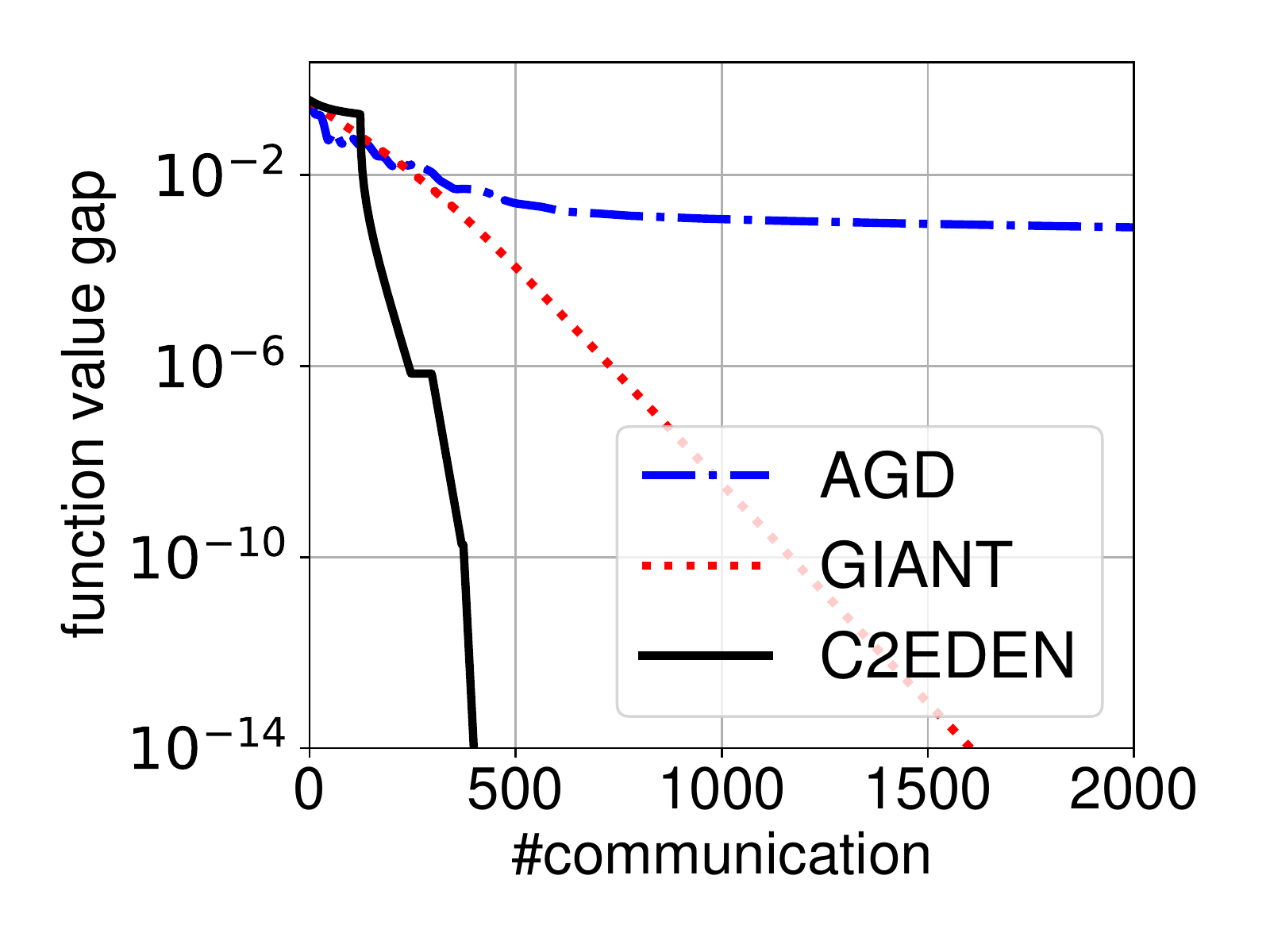}     &  \includegraphics[scale=0.22]{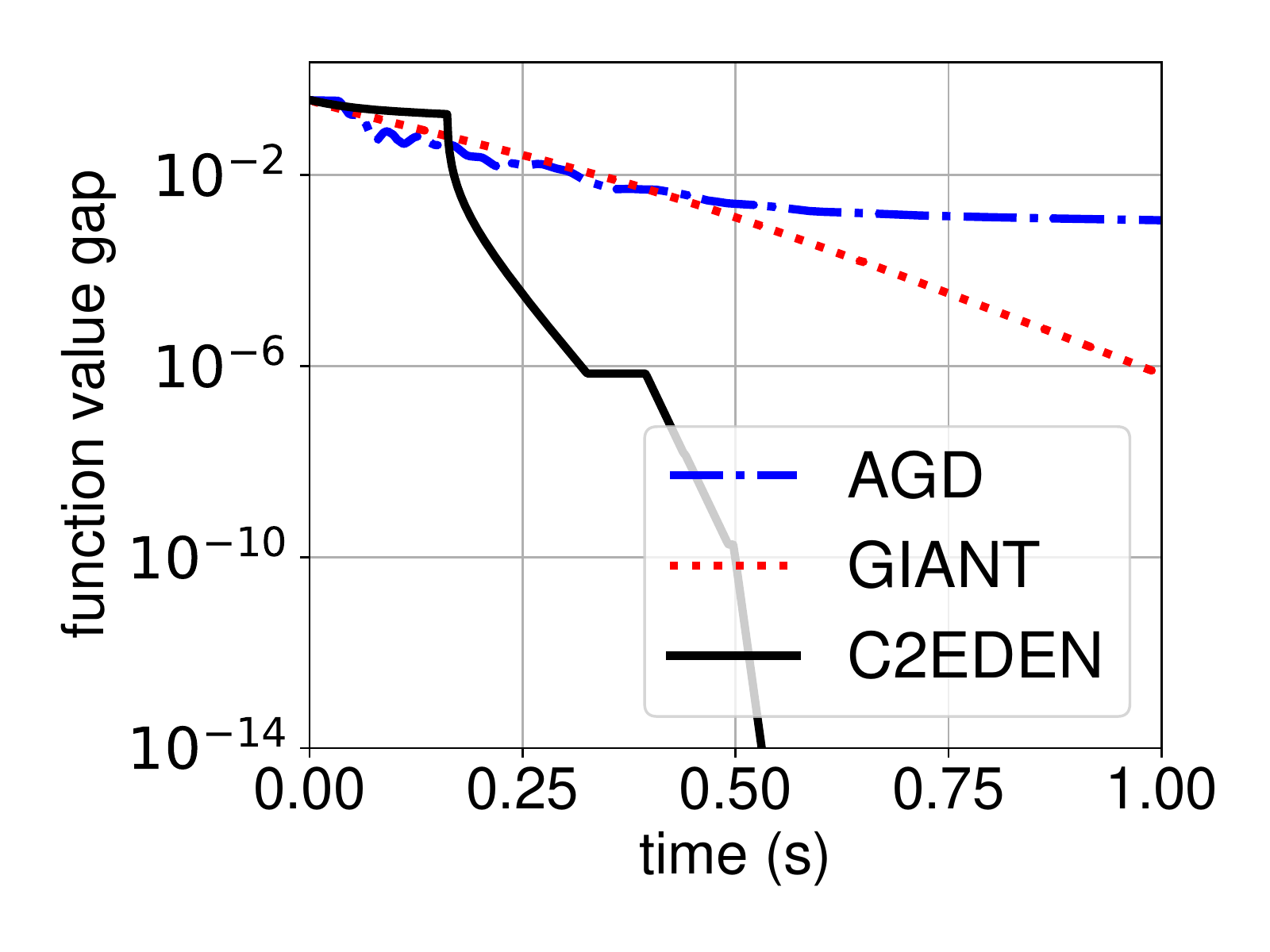} &
    \includegraphics[scale=0.22]{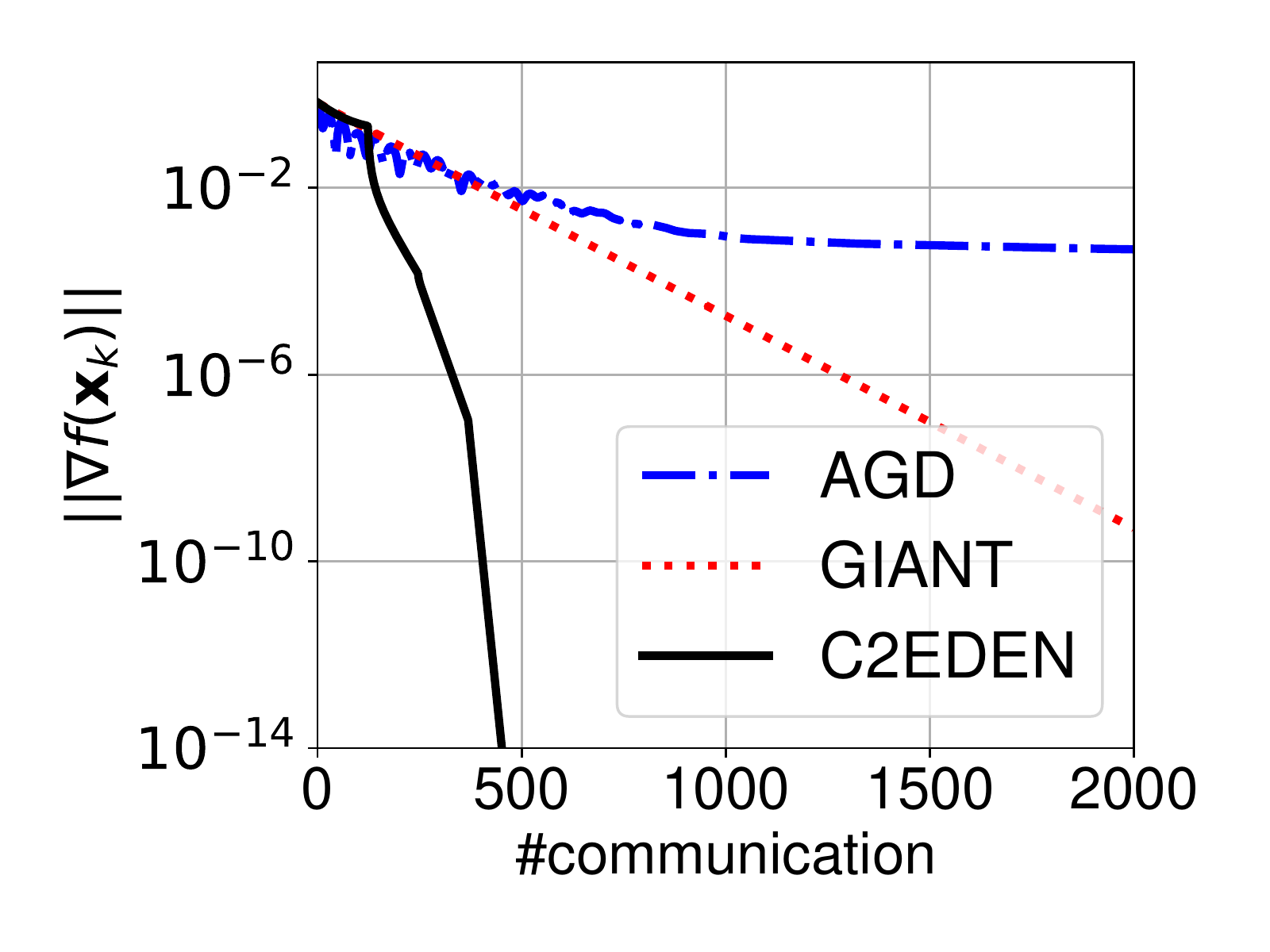} &
    \includegraphics[scale=0.22]{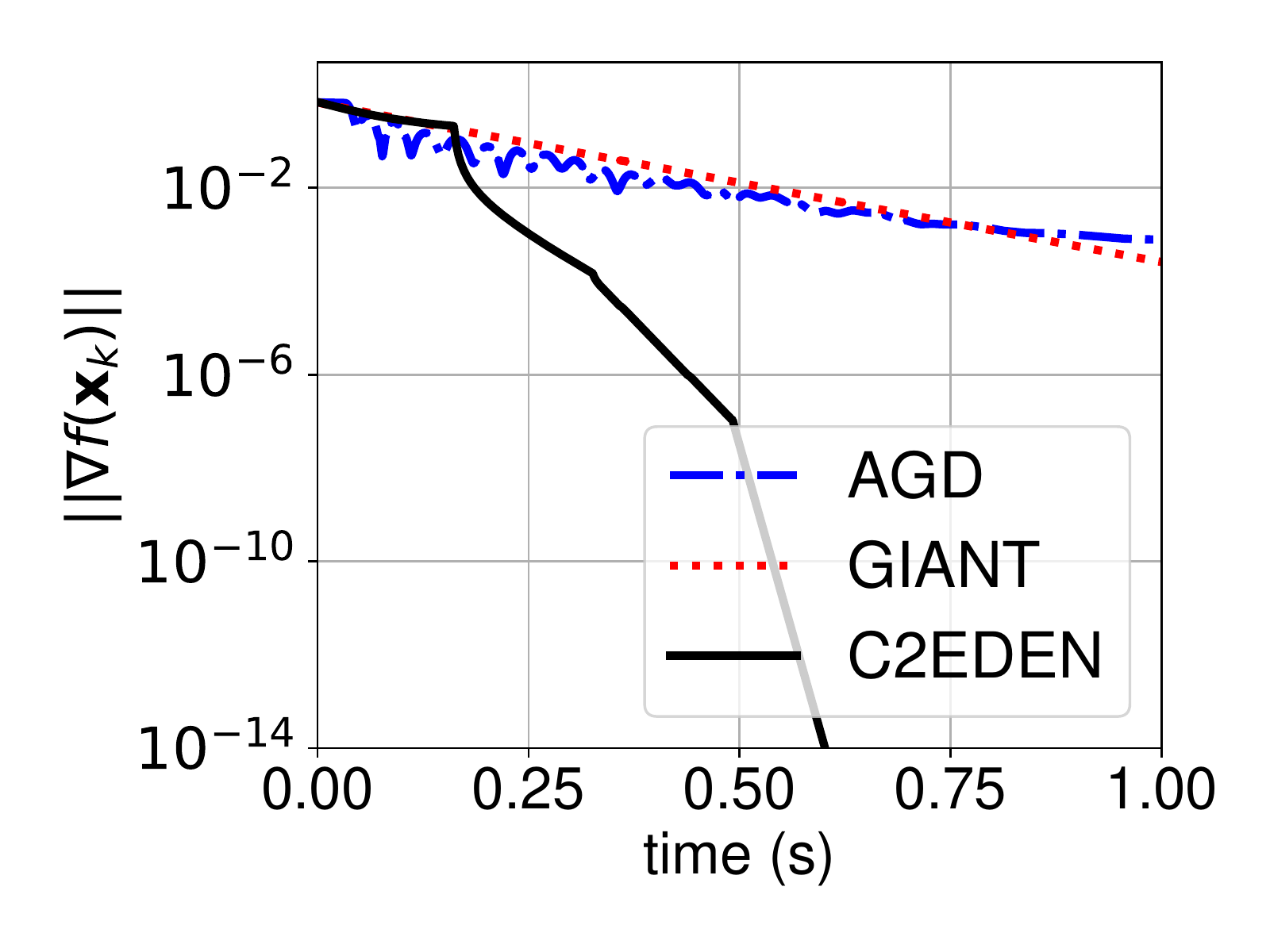}  \\
    (a) $\#$communication vs. gap     &
    (b)   time (s) vs. gap     &
    (c) $\#$communication vs. $ \Vert \nabla f(x_k) \Vert$     &
    (d)   time (s) vs. $ \Vert \nabla f(x_k) \Vert$
    \end{tabular}
    \caption{The results of the model of $\ell_2$-regularized logistic regression on a9a  ($n$=32).}
    \label{fig:sc-a9a}
\end{figure*}

\begin{figure*}[ht]
    \centering
    \begin{tabular}{cccc}
    \includegraphics[scale=0.22]{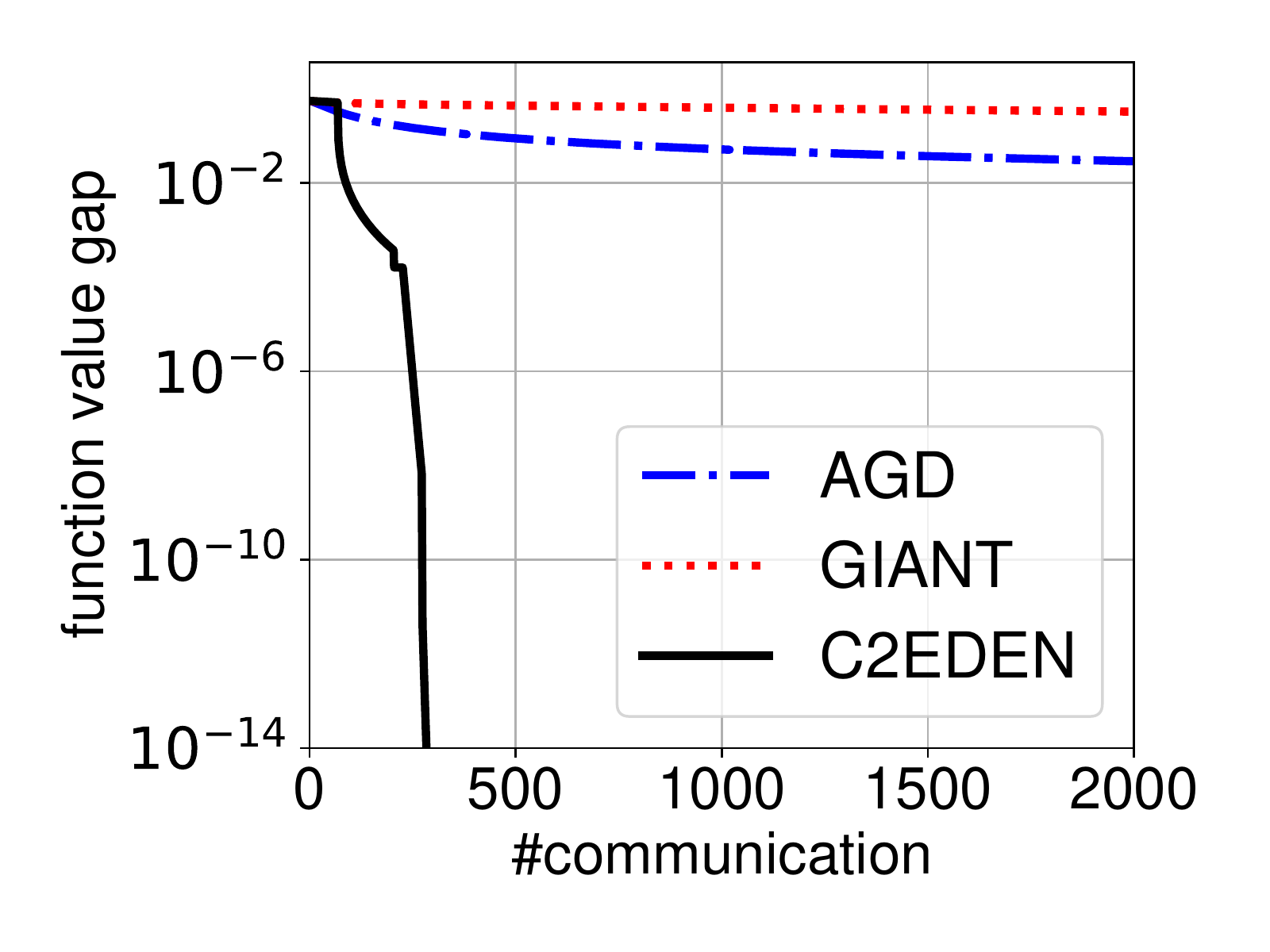}     &  \includegraphics[scale=0.22]{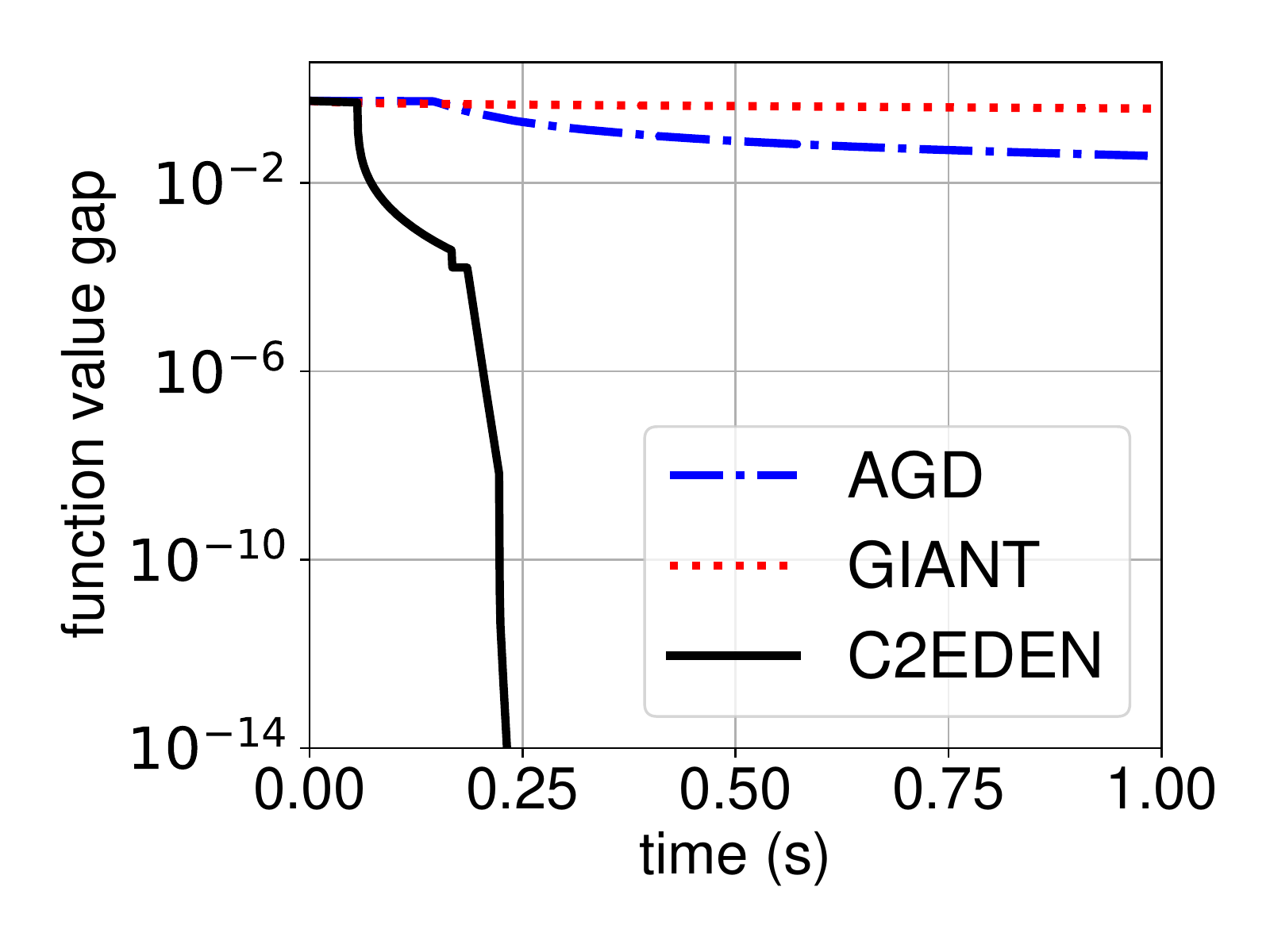} &
    \includegraphics[scale=0.22]{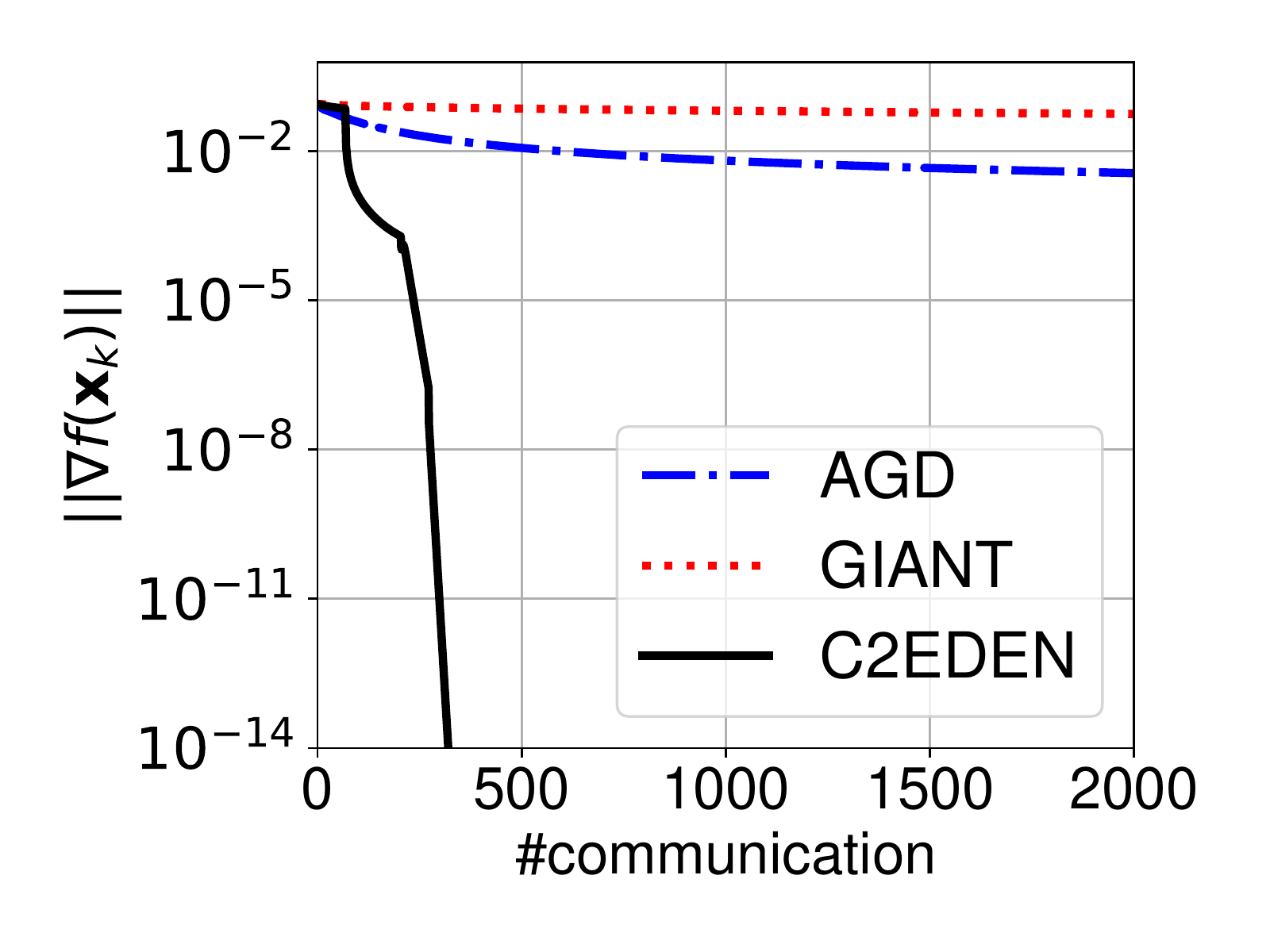} &
    \includegraphics[scale=0.22]{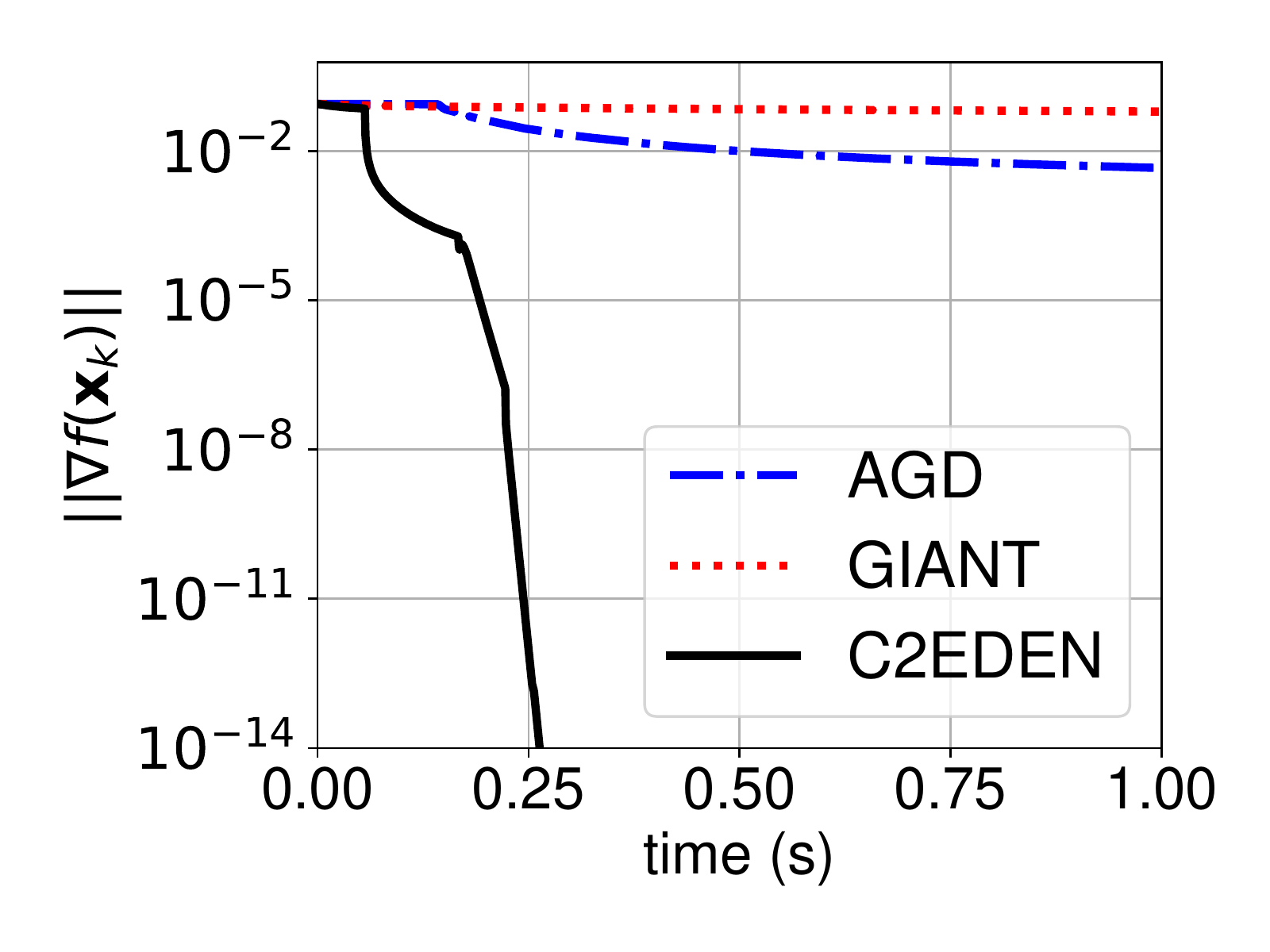}  \\
    (a)   $ \#$communication vs. gap    &
    (b)   time (s) vs. gap      &
    (c) $ \#$communication vs. $ \Vert \nabla f(x_k) \Vert$     &
    (d)   time (s) vs. $ \Vert \nabla f(x_k) \Vert$
    \end{tabular}
    \caption{The results of the model of $\ell_2$-regularized logistic regression on a9a ($n$=32).  }
    \label{fig:sc-phishing}
\end{figure*}

\begin{figure*}[ht]
    \centering
    \begin{tabular}{cccc}
    \includegraphics[scale=0.22]{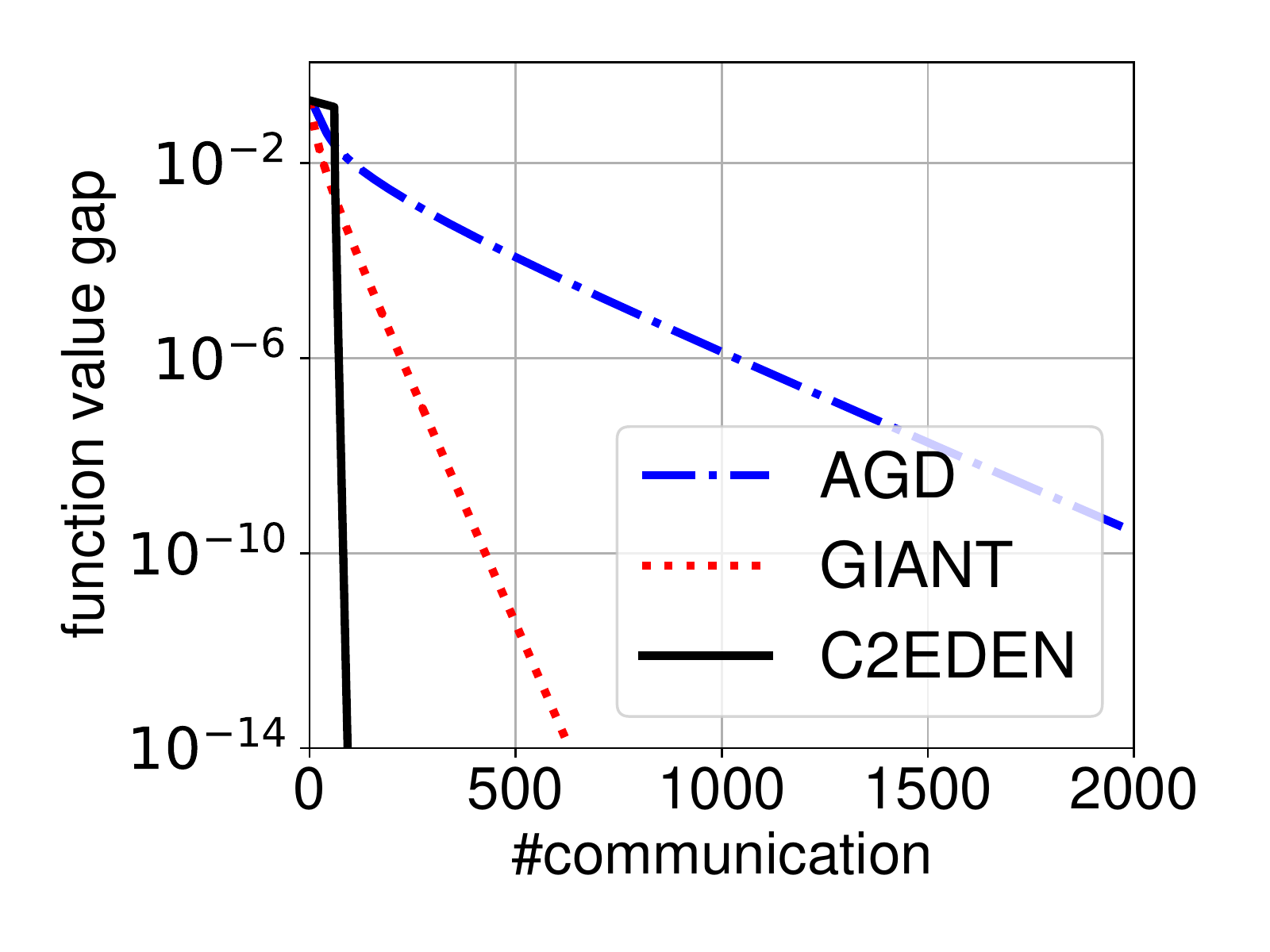}     &  \includegraphics[scale=0.22]{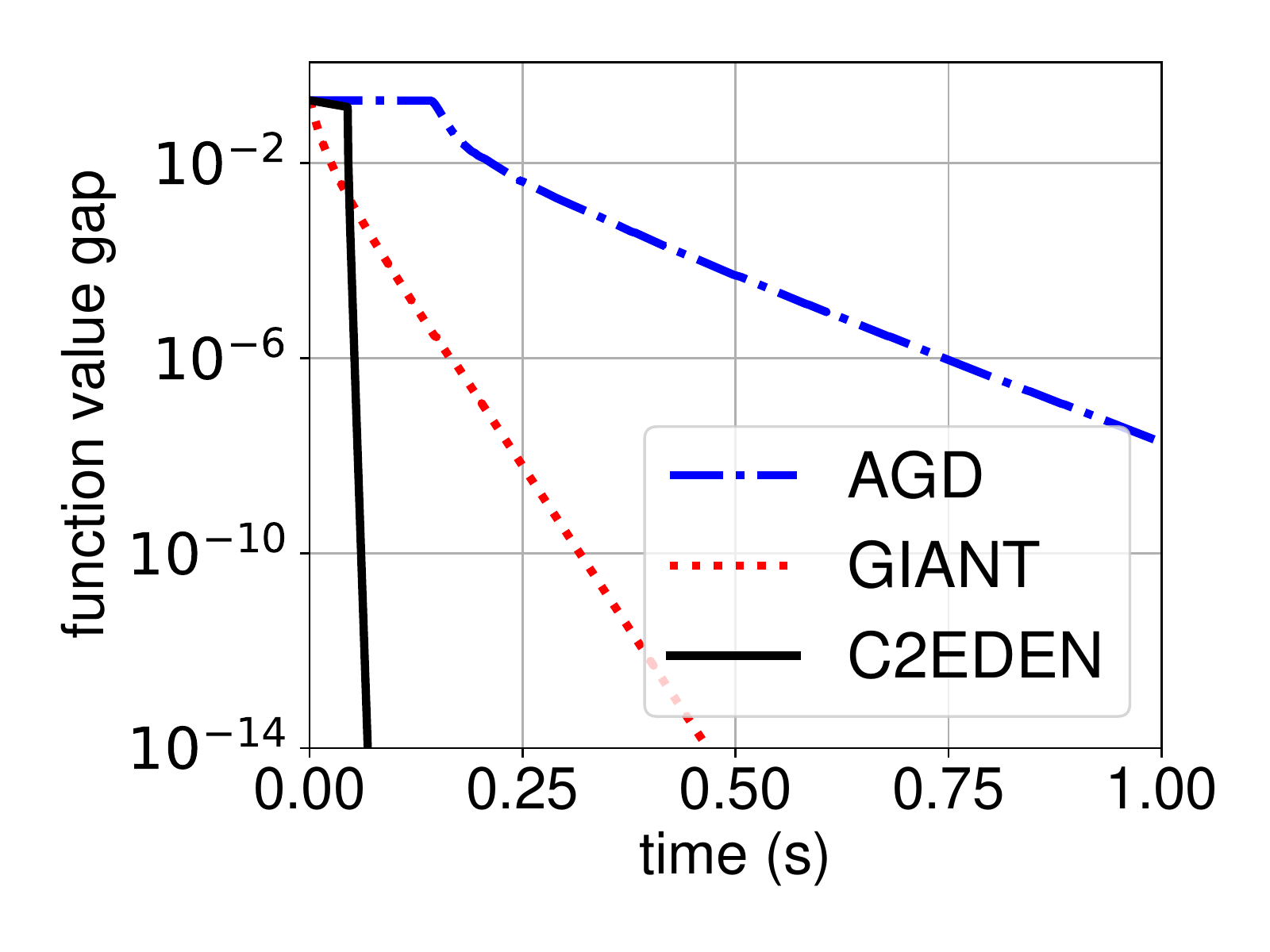} &
    \includegraphics[scale=0.22]{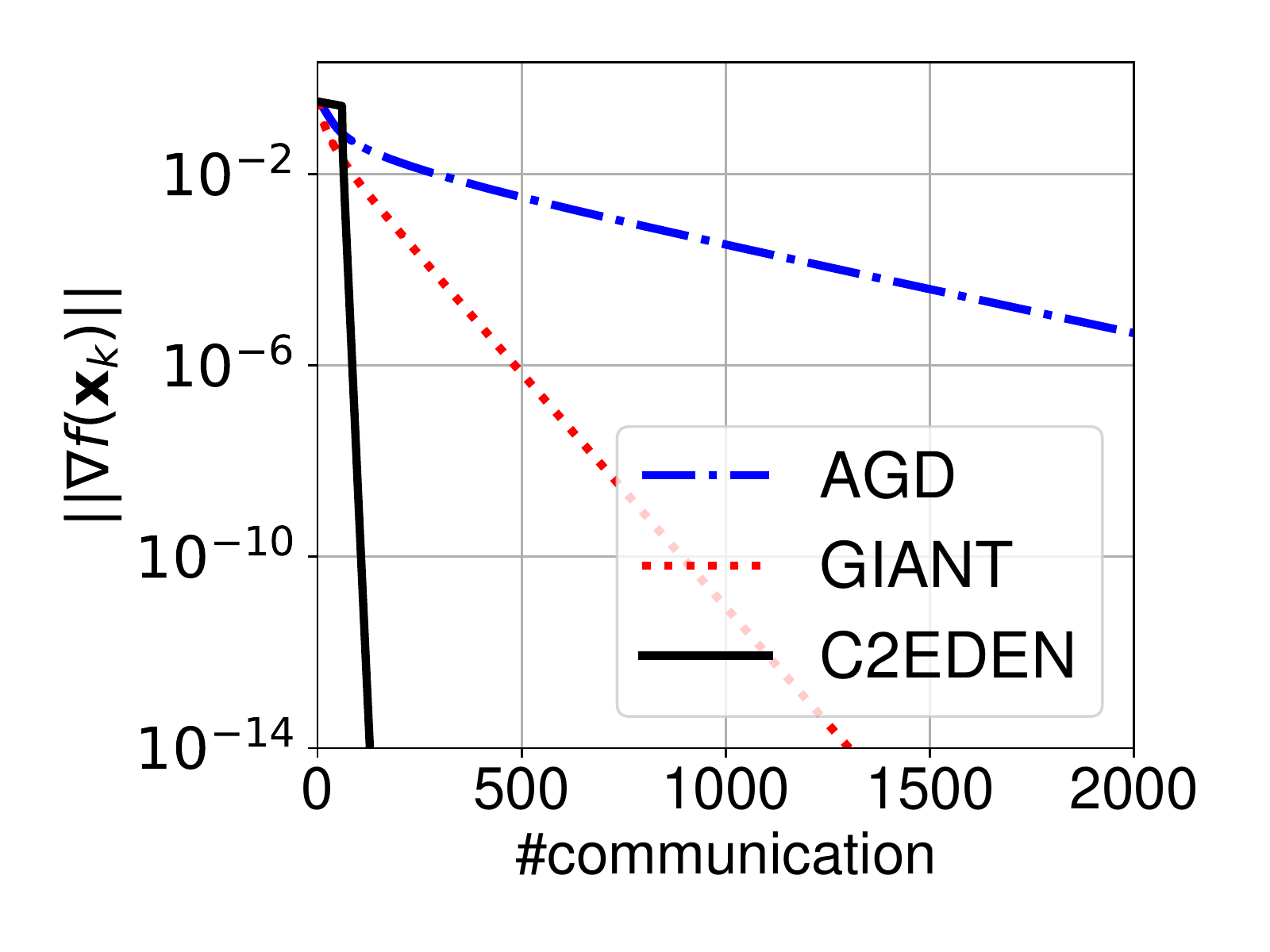} &
    \includegraphics[scale=0.22]{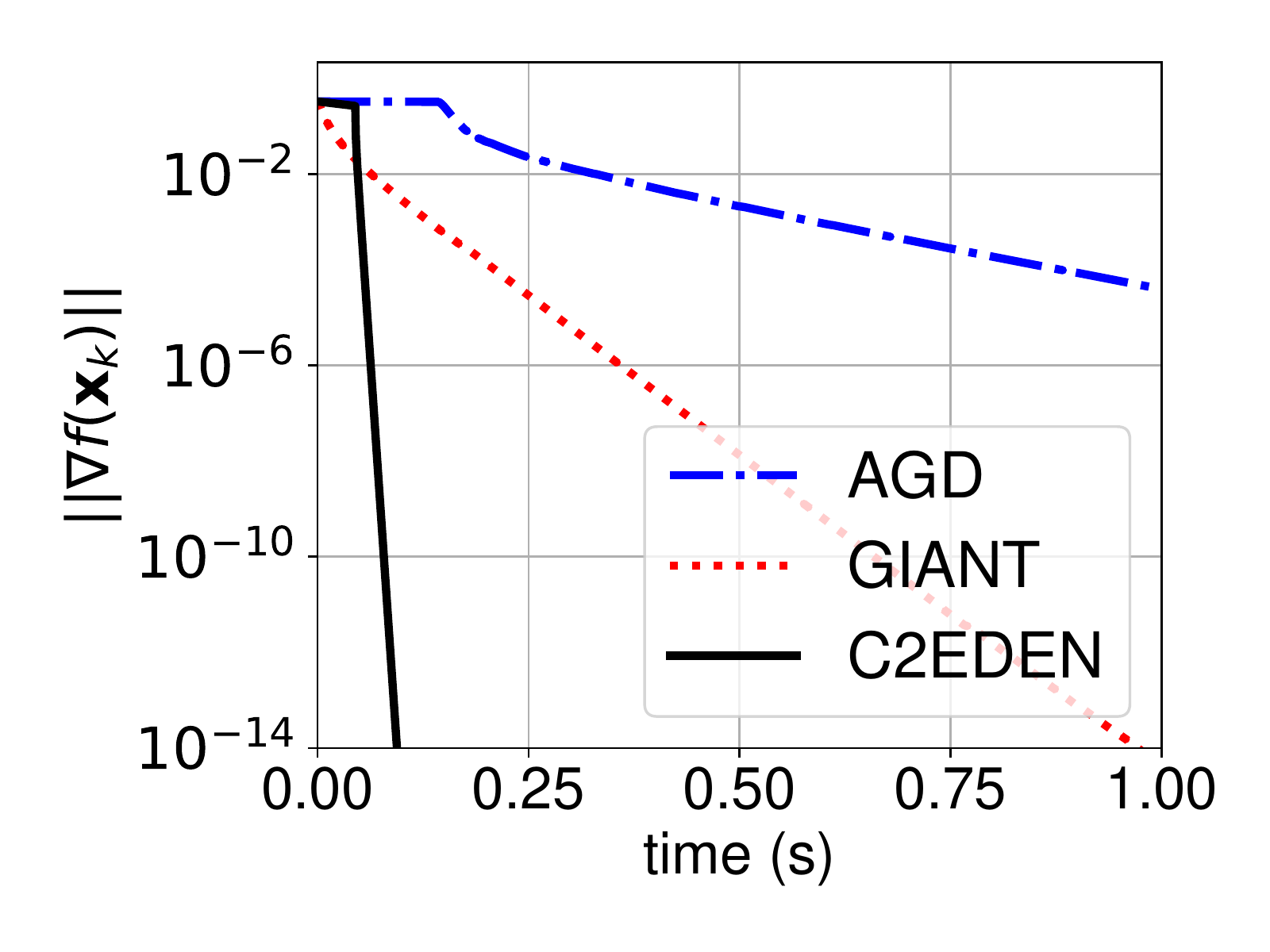}  \\
    (a)   $ \#$communication vs. gap    &
    (b)   time (s) vs. gap     &
    (c) $ \#$communication vs. $ \Vert \nabla f(x_k) \Vert$    &
    (d)   time (s) vs. $ \Vert \nabla f(x_k) \Vert$
    \end{tabular}
    \caption{The results of the model of $\ell_2$-regularized logistic regression on splice  ($n$=32).} 
    \label{fig:sc-splice}
\end{figure*}

\begin{figure*}[ht]
    \centering
    \begin{tabular}{cccc}
    \includegraphics[scale=0.22]{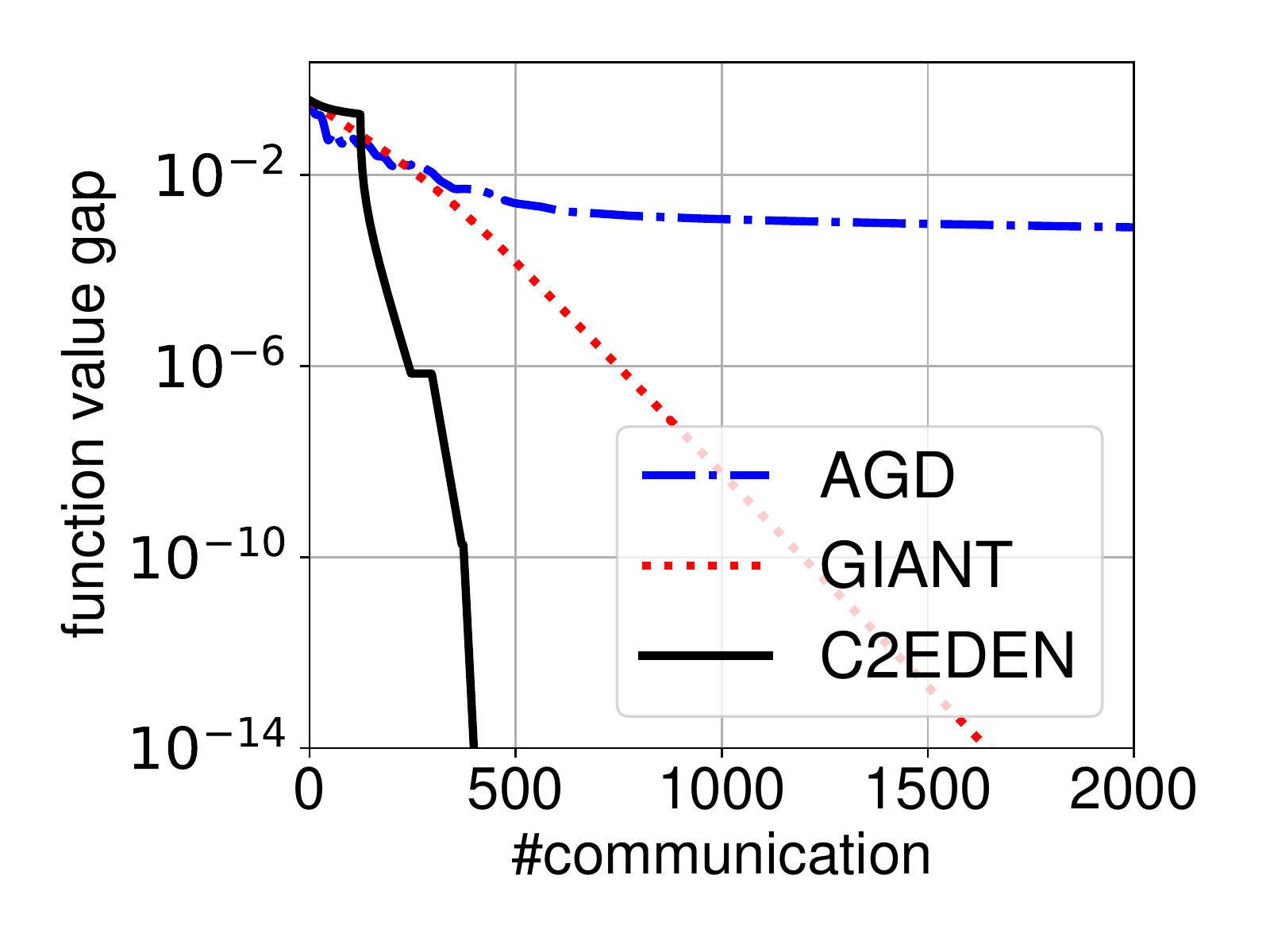}     &  \includegraphics[scale=0.22]{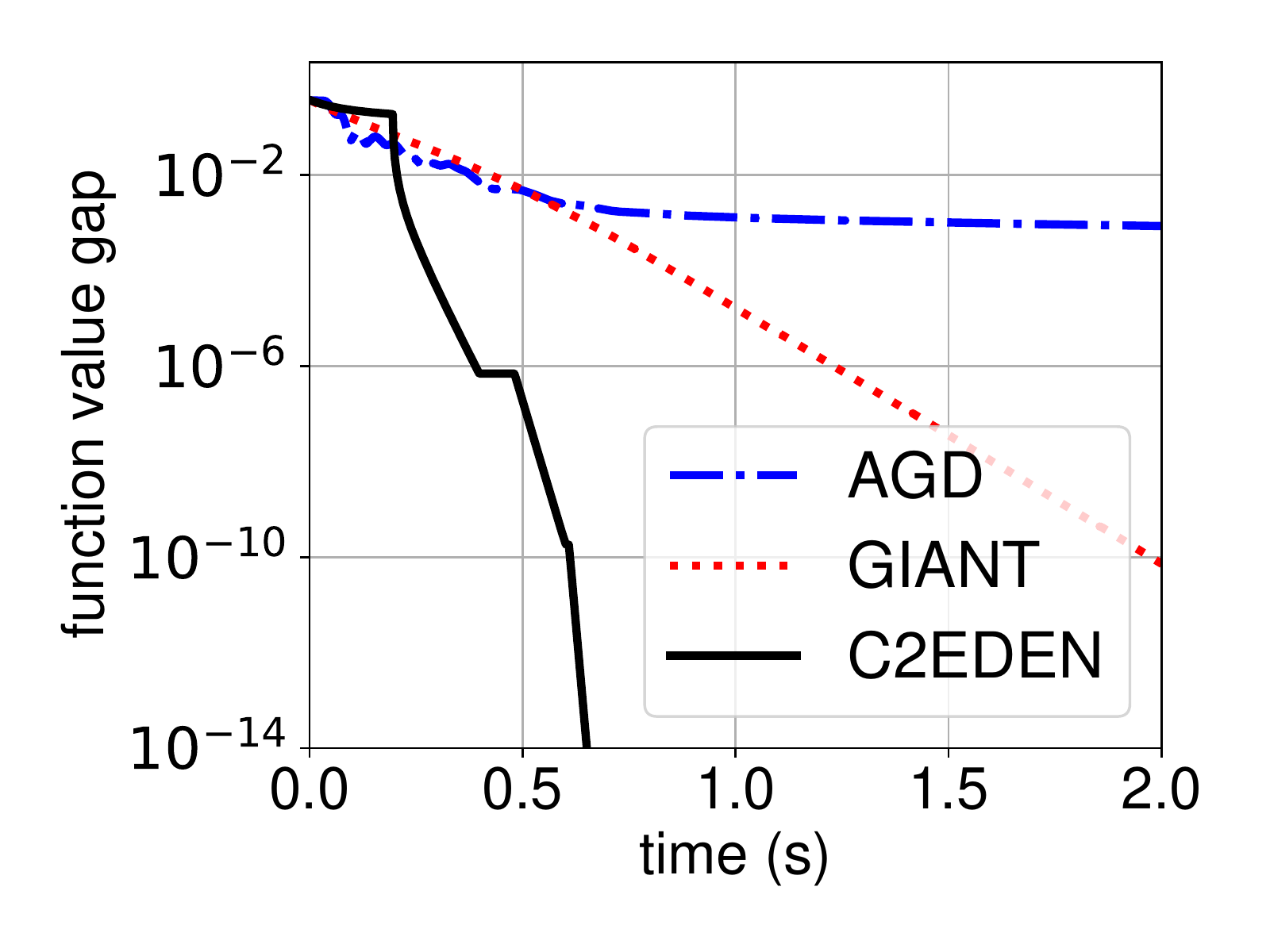} &
    \includegraphics[scale=0.22]{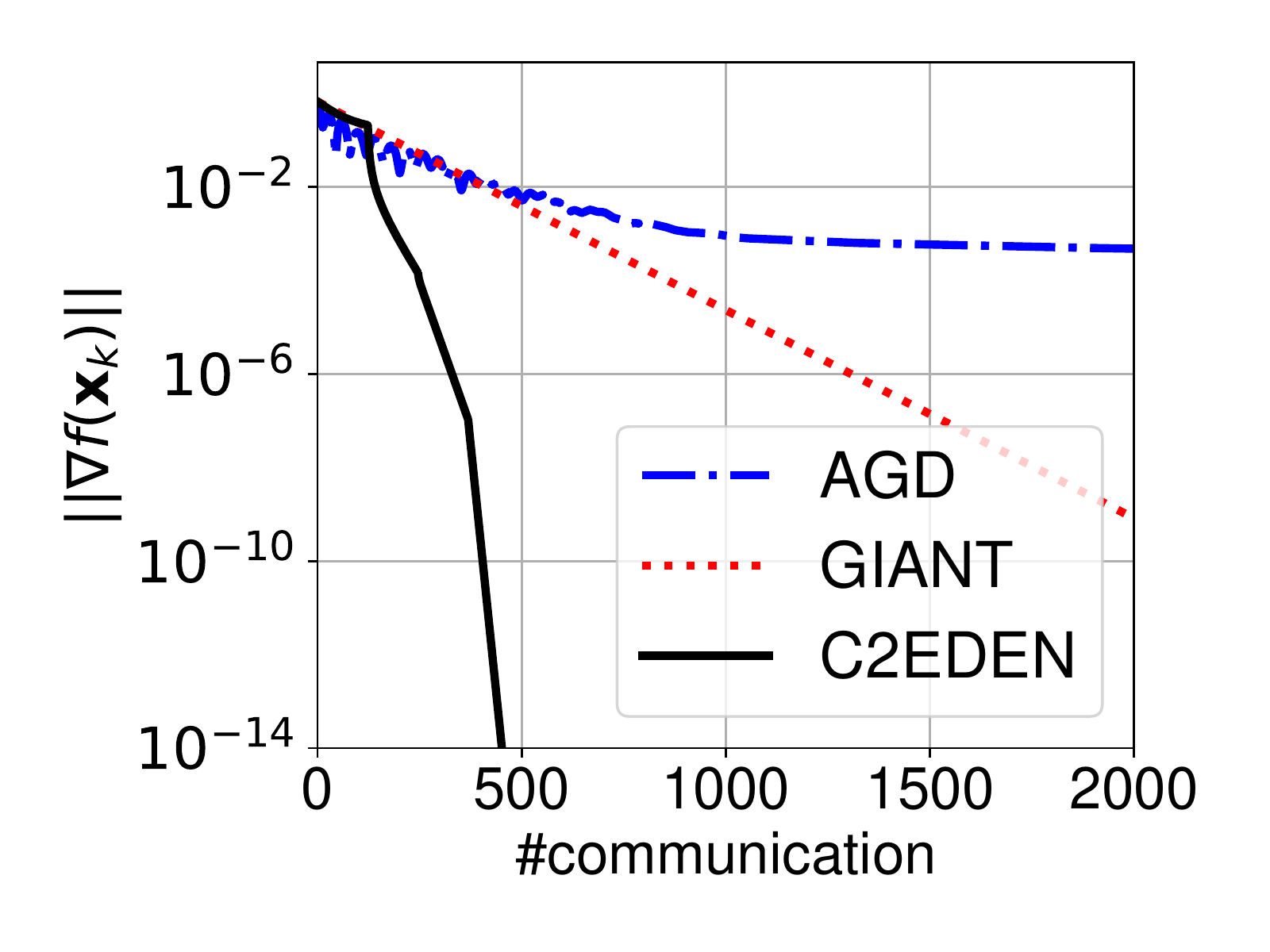} &
    \includegraphics[scale=0.22]{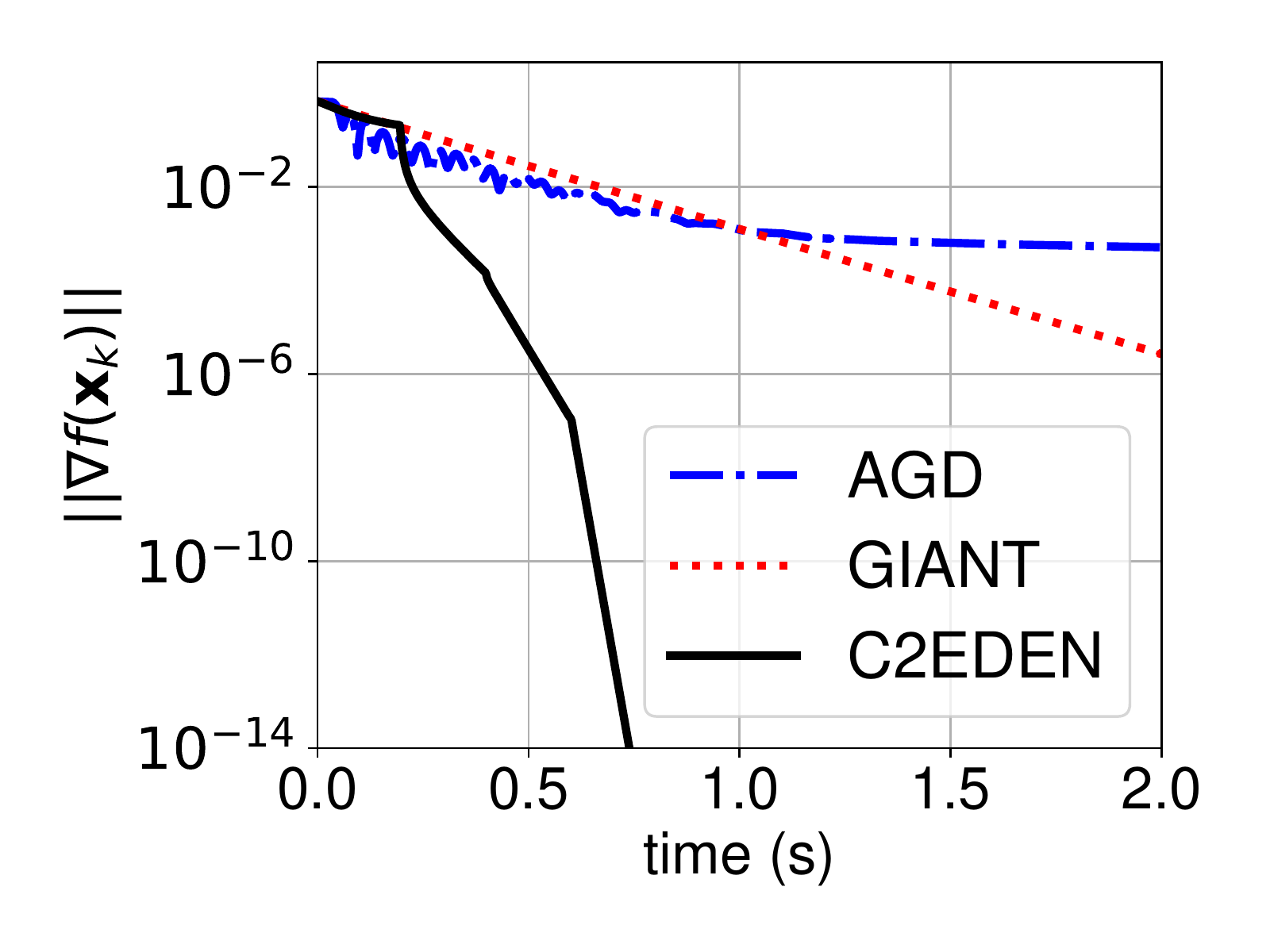}  \\
    (a)   $ \#$communication vs. gap     &
    (b)  time (s) vs. gap     &
    (c) $ \#$communication vs. $ \Vert \nabla f(x_k) \Vert$    &
    (d)   time (s) vs. $ \Vert \nabla f(x_k) \Vert$
    \end{tabular}
    \caption{The results of the model of $\ell_2$-regularized logistic regression on a9a ($n$=16).}
     \label{fig:sc-a9a16}
\end{figure*}

\begin{figure*}[ht]
    \centering
    \begin{tabular}{cccc}
    \includegraphics[scale=0.22]{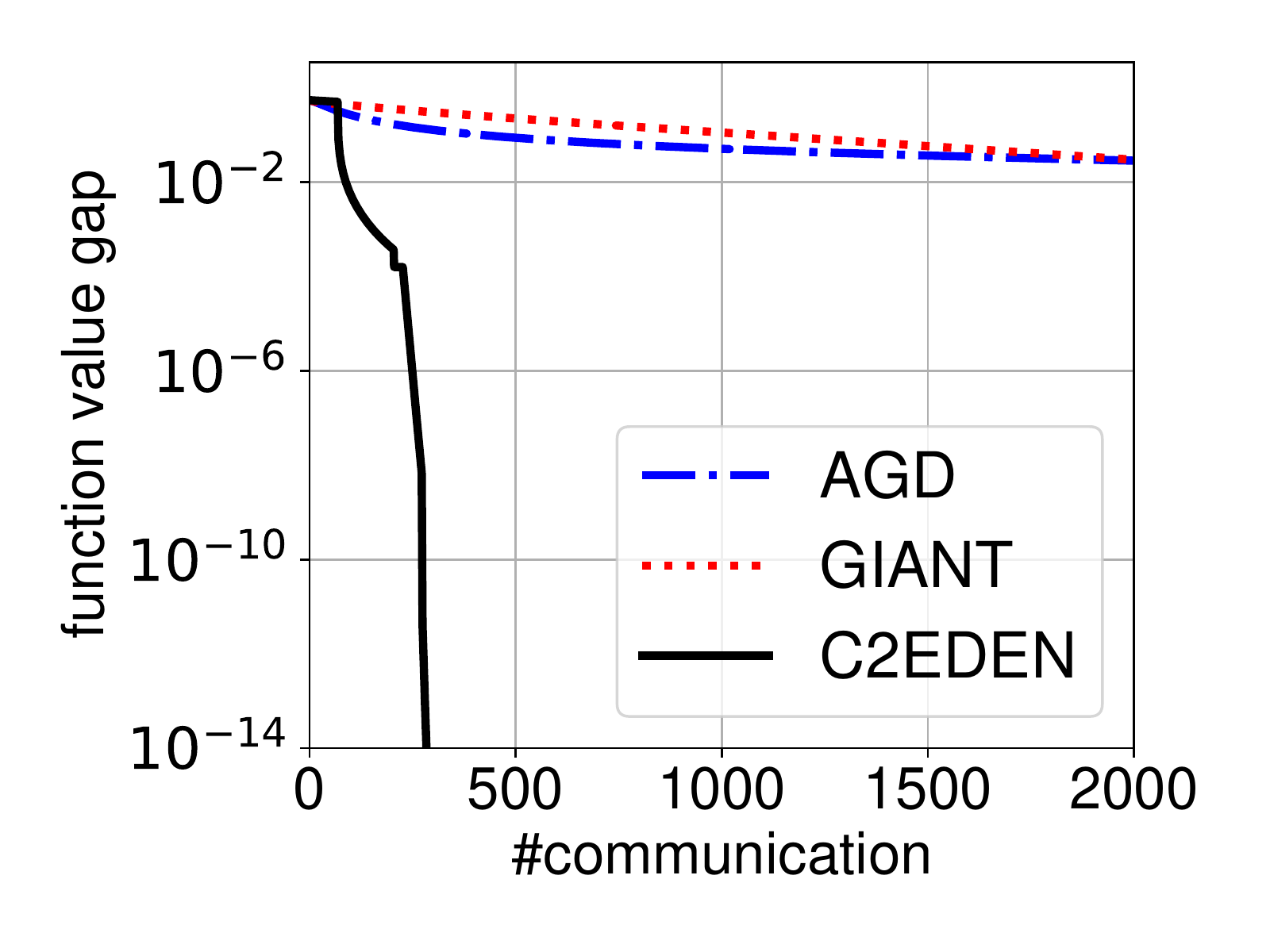}     &  \includegraphics[scale=0.22]{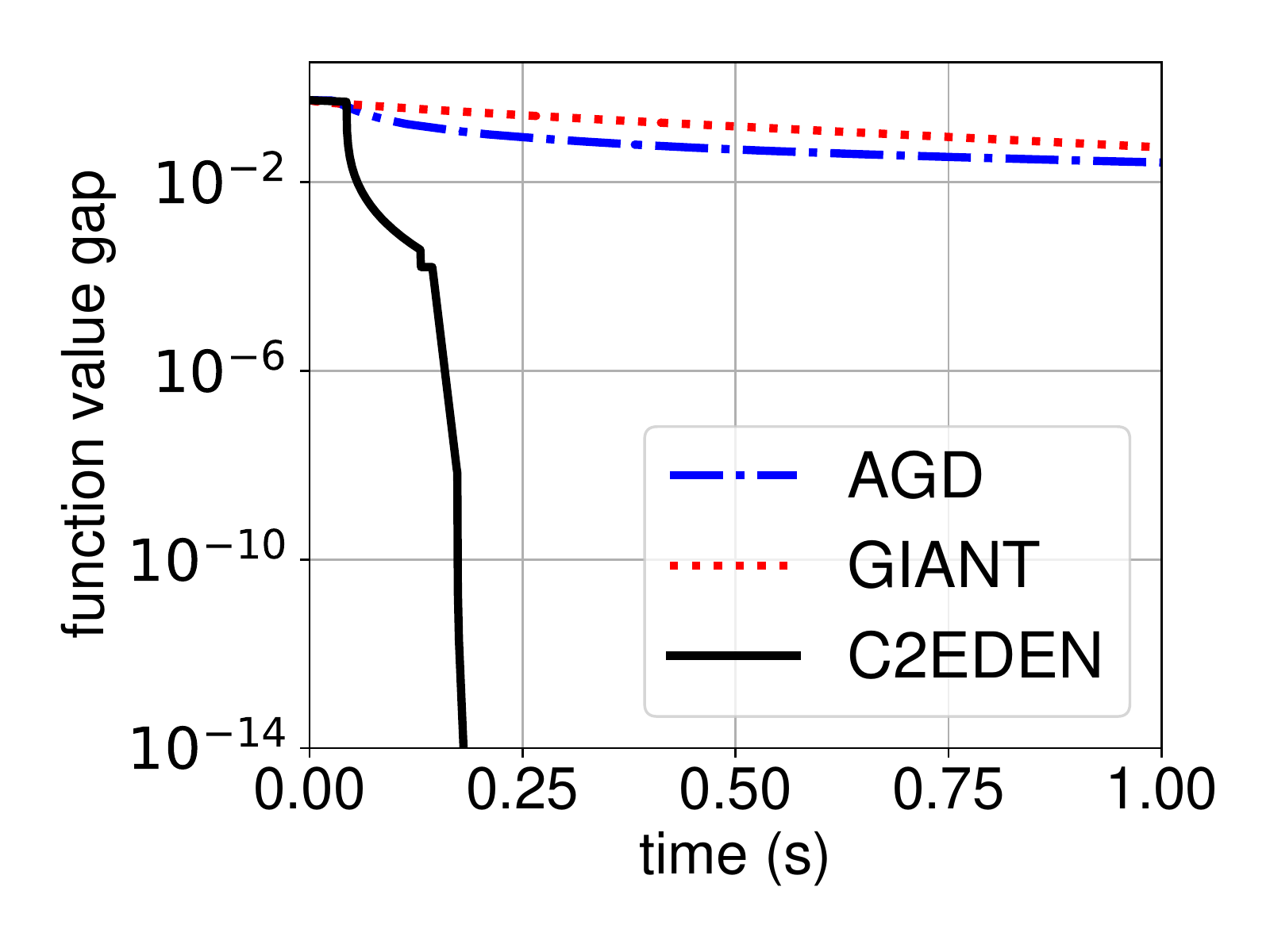} &
    \includegraphics[scale=0.22]{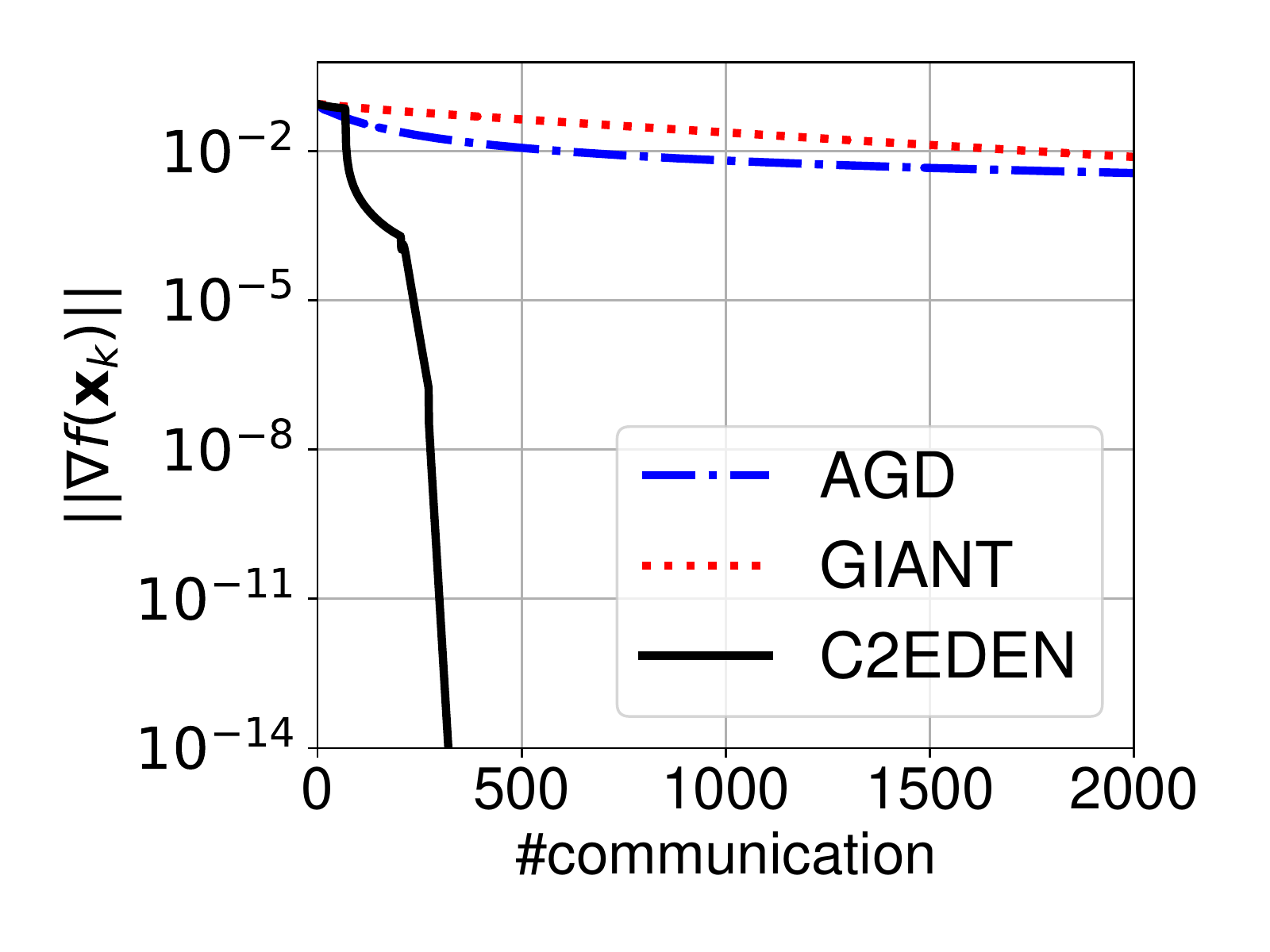} &
    \includegraphics[scale=0.22]{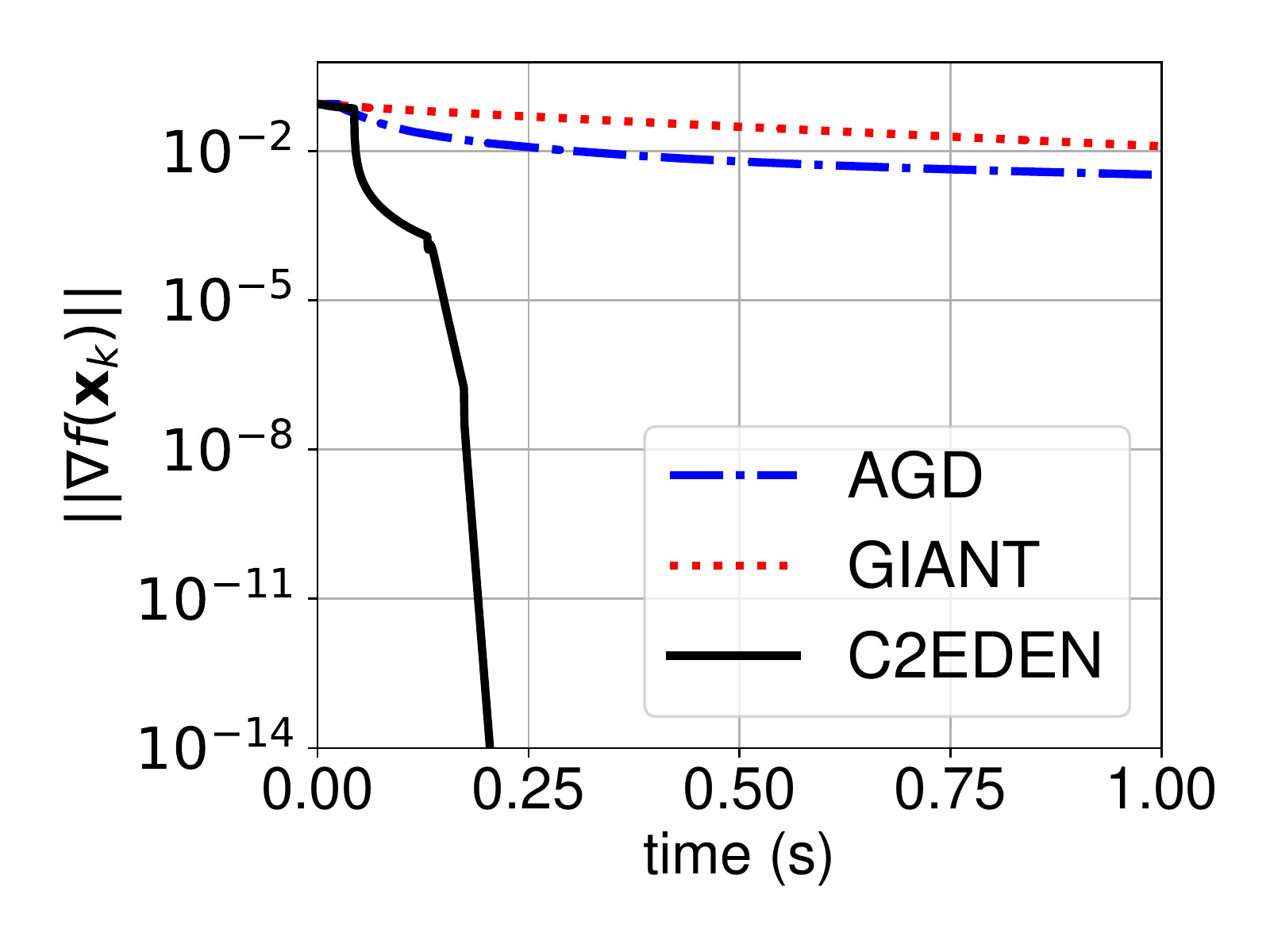}  \\
    (a)   $ \#$communication vs. gap     &
    (b)  time (s) vs. gap      &
    (c) $ \#$communication vs. $ \Vert \nabla f(x_k) \Vert$    &
    (d)   time (s) vs. $ \Vert \nabla f(x_k) \Vert$
    \end{tabular}
    \caption{The results of the model of $\ell_2$-regularized logistic regression on phishing ($n$=16).} 
     \label{fig:sc-phishing16}
\end{figure*}

\begin{figure*}[ht]
    \centering
    \begin{tabular}{cccc}
    \includegraphics[scale=0.22]{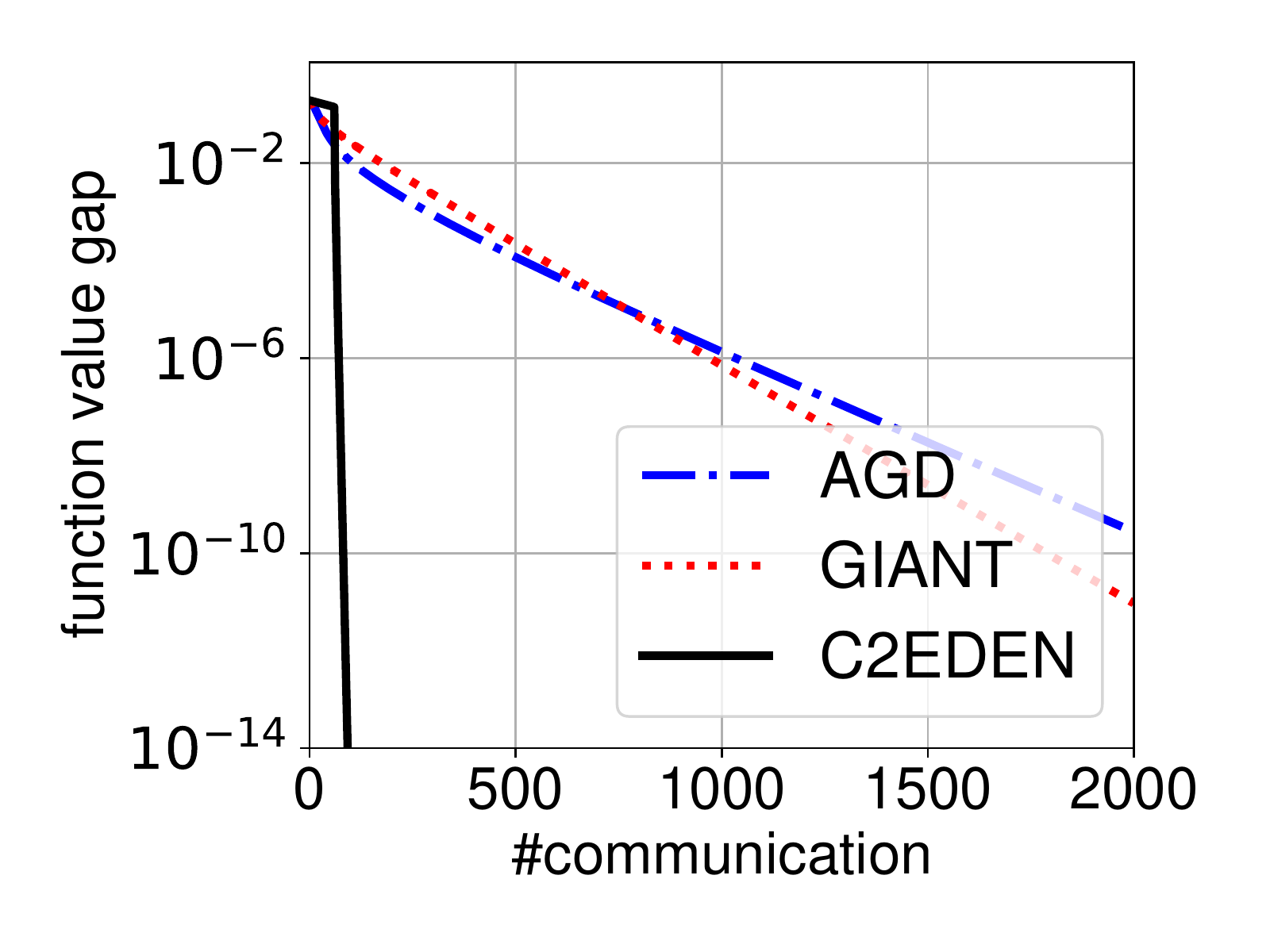}     &  \includegraphics[scale=0.22]{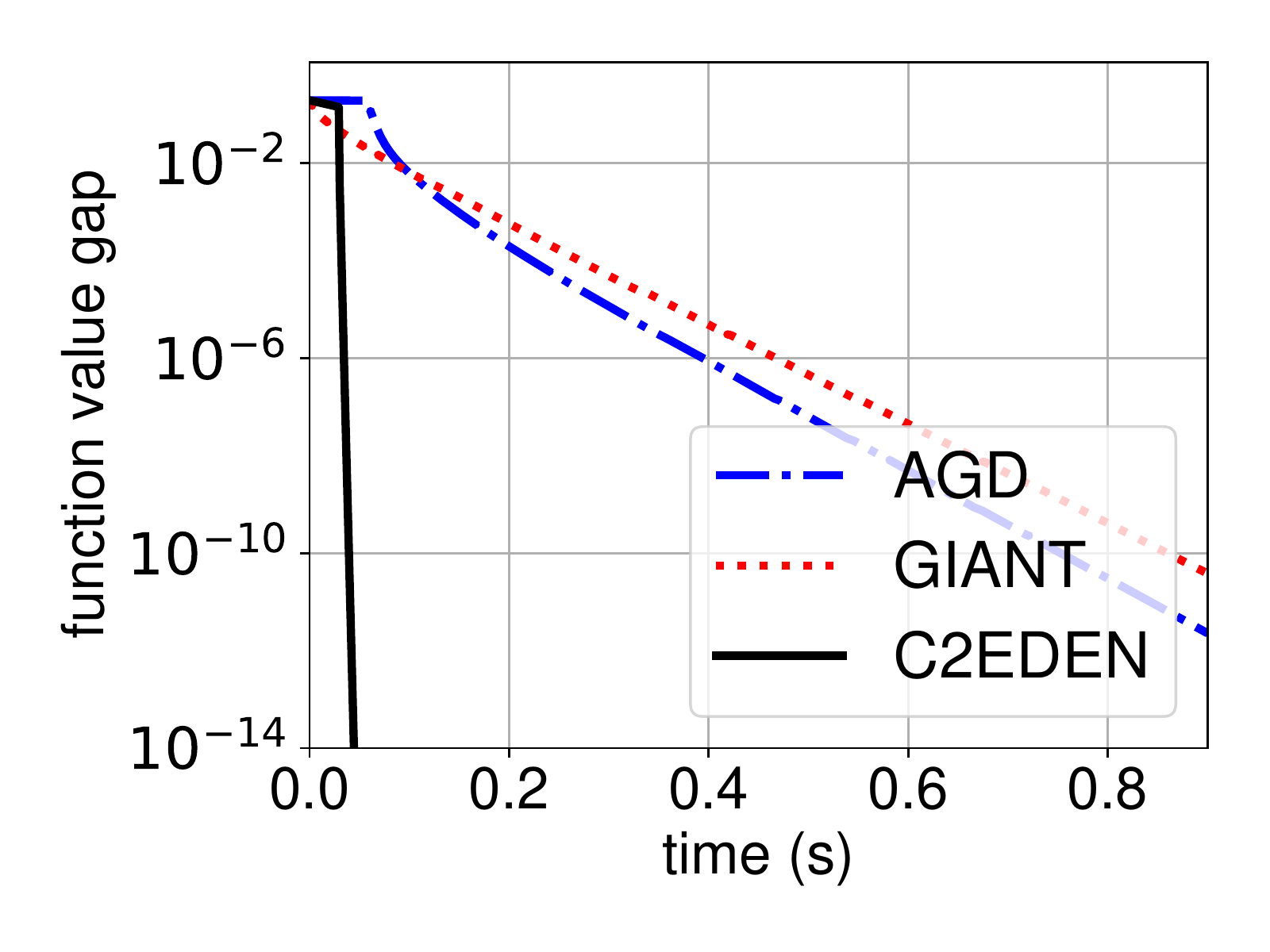} &
    \includegraphics[scale=0.22]{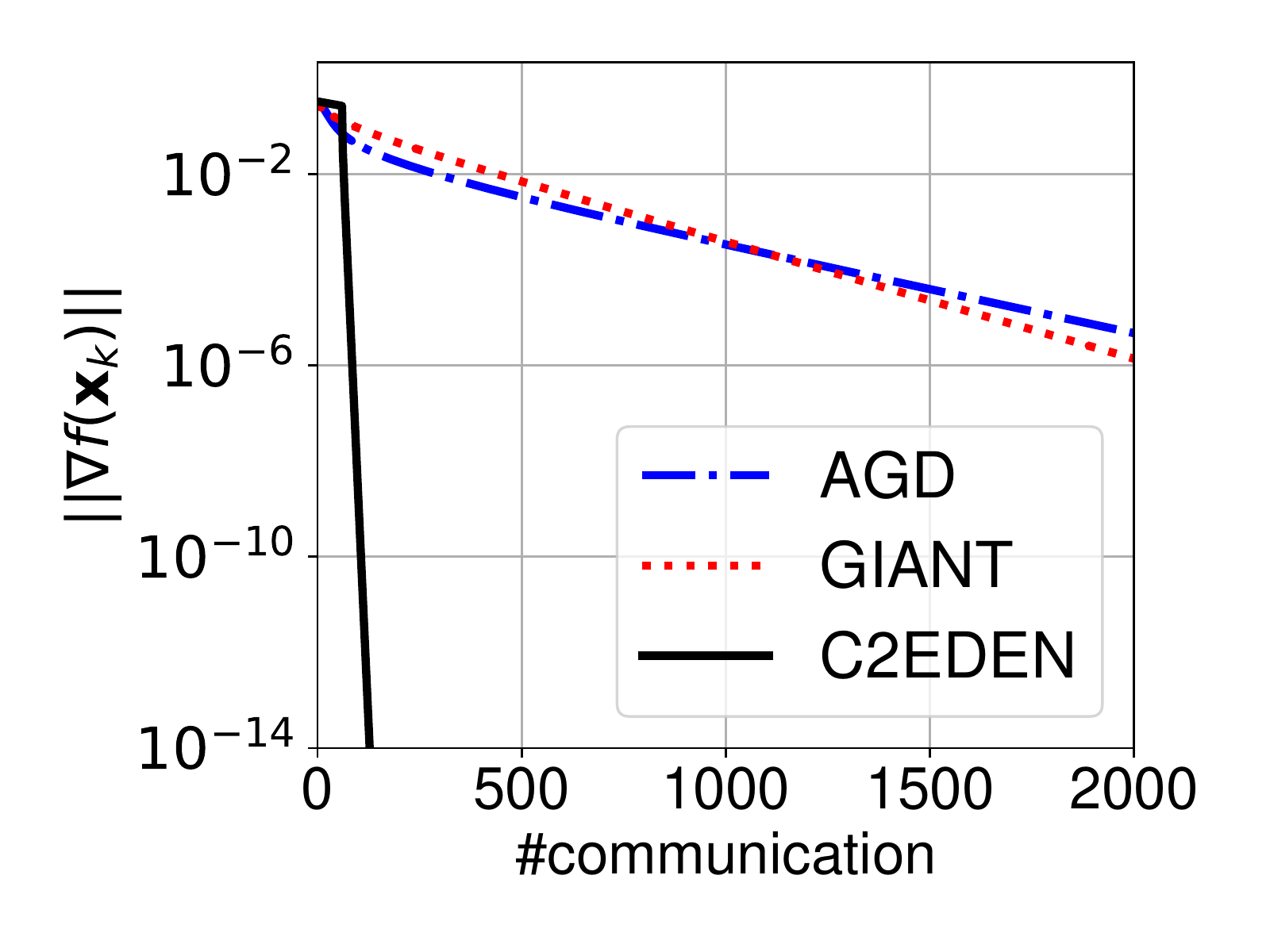} &
    \includegraphics[scale=0.22]{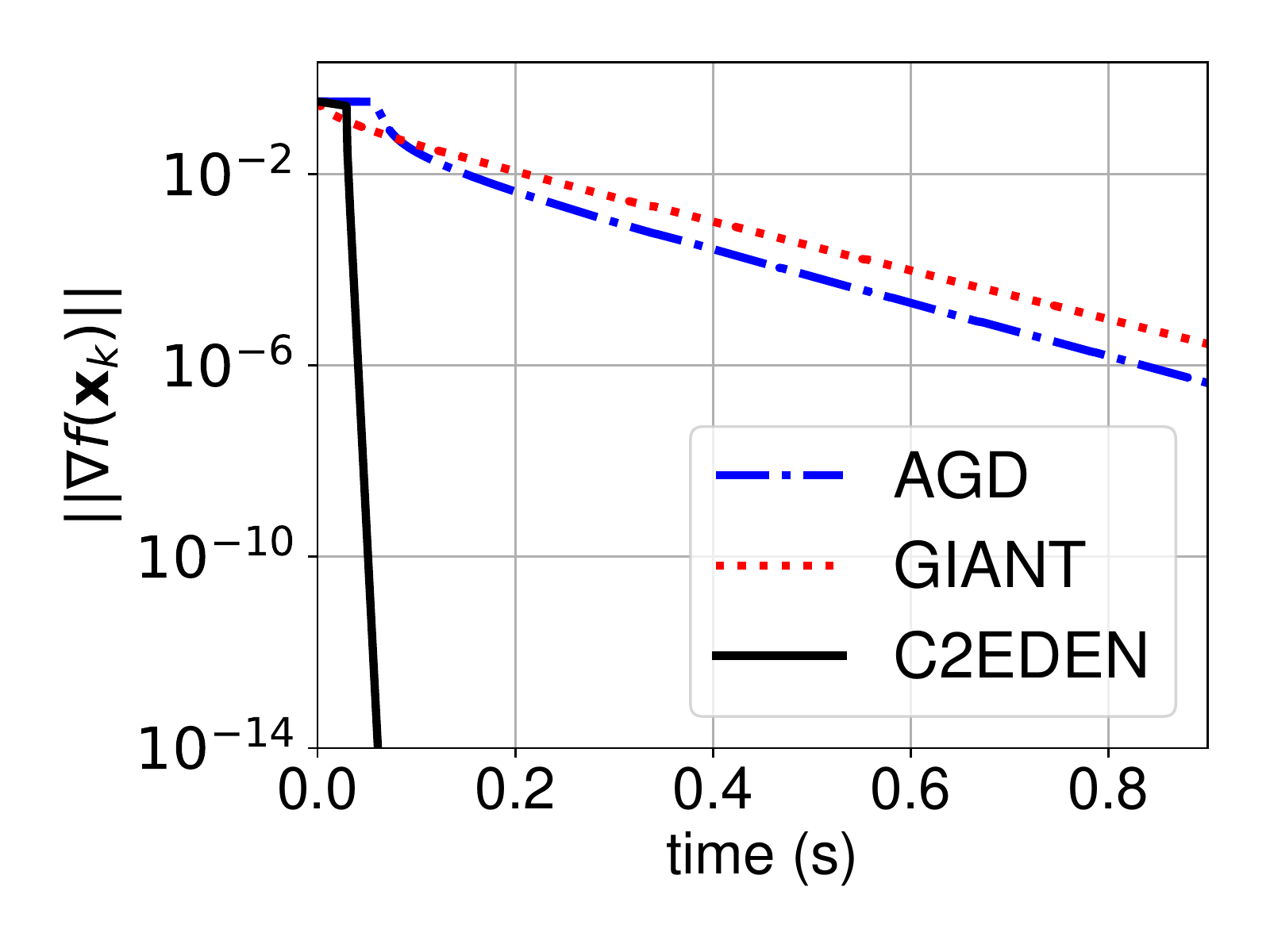}  \\
    (a)   $ \#$communication vs. gap     &
    (b)  time (s) vs. gap      &
    (c) $ \#$communication vs. $ \Vert \nabla f(x_k) \Vert$   &
    (d)   time (s) vs. $ \Vert \nabla f(x_k) \Vert$
    \end{tabular}
    \caption{The results of the model of $\ell_2$-regularized logistic regression on splice ($n$=16).} 
     \label{fig:sc-splice16}
\end{figure*}

In this section, we provide numerical experiments for the proposed Communication and Computation Efficient DistributEd Newton  (C2EDEN) method on regularized logistic regression.
The model is formulated as the form of \eqref{eq:obj} with
\begin{align}\label{eq:logit}
    f_i(\x)= \frac{1}{m_i}\sum_{j=1}^{m_i}\ln \left(1+\exp{(-b_{ij}\x^{\top}\va_{ij})}\right) +\lambda R(\x),
\end{align}
where $\va_{ij}\in\RB^{d}$ and $b_{ij}\in\{-1,+1\}$ are the feature and the corresponding label of the $j$-th sample on the $i$-th client, $m_i$ denotes the number of samples on the $i$-th client, $R(\x)$ is the regularization term.
We set the hyperparameter  $\lambda = 10^{-6}$ and the number of clients~$n=16$ and $n=32$.

Our experiments are conducted on
AMD EPYC 7H12 64-Core Processor with 512GB memory\footnote{The code for our experiments can be downloaded from \url{https://github.com/7CCLiu/C2EDEN} for reproducibility.}. 
We use MPI~3.3.2 and Python~3.9 to implement the algorithms and the operating system is Ubuntu~20.04.2. 
We compare our the proposed C2EDEN with the baseline algorithms that require $\fO(d)$ communication complexity per iteration and avoid $\fO(d^2)$ space complexity on the client.
All of the datasets we used are can be downloaded from LIBSVM repository~\cite{CC01a}.

\subsection{The Nonconvex Case}
We consider logistic regression with the following regularization $ R(\x)= \sum_{p=1}^d\frac{x_{(p)}^2}{1+x_{(p)}^2},$
% \begin{align*}
%   R(\x)= \sum_{p=1}^d\frac{x_{(p)}^2}{1+x_{(p)}^2},
% \end{align*}
where $x_{(p)}$ is the $p$-th coordinate of $\vx\in\BR^d$.
Such setting leads to the objective $f(\cdot)$ be nonconvex~\cite{antoniadis2011penalized}.

We compare the proposed C2EDEN with the distributed gradient descent (GD)~\cite{nesterov2018lectures} and the local cubic-regularized Newton (LCRN)~\cite{ghosh2021escaping} method. 
We evaluate all of the algorithms on datasets ``a9a'' ,``w8a'' and ``mushrooms''. 
For GD, we tune the step-size from $\{10^{-1},10^{-2},10^{-3} \}$. 
For C2EDEN and local cubic-Newton, we tune the regularization parameter $M$ from $\{1,10,100\}$.

We present the results for $n=32$ in Figure \ref{fig:nc-a9a}, \ref{fig:nc-w8a} and \ref{fig:nc-mushrooms} and the results for $n=16$ in Figure~\ref{fig:nc-a9a16}, \ref{fig:nc-w8a16} and \ref{fig:nc-mushrooms16} by showing the comparison on the number of iterations and computational time (seconds) against the function value gap $f(\vx_k)-\hat f$ and gradient norm, where $\hat f$ is the smallest function value appears in the iterations of all algorithms. 

All the figures show that both of the function value gap and the gradient norm of C2EDEN decrease much faster than the baselines. 

% We can also observe from the figures that C2EDEN enjoys the superlinear rate after it enters the local region (i.e. (c) of Figure~\ref{fig:nc-a9a} after $\|\nabla f(\x_k)\|\leq 1{\rm e}-7$) which cannot be achieved by other methods. All these results show that the proposed C2EDEN performs significantly better than baseline methods.

\subsection{The Strongly-Convex Case}

 Then we consider logistic regression with the $\ell_2$-regularization term $ R(\x) =  \frac{1}{2} \Vert \x \Vert^2$,
% \begin{align*}
%       R(\x) =  \frac{1}{2} \Vert \x \Vert^2,
% \end{align*}
which leads to the objective be strongly-convex.

We compare the proposed C2EDEN with the distributed accelerated gradient descent (AGD)~\cite{nesterov2018lectures} and the GIANT~\cite{wang2018giant}. 
We evaluate all of the algorithms on datasets ``a9a'', ``phishing'' and ``splice''. 
For AGD, we tune the step size form $\{10^{-1},10^{-2},10^{-3} \}$ and tune the momentum parameter from $\{0.9, 0.99,0.999\}$. 
We also introduce a warm-up phrase for GIANT to achieve its region of local convergence. 

We present the results for $n=32$ in Figure \ref{fig:sc-a9a}, \ref{fig:sc-phishing} and  \ref{fig:sc-splice} and the results for $n=16$  in Figure \ref{fig:sc-a9a16}, \ref{fig:sc-phishing16} and Figure \ref{fig:sc-splice16}
by showing the comparison on the number of communication rounds and computational time (seconds) against function value gap and gradient norm.

We can observe that C2EDEN performs the superlinear rate from early iterations (i.e. (c) of Figure~\ref{fig:sc-a9a} after $\nabla f(\x_k)\leq 1{\rm e}-2$) and it converges much faster than baselines.

\section{Conclusion}
\label{sec:conclu}
In this work, we propose the Communication and Computation EfficiEnt Distributed Newton (C2EDEN) method for distributed optimization.
We provide a new mechanism to communicate the second-order information along with a simple yet efficient update rule. 
We prove the proposed method possesses the fast convergence rate like the classical second-order methods on single machine optimization problems. 
The empirical studies on both convex and nonconvex optimization problem also support our theoretical analysis.

For the future work, it is very interesting to generalize our ideas to the setting of decentralized optimization. It is also possible to design the stochastic variants of C2EDEN.

% \section*{Acknowledgements}
% Luo Luo is supported by National Natural Science Foundation of China (No. 62206058) and Shanghai Sailing Program (22YF1402900).

\newpage
\balance
\bibliography{example_paper}

\begin{thebibliography}{38}
\providecommand{\natexlab}[1]{#1}
\providecommand{\url}[1]{\texttt{#1}}
\expandafter\ifx\csname urlstyle\endcsname\relax
  \providecommand{\doi}[1]{doi: #1}\else
  \providecommand{\doi}{doi: \begingroup \urlstyle{rm}\Url}\fi

\bibitem[Aji and Heafield(2017)]{aji2017sparse}
Alham~Fikri Aji and Kenneth Heafield.
\newblock Sparse communication for distributed gradient descent.
\newblock \emph{arXiv preprint arXiv:1704.05021}, 2017.

\bibitem[Antoniadis et~al.(2011)Antoniadis, Gijbels, and
  Nikolova]{antoniadis2011penalized}
Anestis Antoniadis, Ir{\`e}ne Gijbels, and Mila Nikolova.
\newblock Penalized likelihood regression for generalized linear models with
  non-quadratic penalties.
\newblock \emph{Annals of the Institute of Statistical Mathematics},
  63\penalty0 (3):\penalty0 585--615, 2011.

\bibitem[Boyd et~al.(2011)Boyd, Parikh, Chu, Peleato, Eckstein,
  et~al.]{boyd2011distributed}
Stephen Boyd, Neal Parikh, Eric Chu, Borja Peleato, Jonathan Eckstein, et~al.
\newblock Distributed optimization and statistical learning via the alternating
  direction method of multipliers.
\newblock \emph{Foundations and Trends{\textregistered} in Machine learning},
  3\penalty0 (1):\penalty0 1--122, 2011.

\bibitem[Boyd(2011)]{boyd2011convex}
Stephen~P. Boyd.
\newblock Convex optimization: from embedded real-time to large-scale
  distributed.
\newblock In \emph{KDD}, 2011.

\bibitem[Chang and Lin(2011)]{CC01a}
Chih-Chung Chang and Chih-Jen Lin.
\newblock {LIBSVM}: A library for support vector machines.
\newblock \emph{ACM Transactions on Intelligent Systems and Technology},
  2:\penalty0 27:1--27:27, 2011.
\newblock Software and datasets available at
  \url{http://www.csie.ntu.edu.tw/~cjlin/libsvm}.

\bibitem[Crane and Roosta(2019)]{crane2019dingo}
Rixon Crane and Fred Roosta.
\newblock {DINGO}: Distributed {Newton}-type method for gradient-norm
  optimization.
\newblock \emph{NeurIPS}, 2019.

\bibitem[Doikov et~al.(2022)Doikov, Chayti, and Jaggi]{doikov2022second}
Nikita Doikov, El~Mahdi Chayti, and Martin Jaggi.
\newblock Second-order optimization with lazy {Hessians}.
\newblock \emph{arXiv preprint arXiv:2212.00781}, 2022.

\bibitem[Ghosh et~al.(2021)Ghosh, Maity, Mazumdar, and
  Ramchandran]{ghosh2021escaping}
Avishek Ghosh, Raj~Kumar Maity, Arya Mazumdar, and Kannan Ramchandran.
\newblock Escaping saddle points in distributed {Newton}'s method with
  communication efficiency and {Byzantine} resilience.
\newblock \emph{arXiv preprint arXiv:2103.09424}, 2021.

\bibitem[Islamov et~al.(2021)Islamov, Qian, and
  Richt{\'a}rik]{islamov2021distributed}
Rustem Islamov, Xun Qian, and Peter Richt{\'a}rik.
\newblock Distributed second order methods with fast rates and compressed
  communication.
\newblock In \emph{ICML}, 2021.

\bibitem[Islamov et~al.(2022)Islamov, Qian, Hanzely, Safaryan, and
  Richt{\'a}rik]{islamov2022distributed}
Rustem Islamov, Xun Qian, Slavom{\'\i}r Hanzely, Mher Safaryan, and Peter
  Richt{\'a}rik.
\newblock Distributed {Newton}-type methods with communication compression and
  {Bernoulli} aggregation.
\newblock \emph{arXiv preprint arXiv:2206.03588}, 2022.

\bibitem[Karimireddy et~al.(2020)Karimireddy, Kale, Mohri, Reddi, Stich, and
  Suresh]{karimireddy2020scaffold}
Sai~Praneeth Karimireddy, Satyen Kale, Mehryar Mohri, Sashank Reddi, Sebastian
  Stich, and Ananda~Theertha Suresh.
\newblock {SCAFFOLD}: Stochastic controlled averaging for federated learning.
\newblock In \emph{ICML}, 2020.

\bibitem[Kone{\v{c}}n{\`y} et~al.(2016)Kone{\v{c}}n{\`y}, McMahan, Yu,
  Richt{\'a}rik, Suresh, and Bacon]{konevcny2016federated}
Jakub Kone{\v{c}}n{\`y}, H.~Brendan McMahan, Felix~X. Yu, Peter Richt{\'a}rik,
  Ananda~Theertha Suresh, and Dave Bacon.
\newblock Federated learning: Strategies for improving communication
  efficiency.
\newblock \emph{arXiv preprint arXiv:1610.05492}, 2016.

\bibitem[Lee et~al.(2018)Lee, Lim, and Wright]{lee2018distributed}
Ching-pei Lee, Cong~Han Lim, and Stephen~J. Wright.
\newblock A distributed quasi-newton algorithm for empirical risk minimization
  with nonsmooth regularization.
\newblock In \emph{SIGKDD}, 2018.

\bibitem[Li et~al.(2013)Li, Zhou, Yang, Li, Xia, Andersen, and
  Smola]{li2013parameter}
Mu~Li, Li~Zhou, Zichao Yang, Aaron Li, Fei Xia, David~G. Andersen, and
  Alexander Smola.
\newblock Parameter server for distributed machine learning.
\newblock In \emph{Big learning workshop on NIPS}, 2013.

\bibitem[Li et~al.(2014)Li, Andersen, Smola, and Yu]{li2014communication}
Mu~Li, David~G. Andersen, Alexander~J. Smola, and Kai Yu.
\newblock Communication efficient distributed machine learning with the
  parameter server.
\newblock \emph{Advances in Neural Information Processing Systems}, 27, 2014.

\bibitem[Li et~al.(2020)Li, Huang, Yang, Wang, and Zhang]{li2019convergence}
Xiang Li, Kaixuan Huang, Wenhao Yang, Shusen Wang, and Zhihua Zhang.
\newblock On the convergence of {FedAvg} on non-iid data.
\newblock In \emph{ICLR}, 2020.

\bibitem[Liu et~al.(2017)Liu, Pan, and Ho]{liu2017distributed}
Sulin Liu, Sinno~Jialin Pan, and Qirong Ho.
\newblock Distributed multi-task relationship learning.
\newblock In \emph{Proceedings of the 23rd ACM SIGKDD International Conference
  on Knowledge Discovery and Data Mining}, pages 937--946, 2017.

\bibitem[Mishchenko et~al.(2022)Mishchenko, Malinovsky, Stich, and
  Richt{\'a}rik]{mishchenko2022proxskip}
Konstantin Mishchenko, Grigory Malinovsky, Sebastian Stich, and Peter
  Richt{\'a}rik.
\newblock {ProxSkip}: Yes! local gradient steps provably lead to communication
  acceleration! finally!
\newblock In \emph{ICML}, 2022.

\bibitem[Mitra et~al.(2021)Mitra, Jaafar, Pappas, and Hassani]{mitra2021linear}
Aritra Mitra, Rayana Jaafar, George~J. Pappas, and Hamed Hassani.
\newblock Linear convergence in federated learning: Tackling client
  heterogeneity and sparse gradients.
\newblock \emph{NeurIPS}, 2021.

\bibitem[Molzahn et~al.(2017)Molzahn, D{\"o}rfler, Sandberg, Low, Chakrabarti,
  Baldick, and Lavaei]{molzahn2017survey}
Daniel~K. Molzahn, Florian D{\"o}rfler, Henrik Sandberg, Steven~H. Low,
  Sambuddha Chakrabarti, Ross Baldick, and Javad Lavaei.
\newblock A survey of distributed optimization and control algorithms for
  electric power systems.
\newblock \emph{IEEE Transactions on Smart Grid}, 8\penalty0 (6):\penalty0
  2941--2962, 2017.

\bibitem[Nesterov(2008)]{nesterov2008accelerating}
Yurii Nesterov.
\newblock Accelerating the cubic regularization of {N}ewton’s method on
  convex problems.
\newblock \emph{Mathematical Programming}, 112\penalty0 (1):\penalty0 159--181,
  2008.

\bibitem[Nesterov(2018)]{nesterov2018lectures}
Yurii Nesterov.
\newblock \emph{Lectures on convex optimization}, volume 137.
\newblock Springer, 2018.

\bibitem[Nesterov and Polyak(2006)]{nesterov2006cubic}
Yurii Nesterov and Boris~T. Polyak.
\newblock Cubic regularization of {N}ewton method and its global performance.
\newblock \emph{Mathematical Programming}, 108\penalty0 (1):\penalty0 177--205,
  2006.

\bibitem[Nocedal and Wright(1999)]{nocedal1999numerical}
Jorge Nocedal and Stephen~J. Wright.
\newblock \emph{Numerical optimization}.
\newblock Springer, 1999.

\bibitem[Reddi et~al.(2016)Reddi, Kone{\v{c}}n{\`y}, Richt{\'a}rik,
  P{\'o}cz{\'o}s, and Smola]{reddi2016aide}
Sashank~J. Reddi, Jakub Kone{\v{c}}n{\`y}, Peter Richt{\'a}rik, Barnab{\'a}s
  P{\'o}cz{\'o}s, and Alex Smola.
\newblock {AIDE}: Fast and communication efficient distributed optimization.
\newblock \emph{arXiv preprint arXiv:1608.06879}, 2016.

\bibitem[Safaryan et~al.(2022)Safaryan, Islamov, Qian, and
  Richt{\'a}rik]{safaryan2021fednl}
Mher Safaryan, Rustem Islamov, Xun Qian, and Peter Richt{\'a}rik.
\newblock {FedNL}: Making newton-type methods applicable to federated learning.
\newblock In \emph{ICML}, 2022.

\bibitem[Shamir et~al.(2014)Shamir, Srebro, and Zhang]{shamir2014communication}
Ohad Shamir, Nati Srebro, and Tong Zhang.
\newblock Communication-efficient distributed optimization using an approximate
  {Newton}-type method.
\newblock In \emph{ICML}, 2014.

\bibitem[Smith et~al.(2015)Smith, Forte, Jordan, and Jaggi]{smith2015l1}
Virginia Smith, Simone Forte, Michael~I. Jordan, and Martin Jaggi.
\newblock L1-regularized distributed optimization: A communication-efficient
  primal-dual framework.
\newblock \emph{arXiv preprint arXiv:1512.04011}, 2015.

\bibitem[Soori et~al.(2020)Soori, Mishchenko, Mokhtari, Dehnavi, and
  Gurbuzbalaban]{soori2020dave}
Saeed Soori, Konstantin Mishchenko, Aryan Mokhtari, Maryam~Mehri Dehnavi, and
  Mert Gurbuzbalaban.
\newblock {DAve-QN}: A distributed averaged quasi-{Newton} method with local
  superlinear convergence rate.
\newblock In \emph{AISTATS}, 2020.

\bibitem[Uribe and Jadbabaie(2020)]{uribe2020distributed}
C{\'e}sar~A. Uribe and Ali Jadbabaie.
\newblock A distributed cubic-regularized {N}ewton method for smooth convex
  optimization over networks.
\newblock \emph{arXiv preprint arXiv:2007.03562}, 2020.

\bibitem[Verbraeken et~al.(2020)Verbraeken, Wolting, Katzy, Kloppenburg,
  Verbelen, and Rellermeyer]{verbraeken2020survey}
Joost Verbraeken, Matthijs Wolting, Jonathan Katzy, Jeroen Kloppenburg, Tim
  Verbelen, and Jan~S Rellermeyer.
\newblock A survey on distributed machine learning.
\newblock \emph{Acm computing surveys (csur)}, 53\penalty0 (2):\penalty0 1--33,
  2020.

\bibitem[Wang et~al.(2018)Wang, Roosta, Xu, and Mahoney]{wang2018giant}
Shusen Wang, Fred Roosta, Peng Xu, and Michael~W. Mahoney.
\newblock {GIANT}: Globally improved approximate {N}ewton method for
  distributed optimization.
\newblock \emph{NeurIPS}, 2018.

\bibitem[Wangni et~al.(2018)Wangni, Wang, Liu, and Zhang]{wangni2018gradient}
Jianqiao Wangni, Jialei Wang, Ji~Liu, and Tong Zhang.
\newblock Gradient sparsification for communication-efficient distributed
  optimization.
\newblock \emph{Advances in Neural Information Processing Systems}, 31, 2018.

\bibitem[Yan et~al.(2014)Yan, Li, Xue, and Han]{yan2014coupled}
Ling Yan, Wu-Jun Li, Gui-Rong Xue, and Dingyi Han.
\newblock Coupled group lasso for web-scale ctr prediction in display
  advertising.
\newblock In \emph{ICML}, 2014.

\bibitem[Yang et~al.(2019)Yang, Yi, Wu, Yuan, Wu, Meng, Hong, Wang, Lin, and
  Johansson]{yang2019survey}
Tao Yang, Xinlei Yi, Junfeng Wu, Ye~Yuan, Di~Wu, Ziyang Meng, Yiguang Hong,
  Hong Wang, Zongli Lin, and Karl~H Johansson.
\newblock A survey of distributed optimization.
\newblock \emph{Annual Reviews in Control}, 47:\penalty0 278--305, 2019.

\bibitem[Ye et~al.(2022)Ye, He, and Chang]{ye2022accelerated}
Haishan Ye, Chaoyang He, and Xiangyu Chang.
\newblock Accelerated distributed approximate newton method.
\newblock \emph{IEEE Transactions on Neural Networks and Learning Systems},
  2022.

\bibitem[Yuan and Li(2020)]{yuan2020convergence}
Xiao-Tong Yuan and Ping Li.
\newblock On convergence of distributed approximate {Newton} methods:
  Globalization, sharper bounds and beyond.
\newblock \emph{Journal of Machine Learning Research}, 21\penalty0
  (206):\penalty0 1--51, 2020.

\bibitem[Zhang and Lin(2015)]{zhang2015disco}
Yuchen Zhang and Xiao Lin.
\newblock {DiSCO}: Distributed optimization for self-concordant empirical loss.
\newblock In \emph{ICML}, 2015.

\end{thebibliography}
\balance
\bibliographystyle{plainnat}

%%%%%%%%%%%%%%%%%%%%%%%%%%%%%%%%%%%%%%%%%%%%%%%%%%%%%%%%%%%%%%%%%%%%%%%%%%%%%%%
%%%%%%%%%%%%%%%%%%%%%%%%%%%%%%%%%%%%%%%%%%%%%%%%%%%%%%%%%%%%%%%%%%%%%%%%%%%%%%%
%\APPENDIX

%%%%%%%%%%%%%%%%%%%%%%%%%%%%
%%%%%%%%%%%%%%%%%%%%%%%%%%%%%%%%
%%%%%%%%%%%%%%%%%%%%%%%%%%%%%%%

%%%%%%%%%%%%%%%%%
%%%%%%%%%%%%%%%%%%%%%%%%%%%%%%
%%%%%%%%%%%%%%%%%%%

%%%%%%%%%%%%%%%%%%%%%%%%%%%%%%%%%%%%%%%%%%%%%%%%%%%%%%%%%%%%%%%%%%%%%%%%%%%%%%%
%%%%%%%%%%%%%%%%%%%%%%%%%%%%%%%%%%%%%%%%%%%%%%%%%%%%%%%%%%%%%%%%%%%%%%%%%%%%%%%

\end{document}